\documentclass[11pt]{article}
\usepackage{latexsym}
\usepackage{amsmath, amsthm, amssymb}
\usepackage{amscd}
\usepackage[english]{babel}
\usepackage{graphicx}
\usepackage[arc,all]{xy}
\usepackage{tikz-cd}
\usepackage{braids}
\usepackage[colorlinks]{hyperref}
\usepackage{pict2e}
\usepackage{ifpdf}
\usepackage{xcolor}
\usepackage{float}
\usepackage{wrapfig}
\hypersetup{
    colorlinks=true,
    linkcolor=blue,
    filecolor=magenta,      
    urlcolor=cyan,
    pdfpagemode=FullScreen,
    }

\newcommand{\colim}{\operatorname{colim}}   

\newtheorem{thm}{Theorem}[section]
\newtheorem{prop}{Proposition}[section]
\newtheorem{cor}[thm]{Corollary}
\newtheorem{lem}[thm]{Lemma}

\theoremstyle{definition} 
\newtheorem{defn}{Definition}[section]

\theoremstyle{remark} 
\newtheorem{rem}{Remark}

\theoremstyle{remark} 
\newtheorem{exmp}{Example}

\begin{document}

\title{\textbf{\large{Bar-Natan theory and tunneling between incompressible surfaces in $3$-manifolds}}}

\maketitle

\makeatletter 
\makeatother

\centerline{\author{Uwe Kaiser, \textit{Boise State University}, ukaiser@boisestate.edu}}

\begin{abstract}
In \cite{K1} the author defined for each (commutative) Frobenius algebra a skein module of surfaces in a $3$-manifold $M$ bounding a closed $1$-manifold $\alpha \subset \partial M$. The surface components are colored by elements of the Frobenius algebra. The modules are called the Bar-Natan modules of $(M,\alpha )$. In this article we show that Bar-Natan modules are colimit modules of functors associated to Frobenius algebras, \textit{decoupling} topology from algebra. The functors are defined on a category of $3$-dimensional compression bordisms embedded in cylinders over $M$ and take values in a linear category defined from the Frobenius algebra. The relation with the $1+1$-dimensional topological quantum field theory functor associated to the Frobenius algebra is studied. We show that the geometric content of the skein modules is contained in a \textit{tunneling graph} of $(M,\alpha )$, providing a natural presentation of the Bar-Natan module by application of the functor defined from the algebra. Such presentations have essentially been stated in \cite{K1} and \cite{AF} using ad-hoc arguments. But they appear naturally on the background of the Bar-Natan functor and associated categorical considerations. We discuss in general presentations of colimit modules for functors into module categories in terms of minimal terminal sets of objects of the category in the categorical setting. We also introduce a $2$-category version of the Bar-Natan functor, thereby in some way \textit{categorify} Bar-Natan modules of $(M,\alpha )$. 
\end{abstract}

\newpage

\tableofcontents

\section{Definitions and statement of some results}\label{sec=defs}

A Frobenius algebra $V$ over a commutative ring $R$ is a unital commutative algebra with a cocommutative comultiplication $\Delta : V\rightarrow V\otimes V$ and counit $\epsilon : V\rightarrow R$. In section \ref{sec=algebra} we will show how maps $K\rightarrow J$ of finite sets naturally define $V^{\otimes K}$-modules $V^{\otimes J}$. Throughout we use the categorical definition of tensor products as defined in \cite{DM}, Notes on Supersymmetry 1.2. A prototype example of a Frobenius algebra is $V=\mathbb{Z}[x]/(x^2)$, see \cite{Kh3} and example \ref{example=Khovanov} below. 

For each topological space $X$ we let $\underline{X}$ denotes the set of components.
Throughout we work in the smooth manifold category. Let $M$ be a compact connected orientable $3$-manifold and $\alpha \subset \partial M$ be a closed $1$-manifold. Each properly embedded surface $S\subset M$ bounding $\alpha $ induces a map $\underline{\alpha }\rightarrow \underline{S}$. Here a \textit{surface} is an orientable (if not stated otherwise) $2$-manifold, not necessarily connected.

\begin{defn}[Bar-Natan \cite{BN}, Asaeda-Frohman \cite{AF}, K. \cite{K1}] 
The \textit{Bar-Natan module} $\mathcal{W}(M,\alpha ;V)$ is the free $R$-module generated by isotopy classes relative to the boundary of surfaces properly embedded in $M$ bounding $\alpha $, with the components of the surfaces colored by elements of $V$, quotiented by the submodule generated by \textit{Bar-Natan relations}. In the figures below, the component colored $y\in V$ represents possibly additional closed or bounded components (could also be present additionally in relation 3.). Our figures are local as usual in skein theory, showing relations inside a $3$-ball in $M$ with the exterior of the $3$-ball and surfaces contained in the exterior fixed. Here are the \textit{Bar-Natan relations}:
\begin{enumerate} 
\item multilinearity in the colors of the components ($r_i\in R, x_i\in V$ for all $i$):
\vskip 0.2in
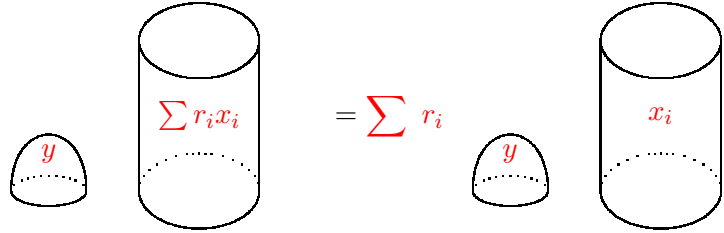
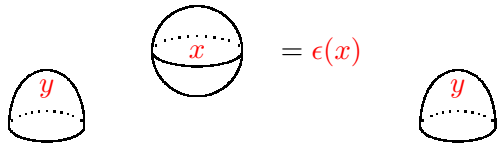
\[ \xy 
(-20,-5)*{\textcolor{red}{y}};
(-10,-5)*\ellipse(5,2)__,=:a(-180){-};
(-10,-5)*\ellipse(5,2){.};
(0,-5)*\ellipse(8,5){.}; 
(0,-5)*\ellipse(8,5)__,=:a(-180){-}; 
(0,5)*\ellipse(8,5){-}; 
 (-25,-10)*{}="TL"; 
(-15,-10)*{}="TR"; 
"TL";"TR" **\crv{(-25,0) & (-15,0) }; 
(0,0)*{\textcolor{red}{\sum r_ix_i}};
(-8,10)*{}="TL"; 
(8,10)*{}="TR"; 
(-8,-10)*{}="BL"; 
(8,-10)*{}="BR"; 
"TL"; "BL" **\dir{-}; 
"TR"; "BR" **\dir{-}; 
\endxy \
\qquad = \textcolor{red}{\sum \ r_i}\quad
\xy
(-20,-5)*{\textcolor{red}{y}};
(0,0)*{\textcolor{red}{x_i}};
(-10,-5)*\ellipse(5,2){.}; 
(-10,-5)*\ellipse(5,2)__,=:a(-180){-};
(0,-5)*\ellipse(8,5){.}; 
(0,-5)*\ellipse(8,5)__,=:a(-180){-}; 
(0,5)*\ellipse(8,5){-};  
(-25,-10)*{}="TL"; 
(-15,-10)*{}="TR"; 
"TL";"TR" **\crv{(-25,0) & (-15,0) }; 
(-8,10)*{}="TL"; 
(8,10)*{}="TR"; 
(-8,-10)*{}="BL"; 
(8,-10)*{}="BR"; 
"TL"; "BL" **\dir{-}; 
"TR"; "BR" **\dir{-}; 
\endxy \
\]
\vskip 0.2in
\item sphere relations ($x\in V$):
\vskip 0.2in
\[ \xy 
(-20,-5)*{\textcolor{red}{y}};
(-10,-5)*\ellipse(5,2){.};
(-10,-5)*\ellipse(5,2)__,=:a(-180){-};
(0,0)*\ellipse(6,6){-}; 
(0,0)*\ellipse(6,2)__,=:a(-180){-}; 
(0,0)*\ellipse(6,2){.}; 
 (-25,-10)*{}="TL"; 
(-15,-10)*{}="TR"; 
"TL";"TR" **\crv{(-25,0) & (-15,0) }; 
(0,0)*{\textcolor{red}{x}};
\endxy \
\qquad = \textcolor{red}{\epsilon (x)} \qquad
\xy
(-20,-5)*{\textcolor{red}{y}};
(-10,-5)*\ellipse(5,2)__,=:a(-180){-};
(-10,-5)*\ellipse(5,2){.};
 (-25,-10)*{}="TL"; 
(-15,-10)*{}="TR"; 
"TL";"TR" **\crv{(-25,0) & (-15,0) }; 

\endxy \
\]
\vskip 0.2in
\item neck-cutting relations ($\Delta (x)=x'\otimes x''$ (Sweedler's notation), $x,x',x''\in V$):
\vskip 0.2in
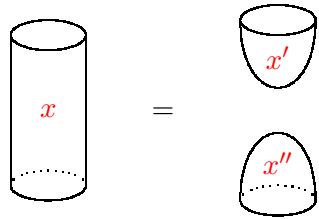
\[ \xy 
(0,0)*{\textcolor{red}{x}};
(0,5)*\ellipse(5,2){-}; 
(0,-5)*\ellipse(5,2){.}; 
(0,-5)*\ellipse(5,2)__,=:a(-180){-};
(-5,10)*{}="TL"; 
(5,10)*{}="TR"; 
(-5,-10)*{}="BL"; 
(5,-10)*{}="BR"; 
"TL"; "BL" **\dir{-}; 
"TR"; "BR" **\dir{-}; 
\endxy
\qquad = \qquad
\xy 
(0,-6)*\ellipse(5,2){.}; 
(0,-6)*\ellipse(5,2)__,=:a(-180){-}; 
(-5,-12)*{}="TL"; 
(5,-12)*{}="TR"; 
"TL";"TR" **\crv{(-5,0) & (5,0) }; 
(0,6)*\ellipse(5,2){-}; 
(-5,12)*{}="TL"; 
(5,12)*{}="TR"; 
"TL";"TR" **\crv{(-5,0) & (5,0)};
(0,7)*{\textcolor{red}{x'}}; 
(0,-7)*{\textcolor{red}{x''}};
\endxy \
 \] 
 \vskip .2in
\textit{Non-separating} neck cuttings are described by the \textit{handle operator} $\mathfrak{k}: V\rightarrow V$, which is defined by changing the color of the resulting component by multiplication in $V$ by $m \Delta (1)\in V$ where $m$ is the multiplication $V\otimes V\rightarrow V$. 
\end{enumerate}
\end{defn}

It follows that if a surface $S$ bounding $\alpha $ in $M$ contains a component colored $0$ then the element is trivial in the skein module. Components without a color will be assumed carrying the color $1\in V$. We sometimes simplify notation $\mathcal{W}(M,\alpha ;V)=:\mathcal{W}(M,\alpha )$ if $V$ is fixed.

\begin{prop} The Bar-Natan modules are naturally $V^{\otimes \underline{\alpha }}$-modules with module structure defined by multiplying Frobenius algebra elements into boundary components.
\end{prop}

\begin{exmp} 
\noindent (a) The figure shows a surface with a single component colored by $x\in V$ and the multiplication by $x_1\otimes \cdots \otimes x_5\in V^{\otimes 5}$ in $V$: ($x_1\otimes \cdots \otimes x_5)\cdot x=x_1\cdots x_5\cdot x$.
\vskip .2in
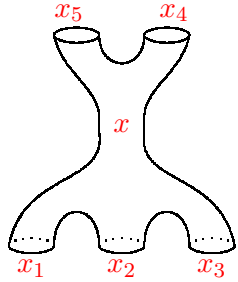
\[ \xy 
(0,0)*{\textcolor{red}{x}};
(-12,-19)*{\textcolor{red}{x_1}};
(0,-19)*{\textcolor{red}{x_2}};
(12,-19)*{\textcolor{red}{x_3}};
(7,15)*{\textcolor{red}{x_4}};
(-7,15)*{\textcolor{red}{x_5}};
(-6,-8)*\ellipse(3,1){.}; 
(6,-8)*\ellipse(3,1){.}; 
(0,-8)*\ellipse(3,1){.}; 
(-6,-8)*\ellipse(3,1)__,=:a(-180){-}; 
(6,-8)*\ellipse(3,1)__,=:a(-180){-}; 
(0,-8)*\ellipse(3,1)__,=:a(180){-}; 
(-3,6)*\ellipse(3,1){-}; 
(3,6)*\ellipse(3,1){-}; 
(-3,12)*{}="1"; 
(3,12)*{}="2"; 
(-9,12)*{}="A2"; 
(9,12)*{}="B2"; 
"1";"2" **\crv{(-3,7) & (3,7)}; 
(-3,0)*{}="A"; 
(3,0)*{}="B"; 
(-3,1)*{}="A1"; 
(3,1)*{}="B1"; 
"A";"A1" **\dir{-}; 
"B";"B1" **\dir{-}; 
"B2";"B1" **\crv{(8,7) & (3,5)}; 
"A2";"A1" **\crv{(-8,7) & (-3,5)}; 
%REFLECT 
(3,-16)*{}="1"; 
(9,-16)*{}="2"; 
"1";"2" **\crv{(3,-10) & (9,-10)}; 
(-3,-16)*{}="1"; 
(-9,-16)*{}="2"; 
"1";"2" **\crv{(-3,-10) & (-9,-10)}; 
(-15,-16)*{}="A2"; 
(15,-16)*{}="B2"; 
(-3,0)*{}="A"; 
(3,0)*{}="B"; 
(-3,-1)*{}="A1"; 
(3,-1)*{}="B1"; 
"A";"A1" **\dir{-}; 
"B";"B1" **\dir{-}; 
"B2";"B1" **\crv{(13,-6) & (2,-8)}; 
"A2";"A1" **\crv{(-13,-6) & (-2,-8)}; 
\endxy \] 
\vskip 0.1in
\noindent (b) Let $B$ the compact $3$-ball. Then 
$$\mathcal{W}(B^3,\alpha ;V)\cong V^{\underline{\alpha }},$$ 
which is the $1$-dimensional $V^{\underline{\alpha }}$-module defined by $\textrm{Id}: \underline{\alpha }\rightarrow \underline{\alpha }$.
\vskip 0.1in
\noindent (c)
If $\Sigma $ is a closed orientable surface of genus $g$ then $\mathcal{W}(\Sigma \times [0,1];\alpha \times \{0\}\cup \beta \times \{1\})$ is a bimodule over algebras $\mathcal{W}(\Sigma \times [0,1];\alpha \times \{0,1\})$ and 
$\mathcal{W}(\Sigma \times [0,1];\beta \times \{0,1\})$.
\end{exmp}

In order to state our first main result we recall a definition from category theory \cite{M}, IX. For further definitions and background in category theory see \cite{M} and for higher categories \cite{Bo}.

\begin{defn}
Let $\mathcal{C}$ be a small category with $\mathcal{C}^0$ the set of objects and $\mathcal{C}^1$ the set of morphisms (arrows). Let
$F: \mathcal{C}\rightarrow \textrm{$S$-Mod}$ be a functor, $S$ a unital commutative ring and \textrm{$S$-Mod} the category of $S$-modules. Then 
$$\colim(F):=\bigoplus_{x\in \mathcal{C}^0}F(x)/Rel,$$
where $Rel$ is the $S$-submodule generated by all elements of the form $v-F(m)v$ for all morphisms $(m: x\rightarrow y)\in \mathcal{C}^1$
and $v\in F(x), F(m)v\in F(y)$ defined in the direct sum using the inclusions of corresponding summands.
\end{defn}

\begin{thm}\label{thm=premain}
\noindent \textbf{(i)} For each $(M,\alpha )$ there exists the \textit{compression bordism category}
$\mathcal{C}(M,\alpha )$ with objects: embeddings of surfaces in $M$ bounding $\alpha $ (representing isotopy classes), and morphisms: compression isotopy classes of compression bordisms in $M\times [0,1]$ (with corners,
products over $\partial M$).

\noindent \textbf{(ii)} For each Frobenius algebra $V$ and finite set $J$ there exists a category $\mathcal{V}(J)$ of $V^{\otimes J}$-modules (with morphisms built from $\mathfrak{k}$, $\Delta$, $\epsilon $).

\noindent \textbf{(iii)} There exists the \textit{Bar-Natan functor}
$$F=F_{(M,\alpha )}: \mathcal{C}(M,\alpha )\rightarrow \mathcal{V}(\underline{\alpha })$$
such that if $\hat{F}$ is defined by $F$ composed with the inclusion into a natural completion of $\mathcal{V}(\underline{\alpha })$ in $V^{\underline{\alpha }}-\textrm{Mod}$ then
$$\colim(\hat{F})\cong \mathcal{W}(M,\alpha ;V).$$
\end{thm}

The compression category will be defined in section \ref{sec=compression}, the category $\mathcal{V}(\underline{\alpha })$ in section \ref{sec=algebra}, and the functor and its completion in section \ref{sec=functor}. 

We will show that for the results of theorem \ref{thm=premain} there exists
\begin{itemize}
\item a \textit{bicategory} version, see section \ref{sec=bicategory}, 
\item \textit{unorientable} and \textit{oriented} versions (see section \ref{sec=functor}). If not indicated otherwise we work in the orientable setting. But the theory we develop carries over to both an unorientable and oriented setting (The oriented setting will 
begin with pairs $(M,\alpha )$ for $M$ an oriented $3$-manifold and $\alpha $ an oriented $1$-manifold embedded in $\partial M$),
\item an \textit{abstract}, i.\ e.\ non-embedded versions for each non-negative integer $k$, see section \ref{sec=functor}, only
depending on the number of components of bounding curves,
\item natural morphisms and functors relating these structures. See \cite{Fa} for even further variations of Bar-Natan modules. 
\end{itemize}

\vskip 0.1in

\textit{Abstract Bar-Natan} theory has been introduced by Bar-Natan in his geometric construction of Khovanov homology in \cite{BN}. The 
embedded Bar-Natan theory as discussed here maybe contributes to the understanding of extended $1-2-3$ TQFT (topological quantum field theory) and the corresponding TQFT with embeddings in dimensions $2-3-4$.

In the second part of the article we show that \textit{embedded} Bar-Natan theory captures the tunneling between incompressible surfaces in $3$-manifolds. 

Let $\mathcal{C}$ be a small category (i.\ e.\ the class of objects is a set)

\begin{defn} A subset $\mathcal{T}\subset \mathcal{C}^0$ is \textit{terminal} if for each $x\in \mathcal{C}^0$ there exists a $t\in \mathcal{T}$ and a morphism $x\rightarrow t$ in $\mathcal{C}^1$. A terminal
set $\mathcal{T}$ is 
\textit{minimal} if $\mathcal{T}\setminus \{t\}$ is not terminal for each $t\in \mathcal{T}$. 
\end{defn}

On the topological side we will need the following definitions using classical $3$-manifold concepts:  

\begin{defn} 
\noindent \textbf{(i)} A proper surface $S$ in $M$ bounding $\alpha $ is called \textit{incompressible} if (1) each $2$-sphere component bounds a $3$-ball, and (2) each loop on $S$, which bounds a disk in $M$ already bounds a disk on $S$, see \cite{He} but note that our surfaces can be disconnected.

\noindent \textbf{(ii)} $(M,\alpha )$ is \textit{Haken-reducible} if there exists an incompressible surface $F$ in $M$ bounding $\alpha $, which contains a $2$-sphere component $S$ that does not bound a $3$-ball and is not parallel to any component of $F\setminus S$, and $F\setminus S\neq \emptyset$.
Otherwise $(M,\alpha )$ is called \textit{Haken-irreducible}. 
\end{defn}

If $M$ is irreducible then $(M,\alpha )$ is Haken-irreducible for each $\alpha $.

\begin{exmp}
$(S^2\times S^1,\emptyset )$ and $(S^2\times [0,1],\emptyset )$ are Haken-irreducible but not irreducible. 
\end{exmp}

\begin{prop} If $(M,\alpha )$ is Haken-irreducible then the set of isotopy classes of incompressible surfaces in $M$ bounding $\alpha $ is minimal.
\end{prop}

Let $\mathcal{T}\subset \mathcal{C}^0(M,\alpha )$ be a minimal set
of isotopy classes of incompressible surfaces with some specific ordering $\preceq $, see \ref{sec=presentations}. (We ignore the possible \textit{loop morphism} for the corresponding incompressible surfaces which map to \textit{permutation isomorphisms} by the functor.) 

Let 
$$S\leftarrow S''\rightarrow S'$$
be a diagram in $\mathcal{C}(M,\alpha )$ with $S,S'\in \mathcal{T}$.
The traces of the morphisms in $M\times I$ determine a $3$-manifold $W$ together with a description of its handle structure by complete \textit{compressions} from $S''$ to both $S$ and $S'$. The decomposition of $W$ into components will decompose $S$ and $S'$ into corresponding \textit{parts}.

Then the relation for $\mathcal{W}(M,\alpha )$ corresponding to $S\leftarrow S''\rightarrow S'$ is determined by 
$$F(S)\leftarrow F(S'')\rightarrow F(S')$$
This will be read off the genera of certain \textit{parts} of $S,S'$ and the number of additional \textit{stabilization} handles (additional
$1$-handles for $S''$ necessary for embedding), which correspond to certain non-separating neck-cuttings, 
given by the action of the handle operator $\mathfrak{k}$ under $F$. 

A \textit{decomposition into parts}, and 
corresponding \textit{stabilization numbers} are collected into a vector $(n_1,\ldots ,n_r)$ called the \textit{tunneling invariant} for $S\preceq S'$. We will see that the tunneling invariant completely determines 
$$F(S)\leftarrow F(S'')\rightarrow F(S')$$
and thus a relation for $\mathcal{W}(M,\alpha )$.  
\textit{Reducible} tunneling invariants are tunneling invariants that can be omitted because the corresponding relations are already generated by other tunneling invariants in a \textit{prescribed way}.

Suppose that $(M,\alpha )$ be Haken-irreducible, and let $\mathcal{T}$ be ordered. 

\begin{defn} A \textit{tunneling graph} of $(M,\alpha )$ (with respect to $\preceq$) is a labeled oriented graph. The vertices correspond to isotopy classes of incompressible surfaces in $(M,\alpha )$ labeled by the genera of the components. The edges from $S$ to $S'$ correspond to \textit{irreducible} tunneling invariants from $S$ to $S'\preceq S$ as defined in section \ref{sec=tunneling}.
\end{defn}

\begin{thm}\label{thm=tunneling} The tunneling graph of $(M,\alpha )$ determines a natural presentation of the Bar-Natan module $\mathcal{W}(M,\alpha )$. 
\end{thm}

It seems to be a difficult problem to explicitly describe tunneling graphs and their properties.

\vskip 0.1in

Here is the outline of the article:
In section \ref{sec=compression} we study the category of compression bodies of $3$-manifolds, which is the category on which the Bar-Natan functor is defined. In section \ref{sec=algebra} we define linear categories from Frobenius algebras. We prove a gluing result on the algebraic level, which will be necessary in the proof of the bicategory version \ref{thm=bifunctor} of \ref{thm=main}. In section \ref{sec=functor} we define the Bar-Natan functor. We also discuss the relations with abstract, oriented and unorientable versions. In section \ref{sec=bicategory} we prove the bicategory version of Theorem \ref{thm=premain}. In section \ref{sec=presentations} we
prove a general result about the presentation of colimit modules defined from functors into colimit modules, with generators given by terminal elements and relations defined by a pullback principle, see in particular section \ref{sec=pullback} where the pullback principle is applied to the Bar-Natan functor. In section \ref{sec=to Bar-Natan} we prove theorem \ref{thm=tunneling}. Section \ref{sec=local} uses the categorical ideas to discuss modules defined by restricting the functor. In section \ref{sec=homology} we discuss the special case of the Frobenius manifold defined from the homology of an oriented manifold. 

Categories and skein modules of foams (certain $2$-complexes colored by elements of Frobenius algebras, bounding webs, which are certain graphs in the boundary) have recently been considered in order to categorify quantum invariants of links in cylinders over surfaces, see e.\ g.\ \cite{QW}. See also the $\textrm{sl}(n)$-theory in \cite{MSV}. It could be interesting to extend these theories along the lines of our article. 

A referee for an early version of this article noted in 2013 without giving details that our results give a \textit{categorification} of Bar-Natan modules. Since the statements do not fall into the usual setting of categorification I just have been guessing what the referee might have thought. Theorem \ref{thm=premain} above provides a \textit{presentation} as a colimit of a functor, separating topology and algebra. But the appearance of categories by itself cannot be considered a categorification. It is the bicategory version, categorically lifting this presentation, which provides a type of categorification. It contains the original module in the form of the functor on a category but embeds it into a higher category. Thus we consider Theorem \ref{thm=bifunctor} as the categorification result. The results of this article are about 10 years old but at the time there did not seem to be of much interest in our results. The recent article by Hogencamp, Rose and Wedrich \cite{HRW} is related to Bar-Natan theory, so maybe our results here give some insights.

\vskip 0.1in

I would like to than Charlie Frohman for his continuing interest in my work on Bar-Natan modules, and Ralph Mueger for some helpful communication concerning the categorical part of the article.

\section{The category of compression bordisms}\label{sec=compression}

We identify a manifold $Y$ with $Y\times \{y\}\subset Y\times [a,b]$ for real numbers $a<b$. We identify $Y\times [a,b]$ with $Y\times [c,d]$ for any real numbers $c<d$ using the natural affine isomorphism 
$[a,b]\cong [c,d]$. We remind that if not stated differently all manifolds will be assumed orientable but we do not fix orientations. In particular bordisms will be orientable manifolds bounding orientable manifolds but we do not fix orientations.

\vskip .05in

For fixed $(M,\alpha )$ given as in section \ref{sec=defs} we define a category $\mathcal{M}_0=\mathcal{M}_0(M,\alpha )$ as follows: The objects are  
properly embedded surfaces $S$ in $M$ with boundary $\alpha $.
The morphisms from $S$ to $S'$ are isotopy classes relative to the boundary of $3$-dimensional manifolds 
(with corners if $\alpha$
is not empty) $W\subset M\times [-1,1]$ properly embedded with $W\cap (M\times \{-1\})=S\times \{-1\}$ and $W\cap (M\times \{1\})=S'\times \{1\}$, which are products along $\partial M\times [-1,1]$. 
We also assume that the bordisms are products near $M\times  \{-1,1\}$. The composition by glueing is well-defined using the natural identifications. But we throughout will restrict to glueings preserving orientability. The identity morphisms are cylinders for $S\neq \emptyset$, and the identity morphism $\emptyset \rightarrow \emptyset$   is the empty $3$-manifold $W=\emptyset$. 

Note that $(W,W \cap (M \times \{-1\}),W\cap (M \times \{1\}))$ is a $3$-manifold triad in the sense of Casson and Gordon \cite{CG}, except that in this article we usually work usually in an orientable only setting, sometimes even allowing unorientable surfaces and $3$-manifolds. In the orientable setting our $3$-manifolds will be assumed orientable. 
The morphisms in $\mathcal{M}_0$ are \textit{embedded} manifold triads. We recall some definitions from \cite{CG}.
A \textit{compression body} is a $3$-manifold triad $(W,\partial_-W,\partial_+W)$, where $W$ is constructed from $\partial_-W\times I$ by
attaching $2$-handles and $3$-handles such that $\partial_+W$ has no $2$-sphere components. If we drop the last assumption we will call
the triad a \textit{compression bordism}.   
Recall that a \textit{Heegaard splitting} of a manifold triad $(M,S,S')$ is a pair of compression bodies 
$W,W'$ such that $S\times \{1\}=\partial_+W$, $S'\times \{1\}=\partial_+W'$ and $M$ results by glueing $W$ and $W'$ along $\partial_-W\cong \partial_-W'$.
If $W,W'$ are compression bordisms (i.\ e.\ two-sphere components are allowed in $\partial_+W$ or $\partial_+W'$) then we still call this a \textit{Heegaard splitting}.
Note that each $3$-manifold triad $(W,S,S')$ has a Heegard splitting in this sense. 

We define the category $\mathcal{C}_0$ with objects as above and morphisms defined by 
compression bordisms $(W,W\cap (M\times \{0\}),W\cap (M\times \{1\}))$ 
embedded in $M\times [0,1]$ with the \textit{compression bordism structure} induced from the embedding. This means more precisely that the 
restriction of the projection $M\times [0,1]\rightarrow [0,1]$ to $W$ is a generic Morse function with only a finite number of critical points of 
index $2$ and $3$ at a finite number of levels $\{t_1,\ldots ,t_r\}$. (Note that the embedding can be isotoped such that the index 
$2$ critical points precede the index $3$-critical points, see \cite{KL}) We consider the embeddings up to isotopy 
\textit{preserving} the $2$-$3$ handle structure. This means that the isotopies of embeddings do not introduce any index $0$ or $1$ critical points.
We call these isotopies \textit{Heegaard isotopies}.
Note that a Heegaard isotopy will induce a regular deformation of generic Morse functions.
It is known from Cerf theory \cite{Ce} that only the following \textit{catastrophes} appear during a generic deformation:
the order of critical points changes, or the deletion or insertion of a canceling $2-3$-handle pair.
Note that changing the order of critical points corresponds to passing through a Morse function with two critical points on the same level.
The composition of two embedded compression bordisms is defined as usual.
Obviously $\mathcal{C}_0$ is \textit{not} a bordism category in the usual sense because a 
morphism $S\rightarrow S'$ does \textit{not} define a morphism $S'\rightarrow S$ by turning it upside down.
We like to think of $\mathcal{C}_0$ as a \textit{directed} version of a bordism category, see \cite{G} for ideas around directed topology.

\begin{exmp} Let $\alpha =\emptyset$ and consider a orientable morphism $W\subset M\times [0,1]$ in the category $\mathcal{C}_0$  
from an orientable surface $S\subset M\times \{0\}$ to $\emptyset\subset M\times \{1\}$. Then $W$ is an embedded compression body, and in fact 
for $S$ connected it is an embedded handlebody. This explains the name Heegaard splitting of a manifold triad. It is  a splitting of a 
morphism in the general $3$-manfold triad category into two morphisms in the compression bordism category. 
\end{exmp}

It will be necessary to replace the above categories by their \textit{skeleton} categories $\mathcal{C}:=\mathcal{C}(M,\alpha )$ respectively
$\mathcal{M}:=\mathcal{M}(M; \alpha )$ for the skeleton categories corresponding to $\mathcal{C}_0$ respectively $\mathcal{M}_0$. See \cite{Ma} p.\ 93 for the general definitions. We first consider $\mathcal{C}$. 
We replace the set of objects by the set of \textit{isotopy classes} of properly embedded surfaces $S\subset M$ with boundary $\alpha $
and choose representative embedded surfaces. Then the morphisms are isotopy classes of compression bordisms embedded in $M\times [0,1]$ relative to the boundaries as above but are equipped with explicit isotopies of $M\times \{0\}$ respectively $M\times \{1\}$ carrying the surfaces in the boundaries to one of the standard representatives of objects.
We need additional morphisms defined by isotopy classes of (ambient) self-isotopies of the standard surface embeddings.
Then $\mathcal{C}(M,\alpha )$ is called the \textit{Bar-Natan category} of $(M,\alpha )$, and we let
$\mathcal{C}(M):=\mathcal{C}(M;\emptyset )$. 

\begin{rem} A bordism $W\subset M\times [0,1]$ can be thought of as a sequence of paths in a space built from the disjoint union of spaces of 
embeddings of surfaces $S\subset M$. The space and paths can be completed by
adding (i) \textit{nodal} surfaces corresponding to index $2$ critical points (a nodal surface is smoothly embedded except for some circle which is contracted to a point in $M$; we assume that there is a distinguished plane in the tangent space of $M$ at the singular point tangential to the two branches of the surface), and (ii) $2$-spheres contracted to points in $M$ corresponding to index $3$ critical points. The nodal surfaces define neck-cuttings from the two natural desingularizations defined by blowing up the singular point to a circle within the given plane,
respectively by doing the usual cut and paste surgery on the circle and replacing the annular neighborhood by two disks. (This process 
is actually only defined up to \textit{local} isotopy but this will not be a problem because we will not explicitly need this construction in the following.)  
The space of surfaces can be thought of as the quotient of the embedding space by the diffeomorphisms of the domains.
We do not discuss the technical details of the topology on this space, and how nodal surfaces are added. This can be done but requires technical work not necessary for the purposes in this paper.
\end{rem}

The mapping space viewpoint suggests a description of the category $\mathcal{C}$ in terms of generating morphisms and relations, i.\ e.\ 
as a quotient of a free category on a directed graph by relations. First define the directed graph $\mathfrak{c}$ with vertex set a set of representative embeddings for each isotopy classes of surfaces $S\subset M$ bounding $\alpha $. 
For each isotopy class of a nodal surface consider the arrow defined as follows: The domain is the isotopy class of surface containg the neighborhood of the critical curve on the surface, and the codomain is the result of replacing the annulus by two disks (the boundary of attaching the $2$-handle). 
More precisely the generating morphisms are defined by those nodal surfaces equipped with fixed isotopies to standard representatives. 
Similarly we define arrows for index $3$ critical points describing the vanishing of an embedded $2$-sphere, equipped with paths to standard representatives. Finally we need arrows corresponding to \textit{isotopy classes of isotopies} of representatives to itself. (These correspond to homotopy classes of loops in the surface space at representative embeddings. Note that by isotopy extension each self-isotopy is induced from an ambient isotopy of $M$ fixing $\alpha $.) The category $\mathcal{C}$ above is a quotient of the free category 
$\mathcal{F}\mathfrak{c}$ on the directed graph $\mathfrak{c}$ by certain relations. By Cerf theory \cite{Ce} these relations are all given by changing the order of 
critical points, and by the cancellation or appearance of canceling $2$/$3$ handle pairs. Note that isotopies can also be rearranged with respect to the 
arrows corresponding to the critical points. This follows from the fact that application of an isotopy followed by a neck-cutting or sphere disappearence 
is equivalent to the same with the ambient isotopy followed afterwards. 

The skeleton category is less natural but technically necessary for our constructions. It is known \cite{M} that it is equivalent to the original category $\mathcal{C}$. In particular the definition of our functors in section \ref{sec=algebra} requires the skeleton category.

\vskip .05in

The category $\mathcal{M}$ is similarly defined by the skeleton category of the category $\mathcal{M}_0$. 
We include all the morphisms \textit{reverse} to those in $\mathcal{C}$, i.\ e.\ morphisms
corresponding to $3$-manifold triads defined by attaching $0$-handles and $1$-handles. Note that there is a natural functor $\mathcal{C}\rightarrow \mathcal{M}$, which is the identity on objects.  
As above we let $\mathcal{M}(M):=\mathcal{M}(M; \emptyset)$.  
The category $\mathcal{M}$ has a similar presentation by a directed graph $\mathfrak{m}$. It is important that now there are additional relations corresponding to the deletion or insertion of 
cancelling $1$/$2$ handle pairs. We have the projection functors $\mathcal{F}\mathfrak{c}\rightarrow \mathcal{C}$ and 
$\mathcal{F}\mathfrak{m}\rightarrow \mathcal{M}$. We call the morphisms of the free categories corresponding to the edges of the graphs
respectively their images in $\mathcal{C}$ respectively $\mathcal{M}$ \textit{elementary morphisms}. Thus an elementary morphism of $\mathcal{C}$ 
corresponds to attaching of a $2$-handle or a $3$-handle or an isotopy. An elementary morphism in $\mathcal{M}$ corresponds to attaching a handle 
or an isotopy. We also introduce the \textit{mirror category} of $\mathcal{C}$ denoted $\mathcal{C}^-$ for which the morphisms $S\rightarrow S'$ are just the 
morphisms $S'\rightarrow S$ in $\mathcal{C}$. This category consists of bordisms from $S$ to $S'$ with only $0$ and $1$-handles. The category 
$\mathcal{C}^-$ is usually just called the opposite category. It is presented as an obvious quotient of a directed graph
$\mathfrak{c}^-$. Note that $\mathfrak{c}\cup \mathfrak{c}^-=\mathfrak{m}$ and $\mathfrak{c}\cap \mathfrak{c}^-$ has the same vertex set as
$\mathfrak{c}$ but the arrows are only those corresponding to isotopies. 

\begin{rem} \label{rem=trivial} The neck-cutting morphisms defined by cutting along circles, which bound disks on the surfaces (trivial neck-cuttings) in $\mathcal{C}$ are morphisms from surfaces $S$ to surfaces $S\cup S_0$, where $S_0$ is a $2$-sphere. We call such a morphism a \textit{trivial $2$-handle morphism} if $S_0$ bounds a $3$-ball in $M$. Then there is morphism back from $S\cup S_0$ to $S$, which attaches the $3$-handle killing the $2$-sphere. Trivial morphisms are the only elementary morphisms in $\mathcal{C}$, which admit left-inverses. 
\end{rem}

It is possible that the category $\mathcal{C}(M,\alpha )$ is the empty category $\bf{0}$. This happens if and only if there is no surface $S\subset M$ bounding $\alpha $. This can be stated homologically as follows. Let $\mathfrak{h}(\alpha )$ denote the \textit{subset} of $H_1(\partial M;\mathbb{Z})$ consisting of all homology classes defined by all possible orientations of the components of $\alpha $. 
Then $\mathcal{C}(M,\alpha )$ is the empty category if and only if the intersection of the image of $\partial : H_2(M,\partial M;\mathbb{Z})\rightarrow H_1(\partial M;\mathbb{Z})$ with $\mathfrak{h}(\alpha )$ is empty. 
On the other hand, the categories of closed surfaces (even in empty manifolds 
$M$) are nonempty because they always have the empty surface vertex.

\vskip .1in

Besides the categories $\mathcal{C}(M,\alpha )$ there exist two natural categories for each non-negative integer $n$: The category $\mathcal{C}[n]$ is defined by non-embedded orientable surfaces with $n$ boundary circles and corresponding bordisms, respectively $\mathcal{U}[n]$ is defined by non-embedded possibly unorientable surfaces with $n$ boundary components and corresponding compression bordisms. Note that we assume that our Bar-Natan relations preserve orientability. 
The boundary circles are fixed, the surfaces are representatives of their diffeomorphism classes
and the morphisms are diffeomorphism classes relative to the boundary of compression $3$-manifolds between those surfaces. We write $\mathcal{C}(\alpha )$ if $\alpha $ is fixed. Note that the categories $\mathcal{C}(\alpha )$
for different $\alpha $ with $|\alpha |=n$ are isomorphic (naturally up to choosing an ordering of the components of $\alpha $), and thus $\mathcal{C}[n]$ actually denotes the \textit{isomorphism class} of the corresponding categories.
Note that a neck-cutting of a non-orientable surface can be a morphism from a non-orientable surface to an orientable surface. There is a natural inclusion functor $\mathcal{C}[n]\rightarrow \mathcal{U}[n]$ for each $n$. Obviously the category
$\mathcal{C}[n]$ is a full subcategory of the category $\mathcal{U}[n]$. 

\vskip .05in

The association $(M,\alpha )\mapsto \mathcal{C}(M,\alpha )$ is functorial in the sense that an embedding $i: (M,\alpha )\hookrightarrow (N, i(\alpha ))$ induces a functor of categories $\mathcal{C}(i): \mathcal{C}(M,\alpha )\rightarrow \mathcal{C}(N,i(\alpha ) )$, with isotopies of embeddings $i$ inducing natural transformations of functors. 
In fact, if $j_t$ is a Heegaard-isotopy of $M\times I$ isotoping a $3$-manifold bordism $W_0\subset M\times I$ to $W_1\subset M\times I$ then 
$(i\times \textrm{Id})\circ j_t$ will be a Heegaard-isotopy  from $(i\times \textrm{Id})(W_0)\subset N\times I$ to $(i\times \textrm{Id})(W_1)\subset N\times I$. Thus $\mathcal{C}(i)$ is well-defined on the level of morphisms.  
In particular diffeomorphisms $d$ of $M$ relative to $\partial M$ induce functors
$$\mathcal{C}(M,\alpha )\rightarrow \mathcal{C}(M;d(\alpha )),$$
which are equivalencies of categories.
So for $M\neq \emptyset$ there exist functors $\mathcal{C}(B)\rightarrow \mathcal{C}(M)$ for the $3$-ball $B$, and in general 
$\mathcal{C}(M,\alpha )\rightarrow \mathcal{C}(\alpha )$. The inclusion of a ball $B\hookrightarrow M$ induces a functor
$$\mathcal{C}(B)\times \mathcal{C}(M,\alpha )\rightarrow \mathcal{C}(M,\alpha ).$$
This functor lifts the module action of the closed Bar-Natan skein module of the $3$-ball on the Bar-Natan skein module of a $3$-manifold, which is just the $R$-module action of the skein module, see section \ref{sec=functor}. 
Let $\mathcal{U}(M,\alpha )$ denote the category defined just as $\mathcal{C}(M,\alpha )$ but including unorientable surfaces.
The functor $\mathcal{U}(B;\alpha )\rightarrow \mathcal{U}(\alpha )$ factors through a functor $\mathcal{U}(B;\alpha )\rightarrow \mathcal{C}(\alpha )$ because closed surfaces embedded in $B$ are orientable.
The functor $\mathcal{U}(B;\alpha )\rightarrow \mathcal{C}(\alpha )$ induces an isomorphism of the corresponding Bar-Natan modules (see section \ref{sec=functor} and \cite{K1}) while the two categories are not equivalent. In fact is is not hard to see that the categories $\mathcal{U}(B;\alpha )$ depend on the embedding $\alpha \subset \partial B\cong S^2$ through possible nesting of the components of $\alpha $, see also \cite{Be} for related observations.

\section{Linear categories from Frobenius algebras}\label{sec=algebra}

The most difficult result in this section is \ref{thm=glueing functor}, which is a statement about \textit{algebraic glueing}. This result will be necessary to prove the functor properties with respect to composition in \ref{thm=main} but in particular for the bicategory Theorem
\ref{thm=bifunctor}. The reader may skip some of the involved algebraic arguments here and first read section \ref{sec=bicategory} for motivation. 

Recall that $R$ is a unital commutative ring. Throughout $\textrm{Id}$ will denote the identity morphism. The tensor products will be over $R$. 
The ordering of factors in tensor products is understood in the \textit{categorical} sense as defined in \cite{DM}.
Thus a tensor product of a family of $R$-modules $\{V_i\}_{i\in I}$ assigns to each linear ordering, i.\ e.\ bijection $\sigma : \{1,2,\ldots ,|I|\}\rightarrow I$ the \textit{usual} ordered tensor product $V_{\sigma(1)}\otimes \ldots \otimes V_{\sigma (|I|)}$. We usually just write $I=\{1,2,\ldots ,|I|\}$ when we have chosen a linear ordering of $I$ and thus identify the elements $i\in I$ with $\sigma ^{-1}(i)\in \{1,2,\ldots ,|I|\}$, and then write
$V_1\otimes \ldots \otimes V_{|I|}$. 
The categorical tensor product we use is defined in the usual category of $R$-modules following 
\cite{DM}. Thus there is the following coherence structure: For each permutation there is defined 
an induced isomorphism. In our case these are the usual permutation isomorphisms of the tensor factors. The morphisms 
in the category are $R$-homorphisms between usually ordered tensor products compatible with the isomorphisms induced by permutations. 
As pointed out in \cite{DM} one can formally define the categorical tensor product as the projective limit of the 
usually ordered tensor products with all possible orderings.
For $V$ a fixed $R$-module and $I$ a finite (unordered) set we will write $V^{\otimes I}$ for the tensor product 
of the family of $\{V_i\}_{i\in I}$ with $V_i=V$ for all $i\in I$. Note that $V^{\otimes \emptyset}=V^{\otimes 0}:=R$.

\vskip .05in

Let $V$ be an $R$-module with coassociative and cocommutative coproduct $\Delta $ with counit $\varepsilon :V\rightarrow R$ satisfying $(\varepsilon \otimes \textrm{Id})\circ \Delta=\textrm{Id}$, equipped with an $R$-endomorphism $\mathfrak{k}: V\rightarrow V$ such that 
$(\textrm{Id}\otimes \mathfrak{k})\circ \Delta=\Delta \circ \mathfrak{k}$. This is called a \textit{weak Frobenius algebra}.
Weak Frobenius algebras usually are defined from commutative Frobenius algebras over $R$ (see below) by forgetting the multiplication.
We will see that the algebraic input of a weak Frobenius algebra suffices 
to define a functor from the topological category $\mathcal{C}$ defined in section \ref{sec=compression} into a category of $R$-modules.
The corresponding colimit defines a Bar-Natan module over $R$ for each weak Frobenius algebra. 
But the glueing result at the end of this section and the $2$-categorical setting of section 4 require the refined module structures based directly on the multiplication of the Frobenius algebra.    

\vskip .05in

Recall that a commutative Frobenius algebra $V$ over the commutative unital ring $R$ is a commutative and associative 
$R$-algebra multiplication $m: V\otimes V\rightarrow V$, unit $\mu :R\rightarrow V$, cocommutative and coassociative comultiplication $\Delta : V\rightarrow V\otimes V$ and counit $\varepsilon : V\rightarrow R$ such that $(\varepsilon \otimes \textrm{Id})\circ \Delta=\textrm{Id}$ and such that $\Delta $ is a bimodule homomorphism. Here $V\otimes V$  is considered a $V$-bimodule using $v \cdot (u\otimes w)=(vu)\otimes w=u\otimes (vw)$.  
Throughout let $\mathfrak{k}:=m\circ \Delta (1)\in V$, 
$1:=\mu (1)$, and $\mathfrak{k}: V\rightarrow V$ also denote the multiplication by $\mathfrak{k}$.
The morphism $\mathfrak{k}$ is also called the \textit{handle operator}. Not surprisingly it will play the most essential role in the study of embedded Bar-Natan theory.
By forgetting the multiplication and unit but keeping the endomorphism $\mathfrak{k}$ each commutative Frobenius algebra defines a weak Frobenius algebra. For $v\in V$ we let $\Delta (v)=v^{(1)}\otimes v^{(2)}$ using Sweedler notation. Note that because of commutativity respectively cocommutativity, $m$ respectively $\Delta $ are obviously well-defined if the tensor products are identified as categorical as above. 

\vskip .05in

For $I$ a finite set we have that $V^{\otimes I}$ is an $R$-algebra with the product defined for a specific linear ordering $I=\{1,2,\ldots j\}$ by:
$$(v_1\otimes \ldots \otimes v_j)(w_1\otimes \ldots \otimes w_j):=(v_1w_1\otimes \ldots \otimes v_jw_j).$$ 
Note that a bimodule map $V\rightarrow V\otimes V$ is a $V^{\otimes 2}$-module map 
using the $V^{\otimes 2}$-module structure on $V$ defined by $(u\otimes v)\cdot w=uvw$ and the algebra structure above on $V\otimes V$. Note that $V^{\otimes \emptyset}=V^{\otimes 0}:=R$ 
 
\vskip .05in 
 
Next we define certain module categories motivated from the category $\mathcal{C}$ in the following way:
Take a surface $S$ bounding a $1$-manifold $\alpha $. Then think of each component of $S$ as a
copy of $V$. Each component of $\alpha $ determines the component of $S$, which it bounds.
Thus $S$ defines a map $\varphi : \underline{\alpha }\rightarrow \underline{S}$. The idea is to have $V^{\otimes \underline{\alpha }}$ act on $V^{\otimes \underline{S}}$
with the multiplication determined by $\varphi $. More precisely the component of $V^{\otimes \underline{\alpha }}$ 
corresponding to a particular element $j\in \underline{\alpha }$ will multiply into the component of $V^{\otimes \underline{S}}$ determined by $\varphi (j)$. 
The morphisms of the linear categories will correspond to the Bar-Natan relations, and the glueing functor is motivated by a corresponding topological glueing functor.

\vskip .05in

Let $J$ be a finite set and $V$ a commutative Frobenius algebra. Then we define the objects of categories $\mathcal{V}(J)$
as follows:  
The objects of $\mathcal{V}(J)$ are $V^{\otimes J}$-modules of the form $V^{\otimes K}$ with $K\neq \emptyset$
if $J\neq \emptyset$. They are
determined by maps $\varphi: J\rightarrow K$.
If $J=\emptyset $ then we have the unique map $\emptyset \rightarrow K$ for any finite set $K$.
The objects of $\mathcal{V}(\emptyset)=:\mathcal{V}$ are the usual $R$-modules $V^{\otimes K}$
for finite sets $K$. 
In general observe that $\varphi $ defines the disjoint sets $I_k:=\varphi ^{-1}(k)\subset I$ for each $k\in \varphi (J)$
such that $I=I_1\cup \ldots \cup I_{\ell }$ with $\ell :=|\varphi (J)|$.
Then the action of $v_1\otimes \ldots \otimes v_{|J|}$ on $w=w_1\otimes \ldots \otimes w_{|K|}$ is 
defined by multiplying the factor of $w$ corresponding to $k$ with all the $v_j$ with $j\in I_k$.
Note that this is well-defined independently of orderings chosen on $J$ and $K$. 
We denote the corresponding $V^{\otimes J}$-module defined by $\varphi : J\rightarrow K$ by $V(\varphi )$. We often write $V^{\otimes K}$
when $\varphi $ is clear. 
For example we have the $V^{\otimes K}$-module $V(\textrm{Id})$ defined by $\textrm{Id}: K\rightarrow K$, which is just the standard $V^{\otimes K}$-module $V^{\otimes K}$ defined above with module structure defined by the algebra structure.
Note that a different module structure is defined if $\textrm{Id}$ is replaced by a nontrivial permutation of the elements of $K$.  
The $V^{\otimes 2}$-module $V$ mentioned above is $V(\varphi )$  with $\varphi : \{1,2\}\rightarrow \{1\}$ the unique map.
\vskip .1in

When describing the module actions we often choose the ordering such that the elements of $\varphi (J)$ precede all the elements of $K\setminus \varphi (J)$. Moreover, given a linear ordering of $J$ there is an induced linear ordering of $\varphi (J)$ as follows.
Note that an ordering of $\varphi (J)$ is the same as an ordering of the set of subsets 
$\{I_1,\ldots ,I_{\ell }\}$.
If we denote the elements of $\varphi (J)$ by $\{1,2,\ldots ,\ell \}$ in the order we want to define then we can arrange that the smallest element in $I_1$ is $1$, the smallest element in $I_2$ is the smallest element in $I\setminus I_1$, and inductively the smallest element in $I_i$ is the smallest element in $I\setminus (I_1\cup \ldots \cup I_{i-1})$. Thus if the set $\{I_1,I_2,\ldots ,I_{\ell }\}$ is ordered by the smallest elements in the $I_j$ then we have that $1\leq i<j\leq \ell$ in $K$ implies $I_i<I_j$. 

\vskip .05in
   
Next we define the morphisms of $\mathcal{V}(J)$.
It will be important for our glueing functors that the morphisms are not all algebraic $V^{\otimes J}$-morphisms
but certain $V^{\otimes J}$-submodules of the modules of all $V^{\otimes J}$-morphisms.
For $J=\emptyset$ the module of morphisms will be generated by taking arbitrary $R$-linear combinations,
compositions and tensor products of $\Delta $, $\mathfrak{k}$ and $\varepsilon $ and permutations of tensor factors.
Let $J$ be an arbitrary finite set (possibly the empty set, in which case we will repeat the definition above).
Then the module of morphisms of $\mathcal{V}(J)$ is generated by 
$V^{\otimes J}$-linear combinations and compositions of morphisms of types (i)-(iv), called the 
\textit{elementary morphisms}.
These model algebraically separating neck-cutting, non-separating neck-cutting, sphere relations and permutations of closed surface components (induced by possible self-isotopies in the embedded setting). 

\textbf{(i):} There are $V^{\otimes J}$-homomorphisms
$V(\varphi )\rightarrow V(\varphi ')$
for $\varphi : J\rightarrow K$ and $\varphi ': J\rightarrow K'$ where
$K'=K\cup \{k'\}$ with $k'\notin K$. 
Let $r\in K$ and let $\phi^{-1}(r)=I_r=I_r'\cup I_r''$ be a disjoint union
(empty sets are allowed).
Then let $\varphi '(j)=\varphi (j)$ for all $j\notin I_r''$ and let $\varphi '(j)=k'$ for 
$j\in I_r''$. 
Then we define the morphism  
$$V(\varphi )\rightarrow V(\varphi ')$$
by the identity on all factors corresponding to $k\neq r$ while the tensor factor 
corresponding to $r$ is mapped to the tensor factors corresponding to $r$ and $k'$ 
using $\Delta $. 
The reason that this is $V^{\otimes J}$-linear is as follows: Firstly, the factors of $V^{\otimes J}$ corresponding to $I_r$
are all multiplied into $w_r=:w$, the $r$-th factor in $V^{\otimes k}$. For simplicity let $I_r=\{1,\ldots ,s\}$,
$I_r'=\{1,\ldots ,s'\}$ and $I_r''=\{s'+1,\ldots ,s\}$. By considering only the relevant factors of the tensor product $V^{\otimes K}$ we have
$\Delta ((v_1\otimes \ldots \otimes v_s)w)=\Delta (v_1\ldots v_sw)=v_1\ldots v_{s'}\Delta (w)v_{s'+1}\ldots v_s=
v_1\ldots v_{s'}w^{(1)}\otimes v_{s'+1}\ldots v_sw^{(2)}$ and also 
$(v_1\otimes \ldots \otimes v_s)\Delta (w)=(v_1\otimes \ldots \otimes v_s)(w^{(1)}\otimes w^{(2)})=v_1\ldots v_{s'}w^{(1)}\otimes v_{s'+1}\ldots v_sw^{(2)}$
because $\Delta $ is a $V$-bimodule map. 

\textbf{(ii):} There are $V^{\otimes J}$-morphisms of the form $\textrm{Id}\otimes \mathfrak{k}\otimes \textrm{Id}: V(\varphi )\rightarrow V(\varphi )$  with $\mathfrak{k}$ acting on a factor of $V^{\otimes K}$ corresponding to some $r\in K$. 
The resulting homomorphisms are $V^{\otimes J}$-linear.  
In fact, let $\varphi^{-1}(r)=\{1,\ldots s\}$. Then we have (ignoring factors in $v_1\otimes \ldots \otimes v_{|J|}$,
which are not in $\varphi^{-1}(r)$): 
$(v_1\otimes \ldots v_s)\mathfrak{k}w=v_1\ldots v_s\mathfrak{k}(w)=\mathfrak{k}(v_1\ldots v_sw)=\mathfrak{k}((v_1\otimes \ldots \otimes v_s)w)$. 

\textbf{(iii):} For each $k\in K\setminus \varphi (J)$ let $K':=K\setminus \{k\}$. Let $\varphi ': J\rightarrow K'$ be defined by restricting the codomain. Then by applying $\varepsilon $ to the factor in $V^{\otimes K}$ corresponding to $k$ and identifying the result using $V^{\otimes K'}\otimes R\cong V^{\otimes K'}$ we have defined a $V^{\otimes J}$-morphism.

\textbf{(iv):} Finally we have permutations of the tensor factors of $K\setminus \varphi (J)$ considered as $\mathcal{V}(J)$-automorphisms of $V(\varphi )$.

\begin{rem} The category $\mathcal{V}(\emptyset )$ only requires a weak Frobenius algebra. 
The structure of the category $\mathcal{V}(J)$ for $|J|\geq 1$ seems difficult at first but is just an algebraic reformulation of how the colorings 
change under neck cutting in a situation with $J=\underline{\alpha}$ the set of boundary components of the surfaces, compare also the combinatorial category defined in section 
\end{rem}

The category $\mathcal{V}(J)$ is an $Ab$-category (see \cite{M}) but not an additive category because
it has no direct sums. 

\begin{defn} \label{def=completion}  For each finite set $J$ let $\overline{\mathcal{V}(J)}$ denote the smallest abelian subcategory of the category of $V^{\otimes J}$-modules containing $\mathcal{V}(J)$, 
which is closed respect to taking arbitrary direct sums and quotients. Equivalently this is the smallest abelian cocomplete subcategory of the category of $V^{\otimes J}$-modules, see \cite{W}, 2.6.8., which contains $\mathcal{V}(J)$.
\end{defn}

The category $\mathcal{V}$ is monoidal. This is \textbf{not} the case for the categories $\mathcal{V}(J)$ for $|J|>0$. But there are obvious 
functors for $J',J''$ any two finite sets and $J=J'\cup J''$:
$$\otimes: \mathcal{V}(J')\times \mathcal{V}(J'')\rightarrow \mathcal{V}(J)$$
with $J=J'\cup J''$. 
It is defined on objects with $\varphi ': J'\rightarrow K'$ and $\varphi '': J''\rightarrow K''$ by 
$V(\varphi ):=V(\varphi ')\otimes V(\varphi '')$ with $K=K'\cup K''$ and $\varphi (i)=\varphi '(i)$ for $i\in J'$ and $\varphi (i)-\varphi '(J")$ for all $i\in J''$. 
Note that for $J'=\emptyset $ with $V^{\otimes \emptyset }$-module $V^{\otimes K}$ and the map $\varphi '' : J''\rightarrow K''$ it follows that $\varphi =\iota \circ \varphi '' $, where $\iota $ is the inclusion $K''\subset K$.

\vskip .1in

In particular $\mathcal{V}(J)$ is a module category over the monoidal category $\mathcal{V}$.
There also exists the natural forget functor $\mathcal{V}(J)\rightarrow \mathcal{V}$  by forgetting the $V^{\otimes J}$-module structure and only keeping the $R$-module structure. In fact any inclusion of finite subsets $J'\subset J$ will induce a forget functor $\mathcal{V}(J)\rightarrow \mathcal{V}(J')$. 

\vskip .1in

The most interesting feature of the categories constructed is the existence of \textit{glueing functors}.

For this it is necessary to introduce first categories $\mathcal{V}(J_1,J_2)$ for finite sets $J_1,J_2$.
We define the objects of the category $\mathcal{V}(J_1,J_2)$ 
to be $V^{\otimes J_1}\otimes V^{\otimes J_2}$-modules (the tensor product in the middle here is understood
as an ordered tensor product, i.\ e.\ we talk about bimodules)
of the form $V^{\otimes K}$ defined by an ordered pair of two maps $J_i\rightarrow K$ for $i=1,2$.
We allow the $J_i$ to be empty in which case we identify $V^{\otimes \empty J_i}=R$ 
and identify objects and morphisms $\mathcal{V}(J,\emptyset )=\mathcal{V}(J)=\mathcal{V}(\emptyset ,J)$.
In the general case objects and morphisms are defined exactly as before by using the 
map $\varphi : J_1\cup J_2\rightarrow K$ defined by $\varphi |J_i=\varphi_i$ for $i=1,2$, and correspondingly restricting to $J_1$ and $J_2$ to define the resulting pair of maps.
The difference is only that we interpret the resulting modules as bimodules, and correspondingly the morphisms
as bimodule morphisms.  

\vskip .1in

\begin{thm} \label{thm=glueing functor} For finite sets $J_1,J_2,J_3$ there is defined a glueing functor
$$\mathfrak{g}=\mathfrak{g}(J_1,J_2,J_3): \mathcal{V}(J_2,J_3)\times \mathcal{V}(J_1,J_2)\rightarrow \mathcal{V}(J_1,J_3)$$
\end{thm}

\begin{proof} Given the object $V(\varphi_1,\varphi_2)$ in $\mathcal{V}(J_1,J_2)$ with $\varphi_i: J_i\rightarrow K'$ for $i=1,2$ and the object 
$V(\varphi_2',\varphi_3)$ in $\mathcal{V}(J_2,J_3)$ with $\varphi_2': J_2\rightarrow K'$ and $\varphi_3: J_3\rightarrow K''$. Then we define 
the set $K$ by taking the set of equivalence classes in the disjoint union $K'\cup K''$ (replace $K'$ respectively $K''$ by $K'\times \{0\}$ respectively $K''\times \{1\}$) by the equivalence relation
$\varphi_2(j)\sim \varphi_2'(j)$ for all $j\in J_2$. We define $\varphi :J_1\rightarrow K$ respectively 
$\varphi ': J_3\rightarrow K$ by composition of $\varphi_1 $ respectively $\varphi_3$ with the maps
$K'\subset K'\cup K''\rightarrow K$ respectively $K''\subset K'\cup K''\rightarrow K$ where $K'\cup K''\rightarrow K$ is the projection onto the set of equivalence classes. This defines 
$\mathfrak{g}(V(\varphi_2',\varphi_3),V(\varphi_1,\varphi_2)):=V(\varphi ,\varphi ')$. 
Note that $V(\varphi_2',\varphi_3)\otimes V(\varphi_1,\varphi_2)=V^{\otimes K''}\otimes V^{\otimes K'}$ is also a 
$V^{\otimes J_1}\otimes V^{\otimes J_3}$-module in the obvious way and there is a natural projection
$V(\varphi_2',\varphi_3)\otimes V(\varphi_1,\varphi_2)\rightarrow V(\varphi, \varphi ')$ 
($V^{\otimes J_1}\otimes V^{\otimes J_3}$-morphism), defined by multiplications 
$m: V\otimes V\rightarrow V$ for factors with  $\varphi_2(j)=\varphi_2'(j)$ and $j\in J_2$.

We will indicate how to construct the functor on morphisms. It suffices to consider elementary morphisms 
(i)-(iv) above and extend using composition and linearity.  
Given $f: V(\varphi_1,\varphi_2)\rightarrow V(\psi_1,\psi_2)$ and $g: V(\varphi_2',\varphi_3)\rightarrow V(\psi_2',\psi_3)$ with $\varphi_i,\varphi_2'$ as above for $i=1,2,3$, and $\psi_i: J_i\rightarrow L'$ for $i=1,2$, $\psi_2': J_2\rightarrow L''$, $\psi_3: J_3\rightarrow L''$. Then we have defined 
$\mathfrak{g}(V(\varphi_2',\varphi_3),V(\varphi_1,\varphi_2))=V(\varphi ,\varphi ')\in \mathcal{V}(J_1,J_3)$ and
$\mathfrak{g}(V(\psi_2',\psi_3),V(\psi_1,\psi_2))=V(\psi ,\psi ')\in \mathcal{V}(J_1,J_3)$
and have to define a homomorphism $\mathfrak{g}(g,f): V(\varphi ,\varphi ')\rightarrow V(\psi ,\psi ')$ in the category $\mathcal{V}(J_1,J_3)$. 
Using the natural projection $P: V(\psi_2',\psi_3)\otimes V(\psi_1,\psi_2)\rightarrow V(\psi, \psi ')$ and $g\otimes f: V(\varphi_2',\varphi_3)\otimes V(\varphi_1,\varphi_2)\rightarrow V(\psi_2',\psi_3)\otimes V(\psi_1,\psi_2)$ we can consider $P\circ (g\otimes f)$. By using $g\otimes f=(g\otimes \textrm{Id})\circ (\textrm{Id}\otimes f)$ we see that it suffices to assume that one of $f$ or $g$ is the identity morphism. Note that the set of $\mathcal{V}(J_1,J_3)$-morphisms is closed with respect to composition. So it suffices to see that $P\circ (\textrm{Id}\otimes f)$ factors through a $\mathcal{V}(J_1,J_3)$-morphism $V(\varphi, \varphi ')\rightarrow V(\psi, \psi ')$. This can be checked for all the elementary morphisms (i)-(iv). For example let $f$ be of type (i) and $L'=K'\cup \{l\}$ for some $l\notin K$ and $l=\psi_2(j)=\psi_2'(j)$ for some $j\in J_2$. Then it follows from the previous observation that
$$
\begin{CD}
V\otimes V@>\textrm{Id} \otimes \Delta>>(V\otimes V)\otimes V\cong V\otimes (V\otimes V)@>m\otimes \textrm{Id}>>V\otimes V
\end{CD}
$$ 
is equal to $\Delta \circ m$ noting that the glueing can be thought of using $m$ to multiply factors of the tensor product that are glued together. If $f$ is of type (ii) then we will have to use the identity
$m\circ (\mathfrak{k}\otimes \textrm{Id})=\mathfrak{k}\circ m$ to see that the above homomorphism factors through a 
morphism of the category $\mathcal{V}(J_1,J_3)$. 
It is instructive to draw a corresponding diagram of surfaces replacing surface components for tensor product factors, and having in mind that commutative Frobenius algebras are 
equivalent to $2$-dimensional topological quantum field theories \cite{Ko} and noticing that glueing is algebraically represented by multiplication. 
\end{proof}

\begin{rem} Frobenius algebras and its relation with $2$-dimensional field theory are well-studied \cite{Ko}. Our comments in section 5 will show that the glueing result above has a $2$-categorical interpretation. In a certain way our results in the following two sections will show that some ingredients of $2$-dimensional field theory extend in a restricted way to $3$-dimensional field theory. It should be observed that $\Delta \circ m: V\otimes V\rightarrow V\otimes V$ satisfies the \textit{Yang Baxter equation}, even though this endomorphism is of course not invertible in all nontrivial cases.
 \end{rem} 

We refer to Kock \cite{Ko} for a detailed  discussion of the structure and properties of Frobenius algebras.
Recall that Frobenius algebras often are defined from the non-degenerate Frobenius form $\mathfrak{p} :V\otimes V\rightarrow R$. 
In terms of our definition $\mathfrak{p}=\varepsilon \circ m$.
By checking the definitions it follows that
$$\mathfrak{k}=(\mathfrak{p}\otimes \textrm{Id})\circ (\Delta \otimes \textrm{Id})\circ \Delta.$$ 
Because of the above definitions we are particularly interested in the algebra of endomorphisms of the tensor algebra of $V$, which is generated by composition, tensor product
and $R$-linear combinations of $\Delta $, $\varepsilon$ and $\mathfrak{k}$ and permutations of factors. 
Let $d_n: V\rightarrow V^{\otimes n}$ be defined by $d_0=\varepsilon $, $d_1=\textrm{Id}$, $d_2=\Delta $ and 
$d_n=(\textrm{Id}\otimes d_2)\otimes d_{n-1}$. 
Then the algebra is $R$-generated by the $d_n$ and $\mathfrak{k}:V\rightarrow V$ (or $\mathfrak{p}$).
Note that the family of morphisms $\{d_n\}$ satisfies a nice \textit{operadic} structure. For example 
$(d_{n_1}\otimes d_{n_2}\otimes \ldots \otimes d_{n_k})\circ d_k=d_{n_1+\ldots +n_k}$. This follows from coassociativity.
Also there is nice compatibility with permutations. $\sigma_* \circ d_n=d_n$ if $\sigma _* $ is the permutation of $V^{\otimes n}$ induced by the permutation $\sigma $.
For $n\geq 1$ also define $\varepsilon^n: V^{\otimes n}\rightarrow V$ inductively by $\varepsilon^1=\textrm{Id}$ and 
$\varepsilon^n=(\varepsilon \otimes \textrm{Id})\circ \varepsilon^{(n-1)}$.
Then $\varepsilon^n\circ d_n=\textrm{Id}$ showing explicitly a left inverse of $d_n$.  

\vskip .1in

The basic examples of commutative Frobenius algebras over fields come from the cohomology modules of 
connected closed $n$-dimensional oriented manifolds $Y\neq \emptyset $ with nonzero cohomology only 
in even dimensions.
For the rest of this section let $R$ be a field.
The Frobenius structure is usually described in terms of de Rham cohomology with the wedge product defining the multiplication and $\varepsilon $ given by integration over the manifold.
We like to give a description in terms of the homology of $Y$. 
For the following see \cite{A} for details. Under Poincare duality also $H_*(Y)$ is a Frobenius algebra. 
The multiplication on $H_*(Y)$ is given by the intersection pairing
$$\mathfrak{i}: H_p(Y)\otimes H_q(Y)\rightarrow H_{p+q-n}(Y).$$
The unit $\mu :R\rightarrow H_*(Y)$ is defined by $1\mapsto [Y]\in H_n(Y)$ with $[Y]$ the fundamental class. 
Let $\varepsilon : H_*(Y)\rightarrow H_0(Y)=H_*(\textrm{pt})$ be the projection on the 
$0$-dimensional homology, which is induced by the projection from $Y$ onto a point $\textrm{pt}\in Y$. 
Then the composition $\varepsilon \circ \mathfrak{i}$ is the Frobenius pairing.  
Let $\Delta : H_*(Y)\rightarrow H_*(Y\times Y)\cong H_*(Y)\otimes H_*(Y)$ be the map in homology induced
by the diagonal and the Eilenberg-Zilber isomorphism. This map is the coproduct on $H_*(Y)$.
It is interesting to note that the coproduct and counit are defined for any topological space and give rise to a coalgebra structure, see \cite{Mc}. This coalgebra structure always has a unit $\eta : R\rightarrow H_*(Y)$ defined by 
mapping $1\in R$ to the generator of the $0$-dimensional homology of $Y$. We have to distinguish this unit from the unit for the multiplication.
The interesting point for us here is that $\Delta $ and $\varepsilon $ are essentially induced by geometric maps. 
Note that the handle operator is the Poincare dual of the Euler class of the manifold. 

\begin{exmp}[Khovanov] \label{example=Khovanov} The Frobenius algebra $H_*(\mathbb{C}P^1)=H_*(S^2)$ has generators $1\in H_0(S^2)$ and 
$b\in H_2(S^2)$ with multiplications $m(1\otimes 1)=0$, $m(1\otimes b)=m(b\otimes 1)=1$ and $m(b\otimes b)=b$.
The coproduct is given by $\Delta (1)=1\otimes 1$ and $\Delta (b)=1\otimes b+b\otimes 1$. Finally we have
$\varepsilon (1)=1, \varepsilon (b)=0$, $\mu (1)=b$, and  
the handle operator $\mathfrak{k}$ is multiplication by $m\circ \Delta (b)=2$ (The Poincare dual of the Euler class of the tangent bundle of $S^2$. Recall that in the Khovanov Frobenius algebra $H^*(S^2)\cong R[x]/(x^2)$ the handle operator given by multiplication by $2x$, which is the Euler class of the tangent bundle of $S^2$.).
\end{exmp}
  
\section{The Bar-Natan functor}\label{sec=functor}

For given $(M,\alpha )$ let $\mathfrak{c}$ be the associated directed graph from the presentation of the category $\mathcal{C}=\mathcal{C}(M,\alpha )$. 
We want to define a functor $\mathcal{F}\mathfrak{c}\rightarrow \mathcal{V}(\underline{\alpha })=:\mathcal{V}(\alpha )$, which factors through the quotient category $\mathcal{C}$ of $\mathcal{F}\mathfrak{c}$. The resulting functor $\mathcal{C}\rightarrow \mathcal{V}(\alpha )$
is called the \textit{Bar-Natan functor}. 

For each vertex $S$ of $\mathfrak{c}$ let $F(S):=V^{\otimes \underline{S}}$. This $R$-module has a natural $(V^{\otimes \underline{\alpha }})$-module structure induced from the map $\varphi $, which assigns to each element of $\underline{\alpha }$ the component of $S$, which contains it in its boundary.
 
Next consider a neck cutting arrow $W: S\rightarrow S'$ determined by a nonseparating curve on the surface $S$. Then there is a natural bijection $\underline{S}\rightarrow \underline{S'}$. 
We assign to this neck-cutting arrow the morphism $F(W): V^{\otimes \underline{S}}\rightarrow V^{\otimes \underline{S'}}$ defined by $\mathfrak{k}$ in the tensor factor corresponding to the component
containing the curve. This a morphism $V(\varphi)\rightarrow V(\varphi )$ in the category $\mathcal{V}(\underline{\alpha })$.
Similarly for an arrow $W: S\rightarrow S'$ defined by a separating curve on $S$ we have a bijection between $\underline{S'}$ and $\underline{S}\cup \underline{S_0}$. 
Here we have to choose a component of $S'$, which we consider to be the new component.
Then $F(W)$ is defined by the 
corresponding morphism of type (i) from section \ref{sec=algebra}. The important point is that the morphism will not depend on which of the two components resulting from the neck-cutting is considered to be the new component. 
This follows from the cocommutativity of $\Delta $.
Finally to some arrow corresponding to the vanishing of a $2$-sphere component assign the 
morphism $\varepsilon $ composed with the isomorphisms $V\otimes R\cong V$ respectively $R\otimes V\cong V$. 
Note that $\varepsilon $ only acts on factors corresponding to closed components.
This corresponds to a morphism of type (iii) in section \ref{sec=algebra}.
Then to a self-isotopy of a 
surface $W:S\rightarrow S$ assign the permutation isomorphism $F(W): V^{\otimes \underline{S}}\rightarrow V^{\otimes \underline{S}}$ defined by the induced permutation of the components 
of $S$.  Note that $F(W)$ can be different from the identity only if $W$ has closed components because boundaries are assumed fixed.
Thus $F(W)$ only permutes the factors corresponding to closed components.  
By the universal property we thus have defined a functor $F: \mathcal{F}\mathfrak{c}\rightarrow \mathcal{V}(\alpha )$. 

For the definition of the colimit $\colim F$ or inductive limit of a functor $F$ see \cite{M}, page 67, and also \cite{W} for colimits of functors into categories of $R$-modules. Recall \ref{def=completion} for the definition of the cocompletion of $\mathcal{V}(J)$ for $J$ a finite set.

\vskip .1in

\begin{thm} \label{thm=main} Let $V$ be Frobenius algebra over $R$ (or a weak Frobenius algebra if $\alpha =\emptyset$). Then the functor $F=F_V$ defined above factors to define a functor also denoted $F=F_V: \mathcal{C}(M,\alpha )\rightarrow \mathcal{V}(\alpha )$. 
The colimit of the composite of this functor with the inclusion functor into $\overline{\mathcal{V}(\alpha )}$ is the Bar-Natan module 
$\mathcal{W}(M,\alpha )=\mathcal{W}(M,\alpha;V)$.
In particular the Bar-Natan module is naturally a $(V^{\underline{\alpha }})$-module and an object in the category
$\overline{\mathcal{V}(\alpha )}$.
\end{thm} 

\vskip .1in

We have to compose the functor $F_V$ with the inclusion functor into $\overline{\mathcal{V}(\alpha )}$ because the colimit in general does not exist in $\mathcal{V}(\alpha )$, at least in all nontrivial cases.  
The functor $F$ is \textit{not} a full functor in general (see the examples below but also compare Theorem 5.1 in \cite{K1} for the case $M$ a closed $3$-ball.). It is \textit{deviation from the functor being full}, which gives the Bar-Natan
modules its specific structure.  
If $\alpha =\emptyset$ then the functor is onto on the object level because there are surfaces with arbitrary number of closed components embedded in a $3$-ball in $M$. But if $\alpha \neq \emptyset$ this is not necessarily the case. and it is not obvious which modules are realized. 
There is a first obstruction from the homology classes of the components of $\alpha $. In fact if $V(\varphi )$ is a module in $\mathcal{V}(\alpha )$ for some $\varphi : \underline{\alpha }\rightarrow K$ with $K$ a finite set,
then the sets $\varphi^{-1}(k)$ for $k\in K$ determine submanifolds $\alpha^k$ of $\alpha $ bounding connected surfaces in $M$, and thus in particular are trivial in $\mathbb{Z}_2$- homology, for all $k\in K$. 
If $\alpha ^k=0\in H_1(M)$ for some $k$ then it bounds a connected surface in $M$. 
But this in general will not suffice to realize the module structure because the $\alpha^k$ have to bound \textit{disjointly} embedded connected 
surfaces in $M$.  If $\alpha $ bounds in $M$ then we only know that there exists some positive integer $n_0$ such that an infinite family of $(V^{\otimes \underline{\alpha }})$-modules
$V^{\otimes K}$ for $|K|\geq n_0$ is realized. (Note that there can exist different $\varphi : \underline{\alpha }\rightarrow K$ with $|K|$ minimal.) In fact if some module is realized then we can always add trivial components 
in a ball to stabilize. If we do not fix a specific $\mathbb{Z}_2$-homology class to be realized then we can omit all closed components.  
If separation allows to form a connected sums of components then we can reduce the cardinality if 
the image of $\varphi $.
Note that we can always arrange that the image of $\varphi $ is just a single element in a connected manifold $M$. Questions of the above form are much more interesting in the oriented case than in the orientable one, which we consider. 
(If we want to keep orientations the possible connected sums are further restricted and it is more difficult to realize the modules.)

\begin{proof} (of Theorem \ref{thm=main}) The first assertion follows easily, for changing the order of two critical points (corresponding to a relation
on the edge set of the free category) means that two induced morphisms, with one defined by composition of two arrows, coincide. But for index $2$ critical points this is clear from the coassociativity of $\Delta $ and the fact that the $\Delta $-actions on different factors of a tensor product commute. The reordering of index $2$ and index $3$ critical points also works because in this case $\varepsilon $ only has to act on a component not associated to the preceding neck-cutting, otherwise the reordering cannot appear. Also the deletion of a cancelling $2$-$3$ handle pair is allowed because $(\varepsilon \otimes \textrm{Id})\circ \Delta =\textrm{Id}$, and thus the morphism will not change. Because of the same reason also cancelling $2$-$3$-handle pairs can be introduced. Finally, changing isotopies by isotopy, or rearranging the appearance of isotopies in the order up to isotopy, will not change the induced linear morphism in a composition of arrows (consider how components are permuted). Obviously the colimit of the functor on $\mathcal{F}\mathfrak{c}$ is equal to the colimit of the functor on $\mathcal{C}$ and can be described as follows. Take the direct sum $\oplus_{v}V_v$ where $v$ is a vertex of $\mathfrak{c}$ representing the istotopy class of a surface $S\subset M$ and $F(v)=V_v=V^{\otimes \underline{S}}$ is the corresponding $(V^{\otimes \underline{\alpha }})$-module in the category $\mathcal{V}(\alpha )$. 
Then take the quotient by all relations $x'=F(a)x$ where $a: v\rightarrow v'$ is any arrow between $v$ and $v'$ and $x\in V_v$ and $x'\in V_{v'}$.   
The isomorphism with the skein module is now clear from the definitions. Note that a factor $V^{\otimes \underline{S}}$corresponds to the subspace of 
the free module of $V$-colored surfaces that is spanned by letting the colors vary for a fixed isotopy class of a geometric surface.
Thus there is a well-defined linear map from the free module modulo component-multilinearity to $\colim F_V$, which is onto. But the skein relations map trivially by definition. Conversely to a vector $v$ in $V^{\otimes \underline{S}}$ assign the isotopy class of the corresponding surface $S$ colored by the vector $v$ if 
$v$ is of the form $v_1\otimes \ldots \ldots v_{|S|}$, ordered in some way corresponding to an ordering of the surface components. Otherwise decompose the vector into vectors of this form and take the corresponding $R$-linear combination. 
Note that if there is an arrow from a vertex to itself permuting the components of a surface then corresponding $V$-colored surfaces are isotopic as $V$-colored surfaces. Thus vectors identified in this way in $V^{\underline{S}}$ will map to the same $V$-colored surface. The other arrows precisely correspond to the skein relations.
This defines a homomorphism of $R$-modules $\colim F\rightarrow \mathcal{W}(M,\alpha )$, which is inverse to the homomorphism above.
It is clear from the construction that the colimit actually is a module in the category $\overline{\mathcal{V}(\alpha )}$, and the module structure coincides with the 
one defined directly above. Moreover the $V^{\otimes \underline{\alpha }}$-module structure defined above obviously coincides with the one induced from the the colimit construction. 
\end{proof}
 
\vskip .05in

It should be noted that the functor on the object level is completely determined by $\varphi : \underline{\alpha }\rightarrow \underline{S}$, it only recognizes $\pi_0(S,\alpha )$. The functor recognizes the \textit{compressible genus} of the surface $S$ on the morphism level. In fact, if a sequence of Bar-Natan relation applies reducing the genus then 
there is a corresponding morphism $S\rightarrow S'$. Then the change in $\pi_0$ is seen by the morphisms induced from $\Delta $ while the change in genus is seen by the morphisms defined from the handle operator. More about this in section \ref{sec=bicategory}.  

\vskip .05in

Let $i: (M,\alpha )\rightarrow (N,i(\alpha ))$ be an embedding. Recall the functor $\mathcal{C}(i): \mathcal{C}(M,\alpha )\rightarrow \mathcal{C}(N,i(\alpha ) )$ defined in section 1: 
$S\mapsto g(S)$ and any compression bordism $W: S\rightarrow S'$ with $W\subset M\times I$ 
is mapped to $(g\times \textrm{Id})(W)$ by the functor. Because $i$ induces a bijection
$\underline {\alpha }\rightarrow \underline{i(\alpha )}$ there exists a natural functor $\mathcal{V}(g): \mathcal{V}(\alpha )\rightarrow \mathcal{V}(i(\alpha ))$ such that the diagram:
$$
\begin{CD}
\mathcal{C}(M,\alpha )@>\mathcal{C}(i)>>\mathcal{C}(N,i(\alpha ))\\
@V{F_V}VV @V{F_V}VV \\
\mathcal{V}(\alpha )@>\mathcal{V}(i)>>\mathcal{V}(i(\alpha ))
\end{CD}
$$
commutes.
Because of the naturality of the colimit there exists an induced module homomorphism
$\mathcal{W}(i): \mathcal{W}(M,\alpha )\rightarrow \mathcal{W}(N,i(\alpha ))$ respecting the module structures, i.\ e.\ 
for $w\in V^{\otimes \underline{\alpha }}$ and $x\in \mathcal{W}(M,\alpha )$ we have
$$\mathcal{W}(i)(w\cdot x)=\mathcal{V}(i)(w)\cdot \mathcal{W}(i)(x).$$ 
If $i$ is a diffeomorphism then the induced module homomorphism is of course an isomorphism.

\vskip .1in

In section \ref{sec=compression} we defined the two categories $\mathcal{U}[n]$ and $\mathcal{C}[n]$ for nonnegative integers $n$.
Recall that we fix a collection $\alpha =\cup_{i=1}^n\alpha_i$ of $n$ circles in the definition.
The notation for the corresponding category is $\mathcal{U}(\alpha )$ rspectively $\mathcal{C}(\alpha )$.
The Bar-Natan module corresponding to $\mathcal{U}[n]$ generalizes the skein modules of non-embedded
orientable dotted surfaces originally defined by Bar-Natan. It has been shown in \cite{K1} that generalized Bar-Natan modules
$\mathcal{W}[n]:=\colim F_V$, 
where $F_V$ is the obvious functor defined on $\mathcal{C}[n]$, are isomorphic to
the standard one-dimensional $V^{\otimes n}$-module $V^{\otimes n}$. In particular the Bar-Natan module of non-embedded closed orientable surfaces is always isomorphic to the ground ring $R$. 
For each $n\geq 0$, $\mathcal{W}[n]$ is an $R$-algebra with multiplication defined by representing an orientable surface bounding $\alpha $ by discs colored by elements of $V$ and multiplying the colors on components.
There also are natural homomorphisms:
$$\mathcal{W}[n]\otimes \mathcal{W}[m]\rightarrow \mathcal{W}[n+m]$$ 
defined by disjoint union.

\vskip .05in
 
Note that forgetting the embedding of surfaces into $M$, defines the homomorphism of $V^{\otimes \underline{\alpha }}$-modules
$$\mathcal{W}(M,\alpha )\rightarrow \mathcal{W}(\underline{\alpha }).$$
The categories $\mathcal{W}(\underline{\alpha })$ are the colimit modules $\colim F_V$ for $F_V$ the functor defined on $\mathcal{C}[n]$. 

\vskip 0.05in

The $V^{\otimes n}$-modules $\mathcal{W}_-[n]$ of possibly non-orientable non-embedded surfaces are equally important. 
First consider $n=0$. Given an arbitrary surface it is easy to see that by doing neck-cuttings and sphere compressions 
the surface can be reduced to a disjoint union of real projective planes $P^2$. This follows from the fact that the connected sum of three $P^2$'s is the connected sum of $P^2$ and a torus. Now given a connected sum $\sharp_3P^2$, neck-cutting can lead to a disjoint union of three $P^2$, or to a disjoint union of a torus and $P^2$, which can be compressed to $P^2$. 
We can in fact identify $\mathcal{W}_-:=\mathcal{W}_-[0]$ with the $R$-module defined by the quotient of symmetric algebra $SV$ on $V$(note that there always exist diffeomorphisms switching the order of real projective space components) by the ideal generated by $\Delta(v)-(\varepsilon \circ \mathfrak{k})(v)$ for all $v\in V$. 
The relation is coming from compressing the Klein bottle $P^2\sharp P^2$ in two different ways. In fact we can do a separating neck-cutting to get a disjoint union of two $P^2$'s or we can do a non-separating neck-cutting giving a $2$-sphere.
Let $(SV)_-$ be the quotient of $SV$ by the ideal generated by these relations. There is an obvious homomorphism $(SV)_-\rightarrow \mathcal{W}_-$ induced by the homomorphism $SV\rightarrow \mathcal{W}_-$ assigning to the element of the symmetric algebra $[v_1\otimes v_2\otimes \ldots \otimes v_r]$ (we consider the symmetric algebra as quotient of the tensor algebra) the disjoint union of $r$ real projective planes colored by the elements $v_i$. This homomorphism is onto. It is not hard to see that application of neck-cutting and sphere relation defines a homomorphism $\mathcal{W}_-\rightarrow SV_-$ (use the normal form for each component given as $\sharp_i T$ respectively $\sharp_iT^i\sharp P^2$), and the composition of the two homomorphisms is the identity. In fact note that on each generator we can reduce the surface to an incompressible one, i.\ e.\  a disjoint union of $P^2$'s by just doing non-separating neck-cuttings and sphere compressions. We just have to show that this is compatible with Bar-Natan relations. It certainly is compatible with sphere compressions and non-separating neck-cuttings. Consider a separating neck-cutting. Then apply the procedure to each of the resulting colored surfaces. If the neck-cutting curve separates an orientable component or splits an orientable component from a non-orientable component the expansion will not change. Assume that the curve separates a non-orientable component into two non-orientable components. Now expand as before. Each of the components will reduce to a $P^2$ so the connected sum of the two reduces to the Klein bottle $P^2\sharp P^2$. In this case the expansions are the same module the above relation.  

\vskip .05in

Now consider $n>0$. Applying neck-cuttings, the surface can be reduced to closed components and components bounding disks. Moreover we can choose a fixed ordering of the components $\alpha_i$ of the boundary for the isomorphism.
We then have:

\begin{prop} The epimorphism
$$(SV)_-\otimes V^{\otimes n}\rightarrow \mathcal{W}_-[n]$$
\textit{is an isomorphism.}
\end{prop}

\begin{proof}  The epimorphism is defined by mapping $w\otimes v_1\otimes \ldots \otimes v_n$ to the disjoint union of some closed surface determined by $w$ and the disk bounding $\alpha_i$ colored by $v_i$ for $i=1,\ldots r$. It is easy to see that this is an isomorphism.
In fact a homomorphism can be defined $\mathcal{W}_-[n]\rightarrow (SV)_-\otimes V^{\otimes n}$ by cutting along 
the parallel of each $\alpha_i$ defined by the collar of $\alpha_i$ in the bounding surface and then identifying the closed colored surfaces with elements of $(SV)_-$. 
\end{proof}

\begin{rem} (Orientations and Bar-Natan skein theory): For us it seemed natural to work in an orientable but unoriented setting. Of course it is possible to construct theories based on oriented surfaces (Bar-Natan relations preserve orientations). The main idea of Bar-Natan theory is to use compressions to simplify the objects of a topological category. The Bar-Natan relations are induced from orientable topological quantum field theory. In the case of orientable $2$-dimensional topological quantum field theory the topological objects are completely compressible. This makes a setting of orientable surfaces \textit{natural}, even in the embedded case. 
Often one can automatically work in the orientable setting like for closed surfaces embedded in $F\times I$ for $F$ an orientable surface. 
Care has to be taken when glueing because $3$-manifolds and surfaces could no longer be orientable after glueing. Throughout we will restrict the composition of morphisms correspondingly, if not stated otherwise.
In \cite{TT} Turaev and Turner study unoriented topological quantum field theory. 
In the $2$-dimensional case topological morphisms exist, which are Moebius bands bounding circles, and corresponding algebraic morphisms are introduced. There is corresponding topological compression relation, essentially because $P^2$ does not bound. But a disjoint union of two $P^2$'s bounds $P^2\times [0,1]$. One can develop a theory by adding a relation to the theory considering the disjoint union of two real projective planes compressible. This seems natural on one hand because of the following idea about compressions of Bar-Natan theory: Replacing an annulus by a disjoint union of two disks, or replacing a $2$-sphere by the empty surface, in each case the Euler characteristic decreases by $2$. In the same way, replacing a
disjoint union of two $P^2$'s by the empty surface also dcreases the Euler characteristic by $2$.
On the other hand the zero-bordism of the disjoint union of two $P^2$'s has the handle structure of attaching a $1$-handle giving the connected sum $P^2\sharp P^2$ followed by a $2$-handle (neck-cutting) and followed by a sphere compression. Thus from considering the \textit{handle structure} only the usual Bar-Natan relations seem more natural.   
In the framework of Turaev and Turner's unoriented topological quantum field theories the Frobenius algebra is extended by a suitable element $\theta $ satisfying various compatibilities with the Frobenius algebra structures.
The disjoint union of two $P^2$'s thus should correspond the \textit{extended Bar-Natan relation}, which deletes a 
disjoint union of two $P^2$'s by multiplication by $(\varepsilon \theta )^2\in R$. This easily extends to the embedded case by requiring $P^2\times I$ to be embedded. If we consider the module $\mathcal{W}^{ext.}$ defined by extended Bar-Natan relations the module $(SV)_-$ will be replaced by $V$ only.     
\end{rem}

\begin{exmp}  Consider a nontrivial knot $K$ in $S^3$. Let $M=S^3\setminus N(K)$, where $N(K)$ is an open tubular neighborhood. Let $\alpha =\emptyset$.  Consider the morphisms from a boundary parallel torus $S$ to the empty surface $\emptyset$. Then $F(S)=V$ and $F(\emptyset )=R$. But there are \textit{no} morphisms at all in $\mathcal{C}(M)$ from $S$ to $\emptyset$. In fact, because $S$ is incompressible the only morphisms from $S$ are trivial neck-cuttings leading to the disjoint union of $S$ and $2$-spheres. The $2$-spheres can be capped off by $3$-balls but this leads back to the original surface. 
So the only way to \textit{relate} $S$ with the empty surface is by first attaching $1$-handles until we get a surface $S'$ which is \textit{fully compressible} in the sense that it bounds a handlebody in $M\times I$. 
There exists a surface $S'$, which we even can assume connected, and two morphisms $S'\rightarrow S$ and $S'\rightarrow \emptyset$. Note that $F(S')\rightarrow F(S)$ is not invertible except in trivial cases. This is why the necessity of \textit{tunneling}  is picked up by the Bar-Natan module in this case, see \cite{K1} and \ref{sec=tunneling}. Note that in a sequence of $1$-handle and $2$-handle attachments we can always arrange that the $1$-handles precede the $2$-handles. In fact suppose that a $1$-handle follows a $2$-handle. The $2$-handle will replace an annulus by two disks. Then first move the attaching disks of the $1$-handle away from the two disks. Then perturb the $1$-handle to avoid the cocore of the $2$-handle. Now we can change the order up to isotopy. 
It is also standard to see that in each sequence of $1$-handle respectively $2$-handle attachments we can assume that the attaching disks of the $1$-handles respectively attaching annuli of the $2$-handles are disjoint. 
But there is a quite subtle difference between bounding a handlebody in $M$ respectively $M\times I$. (Construct a genus-$2$ handlebody embedded in $S^3$ knotted in such a way that it does not bound a handlebody on either side). If a connected surface $S$ is fully compressible we call $t(S)$ the minimal genus of a handlebody embedded in $M\times I$ bounding $S\subset M\times \{0\}$.
We call $t(S)$ the \textit{compression number} of $S$. In the situation above $t(S')\leq t(K)$ where $t(K)$ is the tunnel number of the knot. 
Note that all surfaces in this example are closed and embedded in $S^3$ and thus are orientable.
\end{exmp}

In general if $S\subset M$ is a closed separating surface in $M$ such then there is defined a nonnegative 
number $u(S)$, which is the minimal number
of $1$-handles one has to attach to $S$ such that the resulting surface bounds a disjoint union of handlebodies in $M\times I$. In the example above there 
are \textit{minimal} situations $W: S'\rightarrow S$ and $W': S'\rightarrow \emptyset$ where $F(W): V\rightarrow V$ is 
$\mathfrak{k}^{u(S)}$ and $F(W')=\varepsilon \circ \mathfrak{k}^{t(S')}$ with $t(S')=u(S)+1$ with $1$ being the genus of the boundary torus. 

\begin{exmp} \textbf{(a)} Let $\alpha =\emptyset$ and consider the functor on the set of morphisms $S\rightarrow \emptyset$ for a closed orientable surface $S$ embedded in a $3$-manifold $M$. Now consider a morphism $W$from $S$ to $\emptyset$, which we assume is a disjoint union of handlebodies in $M\times I$ bounding the components of $S\subset M\times \{0\}$. We should point out that it is not at all clear that such a morphism exists.  
Now the functor $F_V$ assigns to the morphism $W$ 
as above a linear homomorphism $V^{\otimes \underline{S}}\rightarrow R$ as follows. (Note that this morphism is completely determined by $S$.)  Suppose $|S|=k$ and choose an order of the components of $S$. Let $g_i$ denote the genus of the $i$-th component for $1\leq i\leq k$. Then $F(W)$ is the morphism which maps $v_1\otimes \ldots \otimes v_k$
to $\prod_{1\leq i \leq  k}\varepsilon \circ \mathfrak{k}^{g_i}$, where the product indicates application of the product map $R\otimes \ldots \otimes R\rightarrow R$.   

\vskip 0.05in

\noindent \textbf{(b):} Recall that attaching a $2$-handle at most increases the number of components of a surface. Let $S\subset M$ be a connected surface 
and let $S'\subset M$ be a surface with $|S'|=k$ and an ordering of the components of $S'$ chosen. 
(Note that if $S$ is orientable then $S'$ is orientable.)
We consider a connected morphism $W: S\rightarrow S'$ with only index $2$ critical points. Then $F(W): V\rightarrow V^{\otimes k}$ 
is defined as follows:  
Let $r$ be the difference between the sum of the genera of the components of $S'$ and the genus of $S$.
(Define the genus of a connected closed surface $S$ to be $g$ if $S$ is diffeomorphic to $P^2\sharp T^g$ or $T^g$, where $T$ is the torus surface. Define the genus of the Klein bottle to be $1$.) Then $F(W)=d_k\circ \mathfrak{k}^r$. 
This follows easily from properties of the coproduct and counit in a commutative Frobenius algebra. In fact recall that $\Delta $ is a map of algebras, where we can consider $V\otimes V$ both as a left or right $V$-algebra. 
Thus the applications of $\mathfrak{k}$ can be arranged first (recall the corresponding identity in weak Frobenius algebras). Also by iterated applications of the coassociativity property of $\Delta $: $(\textrm{Id}\otimes \Delta)\circ \Delta =(\Delta \otimes \textrm{Id})\circ \Delta$ the remaining homomorphisms have the form $d_k$.  More generally consider a morphism $W: S\rightarrow S'$
with only $2$-handles attached. Then each component of $S$ gives rise to a component of $W$ and determines a set of components of $S'$. 
This decomposition of the set of components of $S'$ is called a \textit{partition} of $S'$.
The induced homomorphism $V^{\otimes \underline{S}}\rightarrow V^{\otimes \underline{S}'}$ now is defined from the components of $W$ respectively $S$ as before using the 
tensor product structure. Note that a general morphism in $\mathcal{C}(M,\alpha )$ can be replaced by a morphism with the $2$-handles preceding the $3$-handles. 
(If also $3$-handles are attached the induced morphism $F(W)$ will be the linear morphism defined from the $2$-handles composed with morphisms of the form $\textrm{Id}\otimes \varepsilon \otimes \textrm{Id}$
with $\varepsilon : V\rightarrow R$ corresponding to $2$-sphere components bounding the $3$-handles which are attached.)  
Note that the homomorphism $F(W)$ only depends only on the partition of $S'$ corresponding to the components of $W$, thus it depends on the morphism $W$ in $\mathcal{C}(M,\alpha )$ in a very weak way. 
 
\vskip .05in
 
\noindent \textbf{(c):} The arguments in \textbf{(a)} and \textbf{(b)} above also apply in the case $\alpha \neq \emptyset$. Just define the genus of a surface with boundary 
by the genus of the surface, which results by capping off all boundary components by disks. The description above still holds because the manifolds $W$ are
the products $\alpha \times [0,1]$ in the intersection with $\partial M\times [0,1]$. 

\vskip .05in

\noindent \textbf{(d):} Consider $M=S^1\times D^2$ and $\alpha $ a collection of $2n$ longitudes with some orientations such that the total homology class in $M$ is trivial. Then each crossingless matching (see \cite{R}) determines  a module $V(\varphi )$ which is realized by a bounding surface. In fact all these surfaces are incompressible with $n$ components. The crossingless matching is precisely the mapping $\underline{\alpha }\rightarrow \underline{S}$, where two components of $\alpha $ are matched when they map to the same element of $\underline{S}$. The structural maps 
$\underline{\alpha }\rightarrow \underline{S}$ just can be considered as generalizations of the crossingless matchings.
Note that by forming possible connected sums we find surfaces with modules $V(\varphi ')$ with smaller image. But homomorphisms in the category $\mathcal{V}(\alpha )$ can be applied mapping $V(\varphi ')$ to $V(\varphi )$.     
Note that closed components can enlarge the image but because those components are completely compressible in $M$, but again there will be morphisms from the corresponding modules to a module $V(\varphi )$ defined from a crossingless matching.  
\end{exmp}

\vskip .1in

Next we will briefly discuss the structure of the bordism categories considered up to this point and their relation to each other.
Of course, a presentation of the full category of $3$-manifold triads embedded in $M\times [-1,1]$ can be constructed by first adding the reverse arrows
to the directed graph corresponding to elementary critical points of index $0$ and $1$ bordisms. The we have to add relations, first corresponding to reverse relations between $0$ and $1$-handles. Finally we need morphisms taking into account the deletion and insertion of $1$/$2$-handle pairs, and of course possibly changing arbitrarily the order of handles. 

It is easy to see that the only invertible morphisms in both categories $\mathcal{C}$ or $\mathcal{M}$ are the isotopy morphisms. The elementary morphisms admitting left inverses are those corresponding to $k$-handles, which admit cancelling $k+1$-handles, $k=0,1,2$. For a $0$-handle this means that that it is attached to a nonempty manifold; for a $1$-handle it means that the handle is trivial, i.\ e.\ the parallel of its core in the boundary of the codomain surface bounds a disk in $M$ intersecting the codomain surface only in the curve;
for a $2$-handle this means that the attaching curve is inessential on the domain surface and the $2$-sphere formed by a disk on the surface and the attaching disk is a $2$-sphere bounded by a $3$-ball in $M$ intersecting the codomain surface only in the $2$-sphere.  
Consider the quotient category $\mathcal{H}$ of $\mathcal{M}$, which is defined by introducing the equivalence relation on morphisms that two morphisms are equivalent if the corresponding $3$-dimensional submanifolds (with corners) of $M\times I$ are bordant relative to the boundary in $(M\times I)\times I$. It is not hard to see that each morphism in this category is invertible with the inverse represented by the reverse manifold. (Here is the idea: There is a Pontrjagin map from the set of maps $M\rightarrow \mathbb{RP}^{\infty}$,
mapping the boundary in a fixed way, into the set of submanifolds of $M$ bounding $\alpha $ such that homotopy classes of paths in the space of maps $M\rightarrow \mathbb{RP}^{\infty}$ correspond to bordism classes of submanifolds in $M\times I$. But for each path there is the reverse path and composition 
of both paths is homotopic to the constant path.)

\vskip .05in

\noindent \textbf{Conjecture.} The category $\mathcal{H}$ is equivalent to the category of fractions of the category $\mathcal{C}$, i.\ e.\ the localization at the set of all morphisms, see \cite{W}. 

\vskip .05in

The set of isomorphism classes of this category is in one-to-one correspondence with the set $h_2(M,\alpha ;\mathbb{Z}_2)$ defined by $\partial^{-1}[\alpha ]\subset H_2(M,\alpha ;\mathbb{Z}_2)$, where $\partial :H_2(M,\alpha ;\mathbb{Z}_2)\rightarrow H_1(\alpha ;\mathbb{Z}_2)$ is the boundary operator of the exact sequence of the pair and $[\alpha ]$ is the mod $2$ fundamental class of $\alpha $. If $\alpha =\emptyset$ then $h_2(M,\alpha ;\mathbb{Z}_2)=H_2(M;\mathbb{Z}_2)$. There are the obvious forget functors $\mathcal{M}\rightarrow \mathcal{H}$ and $\mathcal{C}\rightarrow \mathcal{H}$. Note that the set of components of $\mathcal{C}(M,\alpha )$, see \cite{M} page 88 for the definition, is in one-to-one correspondence with the set $h_2(M,\alpha ;\mathbb{Z}_2)$. Correspondingly the Bar-Natan skein modules are graded over $h_2(M,\alpha ;\mathbb{Z}_2)$. 

\vskip .05in

\begin{rem}  \textbf{(a):} The Khovanov construction on thickened surfaces naturally leads to non-orientable surfaces as first observed by Manturov \cite{Ma}.  
 
\noindent \textbf{(b):} We have seen in Example \ref{example=Khovanov} that Poincare duality (cap product with the fundamental class) induces an $R$-isomorphism of Frobenius algebras defined by cohomology respectively homology.
It has been proved in \cite{K1} that this isomorphism $\mathcal{P}$ induces an isomorphism $\mathcal{P}_*$ of Bar-Natan modules
$V:=H^*(Y)\rightarrow H_*(Y)=:V'$.
It is easy to see that it maps for given $(M,\alpha )$ the $(V^{\otimes \underline{\alpha }})$-module
$\mathcal{W}(M,\alpha ;V)$ isomorphically onto the the $(V'^{\otimes \underline{\alpha }})$-module $\mathcal{W}(M,\alpha ,V')$ compatible with the module actions, i.\ e.\ $\mathcal{P}_*(w\cdot x)=\mathcal{P}(w)\cdot \mathcal{P}_*(x)$ for $w\in \mathcal{V}^{\otimes \underline{\alpha }}$ and $x\in \mathcal{W}(M,\alpha ;V)$. 
\end{rem}

\vskip .1in

Note that given any pairs $(M,\alpha )$ and $(N,\beta )$ disjoint manifolds we can form their \textit{disjoint} union $(M\cup N, \alpha \cup \beta )$.
The following result is immediate from the definitions and relates the \textit{monoidal} structures on the topological side and algebraic side.

\vskip .05in

\begin{prop} There is a natural isomorphism of categories
$$\mathcal{C}(N,\beta )\times \mathcal{C}(M,\alpha )\rightarrow \mathcal{C}(N\cup M,\beta \cup \alpha )$$
such that the following diagram of functors commutes:
$$
\begin{CD}
\mathcal{C}(N, \beta )\times \mathcal{C}(M,\alpha )@>>>\mathcal{C}(N\cup M,\beta \cup \alpha) \\
@V{F_V\times F_V }VV @V{F_V }VV \\
\mathcal{V}(\beta )\times \mathcal{V}(\alpha ) @>\otimes >>\mathcal{V}(\beta \cup \alpha)
\end{CD}
$$ 
In particular we also have 
$$\mathcal{W}(N\cup M,\beta \cup \alpha)\cong \mathcal{W}(N,\beta)\otimes \mathcal{W}(M,\alpha ).$$
\end{prop}

\section{$2$-categories in Bar-Natan skein theory}\label{sec=bicategory}

We discuss the categorical set-up from the previous sections following the standard higher category and extended topological quantum field theory framework, see \cite{L}. 
For definitions of $2$-categories and bicategories (weak $2$-categories) we refer to \cite{Bo}.
Recall that a $2$-category is a category whose morphism sets are the sets of objects of a small category, or a category enriched over the set of categories. In particular the composition of the category is enriched by making the composition operations become functors. The morphisms between morphisms, i.\ e.\ the morphisms of the morphism categories, are called \textit{$2$-morphisms}.
In a bicategory the composition functors are not necessarily strictly associative with the associativity replaced by a weak version involving isomorphisms satisfying coherence conditions. Thus each $2$-category is a bicategory with trivial associativity isomorphisms. 

\vskip .05in

We begin by first constructing a combinatorial category, which models the category of $1$-manifolds bounding $2$-manifolds, up to orientability. It will then be shown that there exists a natural $2$-categorical extension based on the idea of \textit{compression}. The combinatorial category is related with ideas of Kerler \cite{Ker}, who constructed a corresponding graph category. This part is not necessary for our constructions but provides a clear intermediate step of the passage from the geometric bicategory to the algebraic $2$-category. It also defines combinatorial morphisms precisely corresponding to the Frobenius algebra morphisms constructed in section \ref{sec=algebra}.  
Measuring $2$-morphisms using appropiate linear functors is the basic idea of the Bar-Natan abstract skein theory. At the same time the $2$-categorical set-up allows an immediate way of generalizing to a codimension-one embedded situation. We also discuss how the Bar-Natan approach allows skein calculation of $2$-dimensional topological quantum field theory functors, see \cite{K1}.

\vskip .05in

Here is the description how the geometric category becomes combinatorial.
Firstly, assign to each $1$-manifold $\alpha $ its set of components $\underline{\alpha }$. Then assign to each surface $S$ bounding 
$\alpha $ the map $\varphi: \underline{\alpha }\rightarrow \underline{S}$. Finally assign to each component of $S$ its genus, i.\ e.\ the genus of the closed orientable surface that results by capping off all boundary circles, thus
a map $\psi :\underline{S}\rightarrow \mathbb{N}=\{0,1,2,\ldots \}$.
Obviously this data completely determines the surface up to orientability. If we distinguish ingoing and outgoing circles, i.\ e.\ we have 
$\partial S=\alpha_2\cup \alpha_1$ a disjoint union then assign the maps $\varphi_i: \underline{\alpha}_i\rightarrow \underline{S}$ for $i=1,2$. Next we describe the glueing of surfaces in the combinatorial model. Let 
$\phi '=(\varphi_1,\varphi_2,\psi ')$ be a morphism $J_1\rightarrow J_2$, i.\ e.\ $\varphi_i: J_i\rightarrow K'$ for a finite set $K'$. Let $\phi ''=(\varphi_2',\varphi_3,\psi '')$ be a morphism $J_2\rightarrow J_3$, i.\ e.\ $\varphi_2': J_2\rightarrow K''$ and $\varphi_3: J_3\rightarrow K''$ for a finite set $K''$. Then we define the morphism $\phi :=\phi ''\circ \phi '$ from $J_1$ to $J_3$. Let $\phi =(\varphi_1',\varphi_3',\psi ')$. 
Firstly define the finite set $K$ to be the quotient of the disjoint union
$K'\cup K''$ (if the sets are not disjoint consider $K'\times \{0\}$ and $K''\times \{1\}$) by the equivalence relation $\varphi_2(j)=\varphi_2'(j)$ for all $j\in J_2$. Then define $\varphi_i'$ for $i=1$ respectively $i=3$ by the 
composition of the inclusions of $K'$ respectively $K''$ into $K$ with the natural projection  $K'\cup K''\rightarrow K$. Next for each $k\in K$ consider the preimages $\{k_1',\ldots k_r'\}$ respectively $\{k_1'',\ldots ,k_{\ell}''\}$ in $K'$ respectively $K''$ under the above maps. Define $\psi (k)=\psi'(k_1)+\ldots +\psi '(k_r)+\psi ''(k_1'')+\ldots +\psi ''(k_{\ell}'')$. Note that this models precisely how the genera of components that are glued together add up.
The closed components of a surface correspond to elements in $K=\underline{S}$, which are not in $\varphi_1(\underline{\alpha_1})\cup \varphi_2(\underline{\alpha }_2)$.

\vskip .05in

The idea will be to study sequences of compressions of $2$-manifolds and to measure the compressions. The point is 
that this can easily be generalized to $3$-manifolds, with the compressions taking place in the $3$-manifold.
The right combinatorial set-up is that of a $2$-category. 
Let us first describe this for our combinatorial model. 

\vskip .05in

Let $\mathcal{F}_2$ denote the $2$-category with objects finite sets $J$ and defined as follows.
Let $\textrm{Hom}(J_1,J_2)$, the collection of morphisms $\phi: J_1\rightarrow J_2$, be 
the collection of triples $(\varphi_1,\varphi_2,\psi )$ as defined above. We defined the composition of morphisms
$$c(J_1,J_2,J_3): \textrm{Hom}(J_2,J_3)\times \textrm{Hom}(J_1,J_2)\rightarrow \textrm{Hom}(J_1,J_3)$$ as
above. For each finite set $J$, $\textrm{Hom}(J,J)$ contains the identity morphism
$(\varphi_1,\varphi_2,\psi )$ defined by $\varphi_1: J\times \{0\}\rightarrow J$ and 
$\varphi_2: J\times \{1\}\rightarrow J$ are defined by the natural identifications, and $\psi (j)=0$ for each
$j\in J$ (this corresponds to the image of a cylinder). 
We now extend the sets $\textrm{Hom}(J_1,J_2)$ to categories such that composition becomes a functor. 
For this we need to define a set of morphisms and a composition operation $\phi \mapsto \phi '$
for each set $\textrm{Hom}(J_1,J_2)$. This is done by first defining four types of \textit{elementary homomorphisms}, which  together with the identity morphism generate the corresponding  
set of $2$-morphisms by compositions of elementary morphisms, where associativity is imposed.  
The actual collections of morphisms then are defined by imposing relations on the collection of morphisms generated by compositions of the elementary ones. 
\textbf{(i):} Let $\varphi_i :J_i\rightarrow K$ for $i=1,2$ and $\psi : K\rightarrow \mathbb{N}$. 
This defines $\phi :=(\varphi_1,\varphi_2,\psi)$.
Let $k\in K$ and $\varphi =\varphi_1\cup \varphi_2: J_1\cup J_2\rightarrow K$, where we assume $J_1,J_2$ disjoint. Let $\varphi^{-1}(\{k\})=I_1\cup I_2$ be a disjoint union. Let $K'=(K\setminus \{k\})\cup \{k_1,k_2\}$ with $k_i\notin K\setminus \{k\}$ for $i=1,2$. Define $\varphi ': J_1\cup J_2\rightarrow K'$ by
by $\varphi $ on $(J_1\cup J_2\setminus \varphi^{-1}(\{k\})$ and map $I_i$ to $k_i$ for $i=1,2$. 
Then define $\varphi_i '$ by $\varphi '|J_i$ for $i=1,2$.
Finally define $\psi ': K'\rightarrow \mathbb{N}$ by $\psi $ on $K\setminus \{k\}$ and by $\psi '(k_1)=g_1$ and $\psi '(k_2)=g_2$ where $g_1,g_2\in \mathbb{N}$ such that $\psi (k)=g_1+g_2$. Then $\phi '=(\varphi_1',\varphi_2,\psi ')$
is an element of $\textrm{Hom}(J_1,J_2)$ and $\phi \mapsto \phi '$ defines an elementary morphism
denoted $\Delta_k$. Note that $\Delta_k$ is only defined on $\phi $ with $\phi_i$ mapping into $K$ and $\psi $ defined on $K$, with $k\in K$. 
\textbf{(ii):} Define an elementary morphism by mapping $(\varphi_1,\varphi_2,\psi )$ with $\psi (k)=g>0$ for some $k\in K$ to $(\varphi_1,\varphi_2,\psi ')$ where $\psi '$ maps by $\psi $ for all $k'\neq k$ and 
$\psi '(k)=\psi (k)-1$. Call this morphism $\mathfrak{k}_k$
\textbf{(iii):} Suppose that $\psi =(\varphi_1,\varphi_2,\psi )$ is such that for some $k\notin \varphi_1(J_1)\cup \varphi_2(J_2)$ we have $\psi (k)=0$. Then define $\psi '=(\varphi_1',\varphi_2',\psi ')$ by the same maps but change the codomain to $K\setminus \{k\}$. We call this morphism $\varepsilon_k$.
\textbf{(iv):} Finally for $\phi =(\varphi_1,\varphi_2,\psi )$ and $k_1,k_2\notin \varphi_i(J_i)$ for $i=1,2$ but 
$\psi (k_1)=\psi (k_2)$ we introduce a \textit{permutation morphism} $\pi_{k_1,k_2}$, which maps $\phi $ to itself.
The relations we impose on the composition of morphisms are the following commutativity relations, whenever composition is defined:
$$\mathfrak{k}_{k_2}\circ \Delta_k=\mathfrak{k}_{k_1}\circ \Delta_k=\Delta_k\circ \mathfrak{k}_k$$
and $\Delta_k$ commutes with $\mathfrak{k}_{k'}$ in all cases where $k$ and $k'$ are not related as described in \textbf{(i)}. Next we want 
$$\varepsilon_{k_1}\circ \Delta_k=\textrm{Id} \quad \textrm{respectively} \quad \varepsilon_{k_2}\circ \Delta_k=\textrm{Id},$$
whenever those compositions are defined. 
Finally we want the permutation morphisms to commute with all the above morphisms whenever compositions are defined.   

\vskip .1in  
  
Now we extend the composition operations $c(J_1,J_2,J_3)$ to functors of categories.
Given a morphism $\lambda_1: \phi_1\mapsto \phi_1'$ of the category $\textrm{Hom}(J_1,J_2)$ and a morphism
$\lambda_2: \phi_2\mapsto \phi_2'$ of the category $\textrm{Hom}(J_2,J_3)$ we have to define a composition being a 
morphism of the category $\textrm{Hom}(J_1,J_3)$. It suffices to do this for elementary morphisms and one of the two morphisms being the identity morphisms by using
$\lambda_2\circ \lambda_1:=(\lambda_2\circ \textrm{Id})\circ (\textrm{Id}\circ \lambda_1)$, where the middle composition is the composition in $\textrm{Hom}(J_1,J_3)$ defined above.
We consider one case, the other cases are discussed similarly. (It is almost easier at this point to use that the combinatorial model precisely corresponds to how the glueing of surfaces relates with Bar-Natan relations in order to understand the definition).
Let $\Delta_k: \phi \rightarrow \phi'$ and $\textrm{Id}: \widetilde{\phi }\rightarrow \widetilde{\phi }$. Then we want to define $\textrm{Id}\circ \Delta_k: \widetilde{\phi }\circ \phi \rightarrow \widetilde{\phi}\circ \phi '$.
Let $\phi =(\varphi_1,\varphi_2,\psi )$ and $\phi '=(\varphi_1',\varphi_2,\psi ')$ be as above in (i).
Let $\widetilde{\phi }=(\widetilde{\varphi }_2,\varphi_3,\widetilde{\psi })$.
Since we assume that compositions are defined both $\varphi_2, \widetilde{\varphi }_2$ have domain the same set $J_2$. 
The definition will depend on how $\widetilde{\varphi }_2$ is acting on the sets $I_1\cap J_2$ and $I_2\cap J_2$ where 
$I_1\cup I_2=\varphi ^{-1}(k)$ and $\varphi : J_1\cup J_2\rightarrow K$. Recall that $K'=K\setminus \{k_1,k_2\}$.
Assume that $\widetilde{\varphi }_2(j_1)=\widetilde{\varphi }_2(j_2)=\widetilde{k}\in \widetilde{K}$ for some $j_1\in I_1\cap J_2$ and $j_2\in I_2\cap J_2$. Then in forming the composition of $\phi $ and $\widetilde{\phi }$ respectively 
$\phi '$ and $\widetilde{\phi }$ we form the quotient of $K$ and $\widetilde{K}$ respectively $K'$ and $\widetilde{K}$ by identifying $\varphi _2(j)$ with $\widetilde{\varphi }_2(j)$ respectively $\varphi_2'(j)$ and $\widetilde{\varphi }_2(j)$ for all $j\in J_2$. Then $\textrm{Id}\circ \Delta_k$ is defined by $\mathfrak{k}_{k''}$ where $k''$ is the element in the quotient corresponding to $\widetilde{k}$. 
If $\widetilde{\varphi }_2(j_1)\neq \widetilde{\varphi }_2(j_2$ for all $j_i\in I_i\cap J_2$ then let $k_1'',k_2''$ be the images of $k_1,k_2$ in the quotient of $K'\cup \widetilde{K}$ defined by the glueing. In this case 
$\textrm{Id}\circ \Delta_k$ is defined by $\Delta_{k'}$ with $k'$ the image of $k$ in the quotient of $K\cup \widetilde{K}$ defined by the glueing, and the quotient of $K'\cup \widetilde{K}$ corresponding to the glueing is formed from the quotient of $K\cup \widetilde{K}$ by substituting $k'$ by $\{k_1'',k_2''\}$.  
Similarly there are definitions in all other cases.
It is important to notice that the functoriality of the definition depends on the relations described above between the compositions of elementary morphisms in an essential way.
The resulting composition functors are associative and thus $\mathcal{F}_2$ is a $2$-category.    
(Note that our morphism collections $\textrm{Hom}(J_1,J_2)$ are not small categories. But this can easily be adjusted by working in a suitable universe of finite sets. One should start from a model where the finite sets are sets of components of closed $1$-manifolds and orientable $2$-manifolds embedded in a Euclidean space of high dimension.)

\vskip .1in

Next we define, using the constructions in section \ref{sec=algebra}, a linear $2$-category $\mathcal{V}_2$. This $2$-category will be a kind of natural linear representation category of the combinatorial $2$-category defined above. 
The objects of $\mathcal{V}_2$ are $R$-modules $V^{\otimes J}$ for $J$ finite sets. The morphism categories
will be categories  
$\textrm{Hom}(V^{\otimes J_1},V^{\otimes J_2})=\mathcal{V}(J_1,J_2)$ with objects given by $V^{\otimes J_1}\otimes V^{\otimes J_2}$-bimodules 
and morphisms between those modules generated by $\Delta $ ,$\mathfrak{k}$, $\varepsilon $ and permutations as defined in section\ref{sec=algebra}. Theorem \ref{thm=glueing functor} can now be reformulated saying that with the morphisms defined in this way, composition is a functor:
$$\textrm{Hom}(V^{\otimes J_2},V^{\otimes J_3})\times \textrm{Hom}(V^{\otimes J_1},V^{\otimes J_2})\rightarrow \textrm{Hom}(V^{\otimes J_1},V^{\otimes J_3}).$$

\vskip .01in

The following is immediate from the definitions. It shows how the algebraic category defined from the Frobenius algebra 
is just an algebraic image of the combinatorial category above. 

\vskip .1in

\begin{prop} There is the $2$-functor
$$\mathcal{F}_2\rightarrow \mathcal{V}_2,$$
defined by mapping finite sets $J$ to modules $V^{\otimes J}$, 
morphisms $(\varphi_1,\varphi_2,\psi )$ with $\varphi_i: J_i\rightarrow K$ to $V^{\otimes J_1}\otimes V^{\otimes J_2}$-bimodules $V^{\otimes K}$, and mapping the $2$-morphisms according to their names in the obvious way.    
\end{prop}

\vskip .1in

Now we briefly describe the basic \textit{cylindrical compression bicategory} $\mathcal{C}_2^{\textrm{cyl.}}$ of codimension $1$ manifold pairs.
The objects of this category are representing pairs $(F,\alpha )$ where $\alpha $ is a $1$-dimensional submanifold of a closed surface $F$, one pair for each diffeomorphism class of pairs. The morphism categories $\textrm{Hom}((F_1,\alpha_1),(F_2,\alpha_2))$ will have collections of objects defined by $3$-manifolds $M$ with $\partial M=F_1\times \{0\}\cup F_2\times \{1\}$ together with 
representatives of isotopy classes of properly embedded surfaces $S\subset M$ such that $\partial S=\alpha_1\times \{0\} \cup \alpha_2\times \{1\}$. Again this will not define small categories. But it is possible to consider $3$-manifolds embedded in 
sufficiently high dimensional Euclidean spaces instead of abstract $3$-manifolds. We omit some of the technical difficulties involved here. The morphisms of the category $\textrm{Hom}((F_1,\alpha_1),(F_2,\alpha_2))$ are 
essentially Heegaard-isotopy classes relative to the boundary of compression $3$-manifolds in $M\times I$.
Recall that morphisms also are self-isotopies of representative surfaces. 
We can think of each morphism set though as generated by neck-cutting compressions, sphere compressions and 
loop isotopies with relations given by Cerf theory.  
Note that because of the embedding we actually see two levels of $2$-categories. On the ambient one we have the surfaces, $3$-manifolds and cylinders over $3$-manifolds. This is why we call the bicategory cylindrical. In the end will briefly indicate how we can generalize to a broader setting. 
Note that the disjoint union of all categories $\mathcal{C}(M,\alpha )$ with $\partial M=F_1\cup F_2$ is the morphism category $\textrm{Hom}((F_1,\alpha_1),(F_2,\alpha_2))$.  
In order to define a bicategory we need composition functors:
$$\textrm{Hom}((F_2,\alpha_2),(F_3,\alpha_3))\times \textrm{Hom}((F_1,\alpha_1),(F_2,\alpha_2))\rightarrow 
\textrm{Hom}((F_1,\alpha_1),(F_3,\alpha_3))$$
On the level of objects of the $\textrm{Hom}$-categories these are defined by the glueing 
of manifolds $(M;\alpha_1,\alpha_2)$ and $(N;\alpha_2,\alpha_3)$ defining $(M\cup N;\alpha_1,\alpha_3)$, glueing of the embedded surfaces and replacing the glued surfaces by the corresponding standard representatives in their isotopy class.
On the level of morphisms we now have to glue $M\times I$ with $N\times I$ and containing \textit{representatives} of the Heegaard isotopy classes of $3$-manifolds $W\subset M\times I$ and $W'\subset N\times I$. Note that these have to be glued along $\alpha_2\times I$. The technical subtlety is the replacement of the glued surfaces in the boundaries by standard representatives in their isotopy classes. In order to define composition for morphisms we need to specify explicit isotopies of the glued surface embeddings to standard embeddings. This will allow to define composition. 
But the composition is not necessarily associative. There will be a \textit{glueing anomaly} that has to be 
recorded in terms of associativity isomorphisms, which we not discuss in detail at this point..   

\vskip .1in

The next result discusses how the Bar-Natan functor behaves with respect to glueing.
Recall that we identify $\mathcal{V}(\alpha_1,\alpha_2):=\mathcal{V}(\underline{\alpha_1},\underline{\alpha_2})$.

\vskip .1in

\begin{thm} The glueing functor $\mathfrak{g}$ from Theorem \ref{thm=glueing functor} defines a commutative diagram of functors: 
$$
\begin{CD}
\mathcal{C}(N;\alpha_2 ,\alpha_3 )\times \mathcal{C}(M;\alpha ,\beta)@>>> \mathcal{C}(N\cup M;\alpha_1 ,\alpha_3) \\
@V{F_V\otimes F_V}VV @V{F_V}VV \\
\mathcal{V}(\alpha_2 ,\alpha_3 )\times \mathcal{V}(\alpha_1 ,\alpha_2 ) @>\mathfrak{g}>>\mathcal{V}(\alpha_1 ,\alpha_3 )
\end{CD}
$$
\end{thm}

\begin{proof}  It only remains to prove commutativity of the diagram. Because $\mathfrak{g}$ is a functor it suffices to check this on generators
of the free category on the directed graph $\mathfrak{c}(M;\alpha ,\beta)$ respectively $\mathfrak{c}(N;\beta ,\gamma)$. 
But then it follows from the very definitions of the functors. 
\end{proof}

\vskip .1in

\begin{rem} For each fixed orientable surface $F$
there is defined a sub-bicategory $\mathcal{C}_2(F)$ of $\mathcal{C}_2^{\textrm{cyl.}}$.
The collection of objects is defined by representatives of isotopy classes of $1$-manifolds $\alpha \subset F$. 
The objects of the morphism categories $\textrm{hom}(\alpha_1,\alpha_2)$ are the objects of $\textrm{Hom}((F,\alpha_1),(F,\alpha_2))$ but restricted to surfaces $S$ embedded in
$F\times I$ with $\partial S=\alpha_1\times \{0\}\cup \alpha_2\times \{1\}$. Note that the Bar-Natan modules corresponding to
$\alpha_1=\alpha_2$ are algebras in this case.
\end{rem}

\vskip .1in

Our main results until now are summarized in the following theorem.

\vskip .1in

\begin{thm} \label{thm=bifunctor} For each Frobenius algebra $V$ the Bar-Natan functor extends to a bifunctor
$$F: \mathcal{C}_2^{\textrm{cyl.}}\rightarrow \mathcal{V}_2,$$
which factors through a $2$-functor
$$\mathcal{C}_2^{\textrm{cyl.}}\rightarrow \mathcal{F}_2$$
and the canonical bifunctor
$$\mathcal{F}_2\rightarrow \mathcal{V}_2.$$
The functor is monoidal with respect to the obvious monoidal structures on the bicategories $\mathcal{C}_2^{\textrm{cyl.}}$ and $\mathcal{V}_2$. 
\end{thm}

Note that the bifunctor depends on the embeddings only through the domain category. The values taken are in each case determined \textit{only} by the underlying surface respectively $3$-manifold.
              
\vskip .1in

Given a Frobenius algebra $V$ let $T_V$ be the associated $2$-dimensional topological quantum field theory.
In \cite{K1} we discussed how the Bar-Natan skein relations can be understood in terms of the \textit{morphism
kernel} of the $2$-dimensional topological quantum field theory $T_V$ associated to the Frobenius algebra $V$.
In \cite{K1}, Theorem 3.2 and Theorem 5.1 we have shown that $T_V$ extends to surfaces with components colored by elements of the Frobenius algebra, and then to the linearized category defined by taking formal $R$-linear combinations of morphisms with the same source and target. Then the Bar-Natan relations span the kernel, see also 
\cite{BN}. In the following we will analyze the precise relation between $T:=T_V$ and $F:=F_V$.
Moreover we will see that the idea of \textit{coloring} and \textit{linearizing} the morphisms of the category also extends to the cylindrical compression bicategory. This process is a kind of \textit{categorical extension}
to which the functor or bifunctor naturally extends.

\vskip .05in

Firstly note that for a given surface $S$ the value of $F(S)$ in general does not determine the value of $T(S)$.
In fact for a surface $S$ with $\partial S=\alpha_1\times \{0\}\cup \alpha_2\times \{1\}$ let $\widetilde{T}(S)$ 
denote the homomorphism $V^{\otimes \underline{\alpha_1}}\rightarrow V^{\otimes \underline{\alpha_2}}$ defined by 
changing the functor $T$ by letting the handle operator $\mathfrak{k}=m\circ \Delta =\textrm{Id}$. Note that the axioms of a topological quantum field theory will still be defined. But the functor $\widetilde{T}$ does not sense any genus information. Now $F(S)$ determines natural inclusion $j_i: V^{\otimes \underline{\alpha_i}}\rightarrow V^{\otimes \underline{S}}$ for $i=1,2$. The dual homomorphisms, in the sense of linear algebra duality and by identifying $\textrm{Hom}(V,R)$ canonically with $V$ using the Frobenius pairing, defines homomorphisms 
$j_i^*: V^{\otimes \underline{S}}\rightarrow V^{\otimes \underline{\alpha_i}}$ for $i=1,2$. 

\begin{prop} For each Frobenius algebra $V$,
$$\widetilde{T}_V(S)=j_2^*\circ j_1$$
\end{prop}

\begin{proof}This can first be seen for a connected surface $S$ because $m$ and $\Delta $ are dual to each other, and $j_1$ thus $j_1$ is induced by products while $j_2^*$ is induced by coproducts.
In fact, in a suitable normal form, and after ordering the components of $\alpha_1,\alpha_2$ suitably, the homomorphism
$T(S)$ has the form
$$(\Delta \otimes \textrm{Id}^{\otimes (|\alpha_2 |-1)})\circ (\Delta \otimes \textrm{Id}^{\otimes (|\alpha_2|-2)}) 
\circ \ldots \circ (\Delta \otimes \textrm{Id})\circ \Delta \circ m\circ (m\otimes \textrm{Id} )\circ \ldots $$
$$\ldots \circ (m \otimes \textrm{Id}^{\otimes (|\alpha_1|-1)})$$
The general result now follows immediately.
\end{proof}

\vskip .05in

Next we extend the category $\mathcal{C}_2^{\textrm{cyl}}$ to a linearized category 
$\mathcal{C}_2^{\textrm{cyl}}(R)$ with colored morphisms.
For this we consider for each $3$-manifold $M$ the morphism category $\textrm{Hom}_M$. This is by definition the 
subcategory of \\ $\textrm{Hom}((F_1,\alpha_1),(F_2,\alpha_2))$ (which is a morphism category of $\mathcal{C}_2^{\textrm{cyl.}}$) with objects all representative surface embeddings $S$ of isotopy classes of surfaces properly embedded in $M$ with $\partial S=\alpha_1\times \{0\}\cup \alpha_2\times \{1\}$ and morphisms all Heegaard isotopy classes of compression $3$-manifolds in $M\times I$. Then replace the set of objects by the free $R$-module  $R\overline{\textrm{Hom}_M}$ generated by the set of colored objects $\overline{\textrm{Hom}_M}$, i.\ e.\ representative surfaces $S$ of isotopy classes of surfaces $S$ in $M$ with colored components (Note that this essentially corresponds to the free $R$-module generated by isotopy classes of surfaces. The subtlety of the chosen representatives comes from the fact that we can choose an ordering of the components of a representative but we cannot choose an ordering for an isotopy class. See also the proof of Theorem 3.5.) The object set of our morphism category
$\textrm{Hom}((F_1,\alpha_1),(F_2,\alpha_2))$ becomes a module, which graded over $3$-manifolds $M$ with 
$\partial M=F_1\times \{0\}\cup F_2\times \{1\}$. In fact the composition by glueing naturally extends using the 
multiplication of colors and bilinear extension
$$R\overline{\textrm{Hom}_N}\otimes R\overline{\textrm{Hom}_M}\rightarrow R\overline{\textrm{Hom}_{N\cup M}}$$
We want to define a corresponding functor.
For this we have to extend the linearization and colorings to $2$-morphisms of $\mathcal{C}_2^{\textrm{cyl.}}$.
Now $R$-homomorphisms are defined on a basis. So it suffices to define the morphisms on some element of 
$\overline{\textrm{Hom}_M}$. Now let $S$ be a representative in $\overline{\textrm{Hom}_M}$. 
Then we consider a Heegaard isotopy class of compression bordisms $W\subset M\times I$ with $\partial W=S\times \{0\}\cup S'\times \{1\}$, which is a morphism in $\textrm{Hom}_M$. We decompose the compression bordism into 
elementary bordisms corresponding to Bar-Natan relations and extend the colors starting from $S$ using the Bar-Natan relations, including self-isotopies permuting closed components of $S$. Then $S'$ will be replaced by an $R$-linear combination of colored representative surfaces. The module of morphisms of the category $\textrm{Hom}_M(R)$ then will be the submodule of the module of all $R$-homomorphisms between the corresponding free modules generated by the   
$R$-homomorphisms defined in this way. It is now not hard to see that the above compositions become functors.
Note that the $2$-morphisms themselves have not been colored but have to respect the colorings of sources and targets. 

\vskip .05in

The linear $2$-category $\mathcal{V}_2$ extends to a colored $2$-category $\overline{\mathcal{V}}_2$ as follows. We will not change the objects of the category. The objects of the morphism categories will be changed as follows. Given $J_1,J_2$ two finite sets then a morphism from $V^{\otimes J_1}$ to $V^{\otimes J_2}$ is an element 
$w\in V^{\otimes K}$ where $V^{\otimes K}$ is a $V^{\otimes J_1}\otimes V^{\otimes J_2}$-module as defined in section \ref{sec=algebra}.
We form the free $R$-module generated by these elements and then take the quotient by relations 
$(V^{\otimes K},w_1)+(V^{\otimes K},w_2)-(V^{\otimes K},w_1+w_2)$ (where we preferred to consider formally $w\in V^{\otimes K}$ as a pair $(V^{\otimes K},w)$).
Note that the composition of morphisms of $\mathcal{V}_2$ involves forming a set $K$ of equivalence classes of some equivalence relation on a $K'\cup K''$ and induces the homomorphism $V^{\otimes K'}\otimes V^{\otimes K''}\rightarrow V^{\otimes K}$. We define the coloring for the composition by taking the image of $w'\otimes w''$ in $V^{\otimes K}$.
Recall that the set of $2$-morphisms has been generated by composition from elementary morphisms defined from the structure morphisms $\Delta, \mathfrak{k}, \varepsilon $ and permutations.
Let $\mathfrak{e}$ be such an elementary morphism.
As above the $2$-morphisms will not be colored by themselves but will relate colored morphisms by mapping $(V^{\otimes K},w)$ to a corresponding $(V^{\otimes K'},\mathfrak{e}(w))$.

\vskip .05in

Let $\mathcal{C}^{\textrm{cyl.}}(R)$ be defined by forgetting the $2$-morphisms of the bicategory $\mathcal{C}_2^{\textrm{cyl.}}(R)$. Then we can summarize our constructions up to this point as follows:
Note that each of the colored categories comes with a natural inclusion functor using colorings by the unit
of the Frobenius algebra only. There is \textbf{no} such inclusion functor for the bicatgeories on the level of 
$2$-morphisms because a $2$-morphism applied to a surface colored by units only will in general not map
the corresponding topological surface colored by units only.

\vskip .1in
 
It turns out to be rather difficult to relate the colored bicategories with the bicategories.
This is mainly due to a \textit{lack of null-morphisms} in the geometric category $\mathcal{C}_2^{\textrm{cyl.}}$.
The null-morphism in this category of course is the empty $3$-manifold $\emptyset \subset \emptyset \times I$ bounding the empty surface $\emptyset $, which is the unique morphism from $(\emptyset ,\emptyset )$ to $(\emptyset ,\emptyset )$. Each of the morphism categories $\textrm{Hom}_M$ for $M$ closed contains a null-object given by the
empty surface embedded in $M$. Note that this is the cylinder, which is the identity object in general, in this case. 

\vskip .05in
 
We will need null-morphisms in \textit{all} morphism sets and also null-$2$ morphisms.
We first extend the bicategory $\mathcal{C}_2^{\textrm{cyl.}}$ by introducing additional null-morphisms between 
any two objects $(F_1,\alpha_1)$ and $(F_2,\alpha_2)$. Then extend the composition in such a way that composition with any null-morphism is the corresponding null-morphism. Extend the $2$-morphisms similarly by introducing null-$2$-morphisms. Then the bifunctor $F_V$ naturally extends to this category by mapping null-morphisms 
to the trivial bimodules and mapping the $2$-morphisms in the obvious way. Call the resulting bicategory 
$\mathcal{C}_2^{\textrm{cyl.}}[0]$. Now consider the bicategory $\mathcal{C}_2^{\textrm{cyl.}}(R)$.  
We call a linear combination of colored surfaces \textit{monic} if it contains with nontrivial coefficients only terms with the same underlying topological surface. Next call an elementary $2$-morphism monic if it is a usual $2$-morphism in $\mathcal{C}_2^{\textrm{cyl.}}(R)$ between two monic linear combinations but additionally, if written as a linear combination in the free module, it does not contain two different topological elementary morphisms.
Consider all possible compositions between those monic $2$-morphisms. 
Let $\mathcal{C}_2^{\textrm{cyl.}}(R,\textrm{mon.})$ be the sub-bicategory defined in this way.

\vskip .05in
 
There is a projection bifunctor 
$$\mathfrak{p}: \mathcal{C}_2^{\textrm{cyl.}}(R,\textrm{mon.})\rightarrow \mathcal{C}_2^{\textrm{cyl.}}[0].$$
It is defined by the identity on objects. Then let $\mathfrak{p}$ map monic non-zero morphisms to the underlying unique non-colored morphism in $\mathcal{C}_2^{\textrm{cyl.}}[0]$, which appears with non-zero coefficients. Map all non-monic morphisms to the null-morphism. Similarly map elementary monic $2$-morphisms to the corresponding elementary $2$-morphism in 
$\mathcal{C}_2^{\textrm{cyl.}}[0]$, and map each elementary non-monic $2$-morphisms to the corresponding null-$2$-morphism. Then extend to compositions in the obvious way. 

\vskip .05in

We define similarly a functor 
$$\mathfrak{q}: \overline{\mathcal{V}}_2\rightarrow \mathcal{V}_2[0].$$
The $2$-category $\mathcal{V}_2[0]$ differs from the bicategory $\mathcal{V}_2$ by adding 
morphisms, which are the trivial $V^{\otimes J_1}\otimes V^{\otimes J_2}$-modules $V^{\otimes K}$ 
also if $J_1\cup J_2\neq \emptyset $ (compare the definitions in section \ref{sec=algebra}.) 
(This corresponds to the case $K=\emptyset $, in which case we do not have maps $\varphi_i : J_i\rightarrow K$,
but we still can consider $V^{\otimes \emptyset }=\{0\}$ as a $V^{\otimes J_1}\otimes V^{\otimes J_2}$-module with the trivial module actions.)
Let $\mathfrak{q}$ map objects to objects. 
Consider the morphisms from $V^{\otimes J_1}$ to $V^{\otimes J_2}$. 
We call an $R$-linear combination of $V^{\otimes J_1}\otimes V^{\otimes J_2}$-modules $(V^{\otimes K},w_K)$ (i.\ e.\ a morphism in
$\overline{\mathcal{V}}_2$) monic if it contains only one term with non-zero coefficient. Then we map monic 
morphisms $(V^{\otimes K},w_K)$ to $V^{\otimes K}$. We map all non-monic linear combinations to the trivial
$V^{\otimes J_1}\otimes V^{\otimes J_2}$-module (this corresponds to $K=\emptyset $). Similarly we define the functor on $2$-morphisms. Note that we do not have to extend on the $2$-morphism level in anon-trivial way because the 
$2$-morphism sets defined in section 3 already always contain a null-morphism.

\vskip .05in

The bifunctor $F_V$ now extends to the \textit{null-extended} bifunctor $\mathcal{C}_2^{\textrm{cyl.}}[0]\rightarrow \mathcal{V}_2[0]$.

\vskip .05in 
 
Let $R$-Mod denote the category of $R$-modules.  

\vskip .05in

\begin{thm} The topological quantum field theory functor $T_V$ extends to a functor:
$\mathcal{C}^{\textrm{cyl.}}(R)\rightarrow R$-Mod. Moreover there is defined a bifunctor:
$$\mathcal{C}_2^{\textrm{cyl.}}(R)\rightarrow \overline{\mathcal{V}}_2
$$
lifting the null-extended bifunctor $F_V: \mathcal{C}_2^{\textrm{cyl.}}[0]\rightarrow \mathcal{V}_2[0]$
with respect to the projection bifunctors $\mathfrak{p}$ and $\mathfrak{q}$ defined above. 
$\square$
\end{thm}

\vskip .05in

In the non-embedded setting each surface $S$ can be compressed to a disjoint union $S'$ of disks bounding
$\alpha_1\times \{0\}\cup \alpha_2\times \{1\}$. It has been shown in \cite{K1} that this allows the local 
calculation of $T_V(S)$ using the extension of $T_V$ to the colored category.   
The bifunctor $F_V$ respectively its extension to the colored bicategory $\mathcal{C}_2^{\textrm{cyl.}}(R)$ allows this calculation in the embedded setting. The upshot is that application of the Bar-Natan relations does \textbf{not} change the value of $T_V$ on the underlying abstract surface.
Explicitly, if $F_V(S')=(V^{\otimes \underline{\alpha _1}}\otimes V^{\underline{\alpha_2}},w_1\otimes w_2)$ then 
$$T_V(S')=T_V(S)=(V^{\otimes \underline{\alpha_1}}\ni v\mapsto \varepsilon^{\otimes \underline{\alpha_1}}(vw_1)w_2).$$
In the embedded case the ways to embed the surface, and moreover the ways to compress the surface are obstructed by the nontrivial topology of the $3$-manifold.
Note that if $S\to S'$ is a morphism between a $V$-colored surface $S$ and $S'=\sum_jr_jS_j$ for $V$-colored surfaces $S_j$, all of the same topological type. The above arguments show that $T_V$ can always be calculated 
from $\widetilde{T}_V(S')=\widetilde{T}_V(S)$ and thus from $F_V(S)$ \textit{together} with the data of the 
genera of the components (i.\ e.\ the map $K\rightarrow \mathbb{N}$ in our combinatorial model).  
Choosing a \textit{generating set} for the Bar-Natan module thus corresponds to choosing normal forms with respect to which the topological quantum field theory can be calculated. It is the insensitivity of the $2$-morphisms with respect to $T_V$, which makes this calculation possible.

\vskip .1in

The cylindrical compression bicategory $\mathcal{C}_2^{\textrm{cyl.}}$ naturally can be generalized as follows
to a bicategory $\mathcal{C}_2$.  
Note that the main point of $M\times I$ used is a kind of \textit{temporal} structure in terms of a natural Morse function. Recall that a general morphism of the category $\textrm{Hom}((F_1,\alpha_1),(F_2,\alpha_2))$ will map
a pair $(M,S)$ to a pair $(N,S')$ (objects of the category $\textrm{Hom}((F_1,\alpha_1),(F_2,\alpha_2))$), where $\partial M=\partial N=F_1\cup F_2$ and $\partial S=\alpha_1\cup \alpha_2=\partial S'$. Such a morphism should be a $4$-manifold $Z$ with corners and a certain isotopy class of $3$-dimensional properly embedded submanifold $W\subset Z$
with corners as follows: $\partial Z=M\cup ((F_1\cup F_2)\times I)\cup N$ (actually we replace $M$ by a copy $M\times \{0\}$ and $N$ by a copy $N\times \{1\}$, and correspondingly for the boundary). $Z$ comes equipped with a Morse function such that the restriction of the Morse function to $W$ is a Morse function with only critical points of index $2$ or $3$. Isotopy of $W$ in $Z$ is supposed to be \textit{Heegaard} as before in the sense that the isotopy only changes the order of critical points
or cancels a pair of index $2,3$ critical points. Also $\partial W\subset \partial Z$ decomposes as 
$S\cup ((\alpha_1\cup \alpha_2)\times I)\cup S'$. As before the composition defined by glueing will define a functor of categories. Because the Bar-Natan functor only measures the change in $S\to S'$ it extends to a bifunctor 
$F_V: \mathcal{C}_2\rightarrow \mathcal{V}_2$.         

\vskip .1in

The bicategory $\mathcal{C}_2$ admits a bifunctor to a category $\mathcal{M}_2$, which forgets the Morse functions and Heegaard condition. Here $\mathcal{M}_2$ is defined just as above but, following the notation from section \ref{sec=compression}, 
the morphisms between surfaces are replaced by general $3$-manifolds $W$ embedded in $4$-manifolds $Z$.

\vskip .05in

\noindent \textbf{Problem} Extend the bifunctor $F_V$ in a \textit{nontrivial} way to the bicategory $\mathcal{C}_2$ such that it also measures the change in the topology $M\to N$ given by the $4$-manifold $Z$.  
 
\vskip .05in

Maybe there exists a \textit{twisting construction} to define a functor on a codimension-one embedded category
in dimensions 3 and 4. (This corresponds to the case where $(F,\alpha )$ is empty.) Here consider a category 
$\mathcal{C}^{\textrm{closed}}$ with objects representatives of diffeomorphism classes of pairs $(M,S)$ where $M$ is a closed $3$-manifold and $S$ is a closed surface embedded in $M$. Then let the morphisms $\textrm{Hom}((M,S),(N,S'))$ be the diffeomorphism classes of pairs $(Z,W)$ relative to the boundary as above, where $Z$ comes equipped with a Morse function restricting to a Morse function on $W$ with only index $2,3$ critical points as before, and such that the critical values for the Morse function $f$ on $Z$ are different from all the critical values of its restriction to $W$. Diffeomorphism of the pair will have to respect the Heegaard property but may allow to move critical points of $f$ acroos critical points of $f|W$. Now suppose we have given a topological quantum field theory $G$ for closed $3$-manifolds bounding $4$-manifolds with values in a category of $R$-modules. Then $G$ naturally combines with the functor $F_V$ for $V$ a weak Frobenius algebra. In fact, an object $(M,S)$ can be mapped to $G(M)\otimes F_V(S)$. Note that functor applied to $S\subset M$ does not take into account the embedding into $M$ at all. Now split $(Z,W)$ according to the critical points of $f$. If we cross a critical point $c$ of $f$ the topology of the ambient $3$-manifold $M$ will be changed by a $4$-dimensional bordism $Z_c$ to a $3$-manifold $N$ but the topology of the embedded surface $S\subset M$ will not change. We assign to this the linear morphism $G(Z_c)\otimes \textrm{Id}$. Similarly if we cross a critical value $d$ of $f|W$ then the topology of $S\subset M$ will be changed to $S'\subset M$ by a compression bordism $W$, and we can apply $\textrm{Id}\otimes F_V(W)$. Because of the compatibility of the tensor product of homomorphisms and composition, $G$ being a topological quantum field theory, and $F_V$ being a functor on $\mathcal{C}(M)$, the
composition of those morphisms will not depend on the decomposition.  

\section{Presentations of colimit functors}\label{sec=presentations}

Throughout let $\mathcal{D}$ be a small category with set of vertices $\mathcal{D}^0$, morphisms will also be called arrows. Recall that for each small category $\mathcal{C}$, 
$\mathcal{C}^1$ denotes the set of morphisms.

\begin{defn} A set $\mathfrak{t}\subset \mathcal{D}^0$ is called \textit{terminal} for $\mathcal{D}$ if for every object $c$ there \textit{exists} a morphism $c\rightarrow t$ for some $t\in \mathfrak{t}$. A terminal set $\mathfrak{t}$ is called \textit{minimal} if $\mathfrak{t}\setminus \{t\}$ is not terminal for each $t\in \mathfrak{t}$. 
\end{defn}

\vskip .1in

Note that a $1$-element terminal set has to be distinguished from a terminal element in the sense of category theory because we do not make any assumptions on the  \textit{uniqueness}. Note that each small category has a terminal set namely the set of objects of the category itself. 
Obviously a terminal set is minimal if and only if there are no morphisms between two different elements of $\mathfrak{t}$. 

\begin{lem} \label{lem Colim terminal} Let $\mathfrak{t}$ be a terminal set for $\mathcal{D}$ and let $G: \mathcal{D}\rightarrow R$-mod be a functor into a category of $R$-modules. Then
$\colim G$ is isomorphic to the quotient of the direct sum of modules $G(t)$ for $t\in \mathfrak{t}$ by the submodule, which is generated by all elements $G(a)v-G(a')v$ where $a:c\rightarrow t$ and $a': c\rightarrow t'$ are any two morphisms and $t,t'\in \mathfrak{t}$. In particular, the module $\colim G$ is generated by a union of sets of generators of the modules $G(t)$ for $t\in \mathfrak{t}$.
\end{lem}

\begin{proof} Let $T$ denote the $R$-module defined by the quotient of the direct sum as above. For $c$ an object in $\mathcal{D}$ define 
an $R$-homomorphism $G(c)\rightarrow T$ by mapping $v\in G(c)$ to the image of $G(a)c\in G(t)$ for any choice of morphism $a: c\rightarrow t$ and $t\in T$,
and then apply the inclusion followed by projection to get the map $G(t)\rightarrow T$. The element in $T$ does not depend on the choice of $a$ by definition of $T$. Next
suppose that $a: c\rightarrow c'$ is a morphism. Then $v\in G(c)$ is equal to $G(a)v$ in $G(c')$ in the colimit by definition. But if $b: c'\rightarrow t$ is a morphism to some element of $\mathfrak{t}$ then we can choose the morphism $b\circ a: c\rightarrow t$. Because $G$ is a functor we see that we have defined a surjective homomorphism $\varphi: \colim G\rightarrow T$. Similarly the homomorphisms $G(t)\rightarrow \colim G$ for $t\in \mathfrak{t}$ define a homomorphism $\psi: T\rightarrow \colim G$, and $\psi \circ \varphi=\textrm{Id}$. Thus $\varphi $ is also one-to-one. 
\end{proof}

The category $R$-mod can be replaced by any abelian cocomplete (i.\ e.\ closed with respect to taking colimits from small categories), see e.\ g.\ the category $\overline{\mathcal{V}(\alpha )}$ as defined in \ref{def=completion}. 
In the following let $\mathcal{L}$ be any cocomplete abelian subcategory of $R$-mod. Note that in particular coproducts 
and pushouts exist and the notion of exact sequence is defined.

Let $G: \mathcal{D}\rightarrow \mathcal{L}$ be a functor. 
For each subcategory $\mathcal{D'}$ of $\mathcal{D}$ we let $G|\mathcal{D'}$ denote the restriction of the functor $G$ to the subcategory. 
For each subset $S\subset \mathcal{D}^0$ we let $\mathcal{D}[S]$ denote the full subcategory of $\mathcal{D}$ generated by $S$.
Let $\mathfrak{t}$ be a terminal set of $\mathcal{D}$. Then $\colim(G|\mathcal{D}[\mathfrak{t}])$ is the quotient of the module $\oplus_{t\in \mathfrak{t}}G(t)$ by the submodule generated by all 
$v-G(a)v$ for all $a: t\rightarrow t'$ and $v\in G(t)$. We use the convention to consider $v-G(a)v\in G(t)\oplus G(t')$ for $t\neq t'$
while the element is in $G(t)$ for $t=t'$.
Obviously if $\mathfrak{t}$ is a minimal terminal set then $\colim(G|\mathcal{D}[\mathfrak{t}])$ is the quotient of $\oplus_{t\in \mathfrak{t}}G(t)$ 
by only \textit{loop contributions} $v-G(a)v$ for $v\in G(t)$ and $a: t\rightarrow t$. 
In fact if there would exist $a: t\rightarrow t'$ for $t,t'\in \mathfrak{t}$ and $t\neq t'$ then $\mathfrak{t}\setminus \{t\}$ would be terminal and $\mathfrak{t}$ would not be minimal.

For each $S\subset \mathcal{D}^0$ consider the category $\mathcal{D}(S)$ consisting of all pairs $<a,a'>$ with 
$a: c\rightarrow t$ and $a': c\rightarrow t'$ morphisms with $t,t'\in S$ and $c$ any object in $\mathcal{D}$. 
This is a full subcategory of the functor category $\mathcal{D}^{\bullet \leftarrow \bullet \rightarrow \bullet}$, see \cite{M}, page 65.
Note that the morphisms
from $<a_1,a_1'>$ to $<a_2,a_2'>$ are commutative diagrams (natural transformations):
$$
\begin{CD}
t_1@<a_1<<c_1@>a_1'>>t_1'\\
@VbVV @Vb''VV @Vb'VV \\
t_2@<a_2<<c_2@>a_2'>>t_2 
\end{CD}
$$

\vskip.05in

The functor $G$ defines a functor $G': \mathcal{D}^{\bullet \leftarrow \bullet \rightarrow \bullet} \rightarrow \mathcal{L}$ as a composition of functors as follows. 
First apply $G$ to diagrams as above with values in the category $\mathcal{L}^{\bullet \leftarrow \bullet \rightarrow \bullet}$,
but introduce the additional sign 
$$<a,a'>\mapsto <G(a),-G(a')>.$$ 
Then apply the pushout functor:
$$
\begin{CD}
G(c)@>-G(a')>>G(t') \\
@VG(a)VV @VVV \\
G(t)@>>>G(a,a')
\end{CD}
$$

Recall that $G(a,a')=\{G(a)v-G(a')v|v\in G(c)\}$. In particular $G'$ assigns to the object $<a,a'>: t\leftarrow c\rightarrow t'$ the module $G(a,a')$.
We call $G'$ the \textit{induced pushout functor}.

Note that $G'(<a,a'>)=G(a,a')$ by definition.

\vskip .1in

We will usually apply $G'$ to subcategories $\mathcal{D}(S)$ of  $\mathcal{D}^{\bullet \leftarrow \bullet \rightarrow \bullet}$, in particular for 
minimal terminal subsets $\mathfrak{t}$.
Special objects in $\mathcal{D}(\mathfrak{t})$ are of the form $<a,\textrm{Id}>: t'\leftarrow t\rightarrow t$. 
For $t=t'$ those are the \textit{loop contributions} at $t$.
Note that the quotient of $\oplus_{t\in \mathfrak{t}}G(t)$ by the submodule generated by the 
$G(a,\textrm{Id})$ is $\colim(G|\mathcal{D}[\mathfrak{t}])$.      

\vskip .1in

We have the following 

\vskip .1in

\begin{thm} \label{thm=cat1} For each terminal set $\mathfrak{t}$ there are exact sequences
$$\bigoplus_{<a,a'>\in \mathcal{D}(\mathfrak{t})^0}G(a,a')\rightarrow \bigoplus_{t\in \mathfrak{t}}G(t)\rightarrow \colim G\rightarrow 0$$
\and
$$\colim(G'|\mathcal{D}(\mathfrak{t}))\rightarrow \colim(G|\mathcal{D}[\mathfrak{t}]) \rightarrow \colim G\rightarrow 0,$$
where the left hand morphisms are defined using $G(a,a')\subset G(t)\oplus G(t')\subset \bigoplus_{s\in \mathfrak{t}}G(s)$ for $t\neq t'$
and by the inclusion of $G(a,a')\subset G(t)$ otherwise.
\end{thm}

\begin{proof} From \ref{lem Colim terminal} we have 
$$\colim G=\bigoplus_{t\in \mathfrak{t}}G(t)/\left< \bigcup_{<a,a'>\in \mathcal{D}(\mathfrak{t})^0}G(a,a')\right>,$$
($<\ >$ denotes the submodule generated by the included set)
which shows exactness of the first sequence. 
Note that each morphism of  $\mathcal{D}(\mathfrak{t})$ 
determines morphisms $b: G(t_1)\rightarrow G(t_2)$ and $b': G(t_1')\rightarrow G(t_2')$ 
such that for $t_i\neq t_i'$ for $i=1,2$ , $b\oplus b'$ maps $G(a_1,a_1')$ to $G(a_2,a_2')$. 
For $t_i=t_i'$ for some $i$ this is similar but looks different in notation.
So these are identified in $\colim(G|\mathcal{D}[\mathfrak{t}])$. Thus the left hand morphism is well-defined
and exactness follows from exactness of the first sequence and the definition of $\colim(G'|\mathcal{D}(\mathfrak{t}))$. 
\end{proof}

\vskip .1in

For $t,t'\in \mathfrak{t}$ let $\mathcal{D}(t,t')$ denote the subcategory of $\mathcal{D}(\mathfrak{t})$, which is generated by all objects 
$t\leftarrow c\rightarrow t'$ and all morphisms

$$
\begin{CD}
t@<a_1<<c_1@>a_1'>>t'\\
@VbVV @Vb''VV @Vb'VV \\
t@<a_2<<c_2@>a_2'>>t' 
\end{CD}
$$

\vskip .1in

Note that if $\mathfrak{t}$ is minimal then the categories $\mathcal{D}(t,t')$ are the components of $\mathcal{D}(\mathfrak{t})$.
It follows that
$$\colim(G'|\mathcal{D}(\mathfrak{t}))=\bigoplus_{(t,t')\in \mathfrak{t}\times \mathfrak{t}}\colim(G'|\mathcal{D}(t,t'))$$

There are relative versions of the above results as follows. Let $S\subset \mathcal{D}^0$. 
Let $S'$ be the set of all objects $t$ in $\mathcal{D}$ for which there exists a morphism $t\rightarrow s$ for some $s\in S$.
Let $\mathcal{D}_S=\mathcal{D}[S']$ be the corresponding full subcategory of $\mathcal{D}$.
Note that by construction the set $S$ is terminal for the category $\mathcal{D}_S$.
Thus we can apply the above results to the categories $\mathcal{D}_S$ with terminal sets $S$.   
But note that $\mathcal{D}_{\mathfrak{t}}=\mathcal{D}$ for $\mathfrak{t}$ any terminal set of the category $\mathcal{D}$. 

We will see below that it is usually \textit{not} true that
$\colim(G|\mathcal{D}_S)$ is the \textit{submodule} of $\colim G$, which is $R$-generated by the set $\cup_{s\in S}G(s)$ 
(because of contributions from  
the category $\mathcal{D}$ outside of $S'$). 
But obviously $\colim(G|\mathcal{D}_S)$ is $R$-generated by $\cup_{s\in S}G(s)$ and projects onto the submodule of $\colim G$ generated by $\cup_{s\in S}G(s)$.
We call $\colim(G|\mathcal{D}_S)$ the $S$-\textit{local} colimit module of the functor $G$.
Often we will apply the above to a subset $\mathfrak{s}$ of an already given minimal terminal subset $\mathfrak{t}$.  
The following is just a reformulation of Theorem \ref{thm=cat1} for the category $\mathcal{D}_S$ with terminal set $S$. 
We use that $\mathcal{D}_S(S)=\mathcal{D}(S)$ because $\mathcal{D}_S^{1}$ contains all possible arrows of $\mathcal{D}$ terminating in
$S$.  

\vskip .1in

\begin{thm} For each $S\subset \mathcal{D}^0$ there are the exact sequences:
$$\bigoplus_{<a,a'>\in \mathcal{D}^0(S)}G(a,a')\rightarrow \bigoplus_{s\in S}G(s)\rightarrow \varprojlim(G|\mathcal{D}_S)\rightarrow 0$$
\textit{and}
$$\colim(G'|\mathcal{D}(S))\rightarrow \colim(G|\mathcal{D}[S])\rightarrow \colim(G|\mathcal{D}_S)\rightarrow 0$$
\end{thm}

If $\mathfrak{t}$ is a \textit{minimal} terminal set for $\mathcal{D}$ and $\mathfrak{s}\subset \mathfrak{t}$ then $\mathfrak{s}$ is a \textit{minimal} terminal set for $\mathcal{D}_{\mathfrak{s}}$.
Thus we have the following:

\begin{cor} \label{cor=exact} If $\mathfrak{t}$ is a minimal terminal set for $\mathcal{D}$ and $\mathfrak{s}\subset \mathfrak{t}$ then there is the exact sequence:
$$\bigoplus_{s,s'\in \mathfrak{s}}\colim(G'|\mathcal{D}(s,s'))\rightarrow \colim(G|\mathcal{D}[\mathfrak{s}])\rightarrow \colim(G|\mathcal{D}_{\mathfrak{s}})\rightarrow 0$$
\end{cor}

In the following we will always assume that we have chosen a well-ordering on the terminal set $\mathfrak{t}$, which we assume is countable.

\begin{cor} \textbf{(i)} \label{cor=cases} If $\mathfrak{t}$ is a minimal terminal set and $\mathfrak{s}=\{s\in \mathfrak{t}|s\leq t\}$ and $t\in \mathfrak{t}$, and $\mathfrak{t}$ is minimal then
$$\bigoplus_{s\leq s'\leq t}\colim(G'|\mathcal{D}(s,s'))\rightarrow \colim(G|\mathcal{D}[\mathfrak{s}])\rightarrow \colim(G|\mathcal{D}_{\mathfrak{s}})$$
is exact.

\noindent \textbf{(ii)} For each minimal terminal subset $\mathfrak{t}$ there is the exact sequence:
$$\bigoplus_{s\leq s'}\colim(G'|\mathcal{D}(s,s'))\rightarrow \colim(G|\mathcal{D}[\mathfrak{t}])\rightarrow \colim(G|\mathcal{D})\rightarrow 0$$
\end{cor}

\begin{proof} This follows from the observation that even though $\mathcal{D}(s,s')\neq \mathcal{D}(s',s)$ those two categories
are naturally isomorphic and their images 
in \\
$\colim(G|\mathcal{D}[\mathfrak{s}])$ respectively $\colim(G|\mathcal{D}[\mathfrak{t}])$ coincide. \end{proof}

\vskip .1in

Note that the sequence in (i) is also exact if the leftmost term is replaced by 
$$\underset{s\leq s'\leq t}{\bigoplus_{<a,a'>\in \mathcal{D}^0(s,s')}}G(a,a'),$$
which projects onto the corresponding colimit.
Also we have the exact sequence:
$$\underset{s\leq s'\leq t}{\bigoplus_{<a,a'>\in \mathcal{D}^0(s,s')}}G(a,a')\rightarrow \bigoplus_{s\leq t}G(s)\rightarrow \colim(G|\mathcal{D}_{\mathfrak{s}})\rightarrow 0$$
The previous observations suggest to consider the following inductive computation of $\colim(G|D_{\mathfrak{t}})=\colim G$ for a minimal terminal 
set for $\mathcal{D}$. First introduce the sets:
$$\mathfrak{s}(t)=\{s\in \mathfrak{t}|s\leq t\}$$
and
$$\mathfrak{s}(t)_-=\{s\in \mathfrak{t}|s<t\}.$$
Note that
$$\colim G=\lim_{t\in \mathfrak{t}}\colim(G|\mathcal{D}_{\mathfrak{s}(t)}).$$
using the natural homomorphisms
$$\colim(G|\mathcal{D}_{\mathfrak{s}(t)_-})\rightarrow \colim(G|\mathcal{D}_{\mathfrak{s}(t)}).$$
and the usual inductive limit in the category $\mathcal{L}$. 

Now 
$$\bigoplus_{s\leq t}G(s)=\bigoplus_{s<t}G(s)\oplus G(t),$$
and from the exact sequence:
$$\underset{s\leq s'\leq t}{\bigoplus_{<a,a'>\in \mathcal{D}^0(s,s')}}G(a,a')\rightarrow \bigoplus_{s\leq t}G(s)\rightarrow \colim(G|\mathcal{D}_{\mathfrak{s}(t)})\rightarrow 0$$
and we deduce

\begin{thm} \label{thm=minimal terminal} For a minimal terminal set $\mathfrak{t}$ there are the exact sequences:
$$\underset{s\leq t}{\bigoplus_{<a,a'>\in \mathcal{D}(s,t)}}G(a,a')\rightarrow \colim(G|\mathcal{D}_{\mathfrak{s}(t)_-})\oplus G(t)\rightarrow 
\rightarrow \colim(G|\mathcal{D}_{\mathfrak{s}(t)})\rightarrow 0$$
\end{thm}

From now on we will always assume that $\mathfrak{t}$ is countable. So we can choose some order such that the sets $\mathfrak{s}(t)$ are finite.
We will see that in our applications the set $\mathfrak{t}$ is the set of isotopy classes of incompressible surfaces in $M$ bounding a closed $1$-manifold $\alpha $ in $\partial M$.
Of course then we can just choose an arbitrary counting 
function, and the Bar-Natan module appears as the direct limit. But, as we will see, there are natural orderings on the set of isotopy classes of incompressible 
surfaces in $M$ closely related to geometry. Often these orderings do \textit{not} satisfy the above finiteness condition.   

\vskip .1in

We will apply the above usually for \textbf{(ii)} in Corollary \ref{cor=cases}. 
Suppose that $\mathfrak{t}$ is a minimal well-ordered terminal set for $\mathcal{D}$.  
For $s\leq t$ and $s,t\in \mathfrak{t}$ let $\mathfrak{d}(s,t)$ denote a set of terminal elements of the category $\mathcal{D}(s,t)$.
Then we can introduce the following 
directed multi-graph $\mathcal{G} (\mathfrak{t})$ to organize the relations. 
The vertices of $\mathcal{G} (\mathfrak{t})$ are the elements of $\mathfrak{t}$. There is a unique directed edge from $s$ to $t$ for each element in
$\mathfrak{d}(s,t)$ and $s\leq t$. 
Note that a priori $\mathcal{G}(\mathfrak{t})$ is \textit{not} a category.
In the next section we will introduce a labelled version of the graph $\mathcal{G}(\mathfrak{t})$, which will contain the complete information 
necessary to calculate the colimit.

If pullbacks are defined in  $\mathcal{D}$ then a composition operation of elements can be defined:
$\mathfrak{d}(t,u)\times \mathfrak{d}(s,t)\rightarrow \mathfrak{d}(s,u)$
for $s\leq t\leq u'$. In general the category $\mathcal{D}$ will not have pullbacks defined but we will be able to use the
pullbacks in $\mathcal{L}$ in a similar way, and we will always assume that $\mathcal{L}$ has pullbacks.
In this case we can make the following observation. Suppose that we have the 
following commutative diagram in $\mathcal{D}$:

$$
\begin{CD}
c''@>>>c'@>>>u \\
@VVV @VVV @. \\
c @>>> t @. \\
@VVV @. @. \\
s @. @.
\end{CD}
$$

It is a simple exercise to prove the following: 

\begin{prop} (pullback principle)  Suppose that $s\leftarrow c''\rightarrow u$ factors in a commutative diagram as above through 
$s\leftarrow c\rightarrow t$ and $t\leftarrow c'\rightarrow u$. Then its contribution to the relations of the corresponding colimit module are already contained 
in those generated by the factors.
\end{prop} 

Note that we do \textit{not} require that $c\leftarrow c''\rightarrow c'$ is the pullback of $c\rightarrow t\leftarrow c'$ in the above diagram.  
By definition pullbacks are \textit{minimal} $c\leftarrow \tilde{c}\rightarrow c'$ in the sense that there is $c''\rightarrow \tilde{c}$ making the obvious diagrams commutative for each $c\leftarrow c''\rightarrow c'$ as above. 
If the pullback diagram exists then it is unique up to isomorphism.

The interesting problem will be to recognize geometrically when two morphisms
in the categories $\mathcal{C}$ defined in section 1 
have a pullback in this category. 
We call an arrow in $\mathcal{G} (\mathfrak{t})$ 
\textit{irreducible} if the arrow cannot be decomposed \textit{as above}.
More precisely we say that an arrow in $\mathfrak{d}(s,u)$ is reducible if there exists some $t\in \mathfrak{t}$ and arrows in 
$\mathfrak{d}(s,t)$ and $\mathfrak{d}(t,u)$ that compose in the above way.
Note that this requires $s\leq t\leq u$.  
By the pullback principle it suffices to consider irreducible arrows of $\mathcal{G} (\mathfrak{t})$.     
  
\vskip .1in

For completeness we compare the $S$-local colimit module $\colim G|\mathcal{D}_S$ with the submodule of $\colim G$ which is generated 
by the $G(t)$ for $t\in S$, where $S$ a subset of the set of objects (note that set of all objects is also a terminal set). The difference is easily seen in the following example. Consider $\mathcal{D}=\mathcal{C}(M)$ for $M=S^3\setminus N(K)$ and $K$ is a nontrivial knot in $S^3$. 
Then let $S=\{\emptyset \}\subset \mathcal{D}^0$.
Then the set $S'$ will consist of all \textit{completely compressible} surfaces in $M$, 
i.\ e.\ surfaces for which there exist a sequence of neck-cuttings and 
sphere deletions completely deleting the surface.
Recall that $\mathcal{D}_S=\mathcal{D}[S']$ has the object set $S'$ and the morphism set all morphisms in $\mathcal{D}$.
It is then easy to see that the local module $\colim(G|\mathcal{D}_S)$ is just $V$ generated by $\emptyset $.   

Then the local module is just $V$ while the submodule 
of $\mathcal{V}(M)$ generated by the $2$-sphere colored in an arbitrary way is the image of $\varepsilon $ in $R$. 
For a subset $T$ of the set of objects of a small category $\mathcal{D}$ let $\overline{T}$ be the union of the set of all objects in $T$ with the set of all objects $s$
in $\mathcal{D}$ for which there exists a morphism $t\rightarrow s$ with $t\in T$. We will apply this to $T=S'$ with $S'$ defined as above, and call
$\overline{S'}$ the \textit{closure} of $S$ in the set of objects of $\mathcal{D}$. 
For each $t\in \overline{S'}\setminus S'$ let $H_S(t)$ be the the union of the images of the linear homomorphisms $G(a): G(s)\rightarrow G(t)$ over all $s\in S'$. 
For $t\in S'$ let $H_S(t)=G(t)$.
Then define a new category $\overline{\mathcal{D}}_S$, the closure of 
the $\mathcal{D}_S$ in $\mathcal{D}$ as follows. It is the subcategory of $\mathcal{D}$ with object set $\overline{S'}$ and morphisms 
all morphisms $a: t\rightarrow s$ for $s,t\in \overline{S'}$ such that $G(a)$ maps $H_S(t)$ into $H_S(s)$. Then we define a new functor $H_S: \overline{\mathcal{D}}_S\rightarrow \mathcal{L}$ 
as follows: On objects in $S'$ it is the functor $\mathcal{D}$, and on objects in $\overline{S'}\setminus S'$ it assigns $H_S(t)$ as defined above. 
On the morphisms it is just the corresponding assignment of the category $\mathcal{D}$.
Then we have almost by construction the following:

\vskip .1in

\begin{prop} The submodule of $\colim G$, which is generated by the union of $G(t)$ for all $t\in S$ is naturally isomorphic to
$\colim(H_S)$. \end{prop}

\begin{exmp} Consider the case of the Bar-Natan catageory $\mathcal{C}$ for some $(M,\alpha )$. Let $S=\{s\}$ for some 
$s\in \mathcal{C}$. Then the set of objects of $\mathcal{D}_{[S]}$ is the set of all objects in the component of $s$. Geometrically this is
the set of all isotopy classes of surfaces homologically equivalent to $s$ in the way that the two orientable surfaces are homologous up to a choice of orientations.
In fact, it follows from \cite{KL} that there always exists an embedded $3$-manifold in $M\times I$ bounding any two surfaces $F,F'$ 
in the same homology class. 
By handle operations we can find morphisms $F\leftarrow F'' \rightarrow F'$. Geometrically this means that for any two given surfaces $F,F'$
in $M$ bounding $\alpha $ there exists a linear combination of $V$-colored surfaces with underying surface $F'$ in the submodule of 
$\mathcal{W}(M,\alpha )$ generated by the collection of $V$-colored surfaces with underlying surface $F$.
\end{exmp}

Local modules are interesting on their own. We discuss some of the local modules for the Bar-Natan functor in section .  

\vskip .1in

The following is immediate from the definition of the colimit of a functor into $R$-mod or $\mathcal{L}$. 
Let $a: c\rightarrow d$ be a morphism. Define a new category $\mathcal{D}^a$ by deleting all morphisms of the form $a''\circ a\circ a'$ for all 
morphisms $a'',a'$ in $\mathcal{D}$ from the category. 

\begin{lem} Suppose that a morphism $a: c\rightarrow d$ has a left-inverse $b: d\rightarrow c$ in $\mathcal{D}$.
Let $G: \mathcal{D}\rightarrow \mathcal{L}$ be any functor. 
Let $G^a: \mathcal{D}^a \rightarrow \mathcal{L}$ be the restriction. Then $\colim G=\colim G^a$.
\end{lem}

\begin{proof} Let us assume $c\neq d$. The relations from the morphism $a$ are \\ $(v,G(a)v)\in G(c)\oplus G(d)$ for all $v\in G(c)$ while the relations from $b$ are $(G(b)w,w)\in G(c)\oplus G(d)$ for all 
$w\in G(d)$. So if we apply the second relation to $w=G(a)$ and use $\textrm{Id}=G(\textrm{Id})=G(b\circ a)=G(b)\circ G(a)$ we see that the relations from $b$ 
imply the relations from $a$. The case $c=d$ is even easier. \end{proof}

\begin{rem} Because of the previous lemma and Remark \ref{rem=trivial}. we will from now on delete all trivial $2$-handle morphisms
from our category $\mathcal{C}$. More precisely, we first eliminate these morphisms from the directed graph $\mathfrak{c}$ and thus from the free category $\mathcal{F}\mathfrak{c}$, and we also eliminate
also all compositions with these morphisms. Then we
take the quotient by the relations defined by our restricted form of isotopy as before. Note that trivial $2$-handle morphisms are not in any nontrivial way compositions of 
other morphisms. (Of course it could appear as the composition of a self-isotopy morphisms with the composition of its inverse with the trivial 
morphism.)  
We can also go back to the very definition of the category and only include bordisms in $M\times I$ with $2$- and $3$-handles but no trivial $2$-handles attached 
to surfaces in $M$.
On the level of the Bar-Natan skein module this means that we reduce to neck-cuttings on curves,
which are essential on the surface, except if none of the resulting $2$-sphere components bounds a $3$-ball in $M$.
This can only occur when $M$ is 
reducible. It should be observed that the effect of a trivial $2$-handle on a surface is the same as that of attaching a $0$-handle (even though the 
bordisms are different). In fact, the trivial $2$-handle can be cancelled by a $3$-handle, and the bordism representing this composition is 
isotopic to the bordism where these have been cancelled. Similarly the trivial $0$-handle can be cancelled by a $1$-handle. 
\end{rem}

We like to observe that in the case of the Bar-Natan module of a closed connected oriented $3$-manifold the submodule of $\mathcal{W}(M)$, which is generated by the empty surface $\emptyset $, is isomorphic to the $\{\emptyset \}$-local module. This is because the closure of the set coincides with the set in this case.
This is because there are \text{no arrows out of this} set. Note that the same is always true for certain subsets of the final set of incompressible surfaces.
It will turn out to be true for \textit{all} subsets of $\mathfrak{t}$ in a large class of $3$-manifolds, which will be discussed in the next section, including the irreducible $3$-manifolds. 

\section{Towards presenting Bar-Natan modules}\label{sec=towards}

Recall that a properly embedded surface $F$ embedded in $M$ is \textit{compressible} if there is a loop on $F$, which does not bound a disk on $F$ but bounds a disk 
in $M$ intersecting $F$ only in the loop, or $F$ has a $2$-sphere component $S$ bounding a $3$-ball in $M$ intersecting $F$ only in $S$. Otherwise $F$ is called \textit{incompressible}. $M$ is \textit{irreducible} if every embedded $2$-sphere in $M$ bounds a $3$-ball. 

\begin{defn} Suppose $M$ is a connected orientable $3$-manifold. The pair $(M,\alpha )$ is called \textit{Haken-reducible} if there exists an
incompressible surface $F$ in $M$ bounding $\alpha $ 
containing a $2$-sphere component $S$, which does not bound a $3$-ball 
and is not parallel to any component of $F\setminus S$, and $F\setminus S\neq \emptyset$. 
The pair $(M,\alpha )$ is called \textit{Haken-irreducible} if it is not Haken-reducible.
\end{defn}
 
Note that $M$ irreducible implies $(M,\alpha )$ Haken-irreducible for all $\alpha $. The converse is not true.
$S^2\times S^1$ and $(S^2\times I,\emptyset)$ are Haken-irreducible but not irreducible. But $(S^2\times I,\alpha )$ is Haken-reducible for $\alpha \neq \emptyset $.
If a connected $3$-manifold $M$ contains a non-trivial $2$-sphere $S$ and it also contains a surface $F$ bounding $\alpha $, which is not parallel to $S$ then if the transversal intersection $S\cap F\neq \emptyset$ we can do double point surgery on the intersection curves to get a surface $F'$, which will be disjoint from $S$. Then $F'\cup S$ is a surface showing that $(M,\alpha )$ is Haken-reducible. 
A compact orientable $3$-manifold $M=(M,\emptyset )$ is Haken-irreducible if and only if 
it has a component with a prime factorization with at least three factors, or a prime decomposition with two factors and one of them is
diffeomorphic to $S^2\times S^1$ or to a Haken manifold (which is a compact orientable irreducible $3$-manifold containing a two-sided incompressible surface). This follows easily from \cite{He}, 3.8.
Given an incompressible surface $F\subset M$ with $\partial F=\alpha $, with $M$ not a connected sum, there is a unique up to isotopy system of incompressible nonparallel $2$-spheres 
in $M\setminus F$ which is maximal in the sense that none of these $2$-spheres is a connected sum of two other nontrivial $2$-spheres. Call this a 
\textit{maximal system} for $F$. 
The first assertion of the following theorem is the observation used by Asaeda and Frohman to prove the Corollary of this theorem.

\begin{thm} The set of isotopy classes of incompressible surfaces in $(M,\alpha )$ is a terminal set for the category
$\mathcal{C}(M,\alpha )$. It is a minimal terminal set if and only if $(M,\alpha )$ is Haken-irreducible. Moreover, if $M$ is not Haken-irreducible then there does not exist \textit{any} minimal terminal set for $\mathcal{C}(M,\alpha )$. 
\end{thm}

\begin{proof} First note that  
neck-cuttings ($2$-handle attachments) morphisms and capping off $2$-spheres by $3$-balls morphisms we find morphisms from each isotopy class of surface in $(M,\alpha )$ to an incompressible surface. 
Thus the set of isotopy classes of incompressible surfaces is a terminal set. 

Next suppose that $(M,\alpha )$ is Haken-irreducible. 
We claim that in this case there are no arrows in the category $\mathcal{C}(M,\alpha )$ out of incompressible surfaces. (Recall that we have eliminated those resulting from trivial neck-cuttings from our category.) In fact because of incompressibility no non-separating neck-cuttings are possible. A separating neck-cutting, which is nontrivial leads to 
a surface with two new components resulting from the cutting. 
None of these is a $2$-sphere bounding a $3$-ball in $M$. If the resulting two components are 
two parallel $2$-spheres then the original component would be a $2$-sphere bounding a $3$-ball. This would contradict that we started from an incompressible surface. Thus we get a union of two incompressible surfaces, one of which a nontrivial $2$-sphere $S$, and a second component not parallel to the $2$-sphere. But this cannot occur because we assume Haken-irreducibility. 

Conversely, if $(M,\alpha )$ is Haken-reducible then we can form the embedded connected sum $F\sharp S$ of some incompressible surface $F$ in the complement of some incompressible not parallel $2$-sphere $S$(choose orientations suitably). This will be a new incompressible surface, not isotopic to $F$. Thus there is a morphism $F\sharp S\rightarrow F\cup S$ with the resulting surface incompressible. In particular there is an arrow out of 
an incompressible surface. This shows that we can omit the incompressible surface $F\sharp S$ and still have a terminal set. 
Thus the set of incompressible surfaces is not minimal. 

In order to prove the last claim 
first note that each minimal terminal set is a subset $\mathcal{T}$ of the set of isotopy classes of incompressible surfaces.
Given $F$ and $S$ as in the definition of Haken-reducibility note that $F$ is isotopic to $F\sharp S\sharp S$ and two neck cuttings
define a morphism $F\to F\sqcup _2S=:F_1$. This process can be iterated to define an infinite sequence of incompressible surfaces 
$F\to F_1\to\ldots \to F_i\to F_{i+1}\to \ldots $. Note that the only possible morphism out of some $F_i$ is to $F_{i+1}$ or to some 
$F_i\sqcup_2S'$ for a different $2$-sphere $S'$ meeting the definition of Haken-reducibility for $F_i$ and $S'$. By compactness the set of isotopy classes of possible $2$-spheres is finite. So after reordering, we have a countable collection of isotopy classes of incompressible surfaces $F_i$ related by corresponding neck-cutting morphisms. Now we need $F_i\in \mathcal{T}$ for some $i$ for $\mathcal{T}$ to be a terminal set. On the other hand, $\mathcal{T}\setminus \{F_i\}$ is terminal because there will be some $F_j$ with $j>i$ and a morphism $F_i\to F_j$. Thus $\mathcal{T}$ is not minimal terminal.  
\end{proof}

\begin{cor} The skein module $\mathcal{W}(M,\alpha )$ is $R$-generated by isotopy classes of incompressible oriented surfaces
bounding $\alpha $ with the components colored by the elements of a basis of $V$. The module is $V^{\otimes |\alpha |}$-generated by isotopy classes of incompressible surfaces with the bounded components colored $1\in V$ and closed components colored arbitrarily. $\square$ \end{cor}

\begin{exmp} \textbf{(a)} $\mathcal{W}(D^3;\alpha )$ is the cyclic free $V^{|\alpha |}$-module $V^{|\alpha |}$. This has been shown in \cite{K1}.

\noindent \textbf{(b)} If there are no closed incompressible surfaces in $M$ then the module $\mathcal{W}(M,\alpha )$ is $V^{\otimes |\alpha |}$-generated by
the set of incompressible surfaces in $M$ with components colored $1$. 
\end{exmp}

\begin{thm} Suppose that $M$ is irreducible and is not diffeomorphic to a surface bundle over $S^1$. Then $\colim(F|\mathcal{C}[\mathfrak{t}])$ is isomorphic to $\oplus_{t\in \mathfrak{t}}F(t)$, where $\mathfrak{t}$ is the minimal terminal set of incompressible surfaces in $M$ bounding $\alpha $.\end{thm}

\begin{proof} Because of irreducibility the set of incompressible surfaces is minimal. By a result of Hatcher \cite{H} the category
$\mathcal{C}[\mathfrak{t}]$ has no loop isotopies if $M$ is not a surface bundle over $S^1$. Thus the category $\mathcal{C}[\mathfrak{t}]$ 
is discrete and the result follows. \end{proof}

We return to the general set-up.
If $(M,\alpha )$ is Haken-irreducible we are in the situation of \ref{cor=exact}. The functors $F'|\mathcal{C}(t,t')$ for $t,t'\in \mathfrak{t}$ then have the following 
interpretation: The objects of $\mathcal{C}(t,t')$ are isotopy classes of $3$-manifolds $W\subset M\times [-1,1]$, properly embedded and bounding 
incompressible surfaces representing $t,t'$ in $M\times \{\pm 1\}$, product along $\alpha $, such that the restriction of the 
projection to $M\times [-1,1]\rightarrow [-1,1]$ is a generic Morse function with all critical points of index $0,1$ mapping to $t<0$ and those of index $2,3$ mapping to $t>0$. Isotopies between these manifolds will keep this separation of indices. 
This is called a \textit{Heegaard embedding} up to \textit{Heegaard isotopy}. Note that for irreducible $M$ the incompressible surfaces have no
$2$-sphere components and $W$ are compression bodies. But since in the more general case $2$-sphere components of $t,t'$ are still incompressible 
this seems to be the natural version of the idea of a compression body in the case of embeddings.  
The morphisms $W\rightarrow W'$ are oriented $3$-manifold representing elements in $\mathcal{C}(M,\alpha )$, i.\ e.\ with only $2$- and $3$-handle attachments, from $W\cap (M\times \{0\})$ to $W'\cap (M\times \{0\})$ such that the compositions with $W'\cap (M\times [-1,0])$ respectively $W'\cap (M\times [0,1]$ are isotopic in the strong sense to $W\cap (M\times [-1,0])$ respectively $W\cap (M\times [0,1])$. Essentially this allows to cancelling $1$- and $2$-handles from 
$W$ if possible. A terminal set in $\mathcal{C}(t,t')$ is thus a set of isotopy classes of \textit{minimal} Heegaard embeddings
in $M\times I$. 
  
Let $\alpha $ be a closed $1$-manifold. Recall from section \ref{sec=compression} that $\mathcal{C}(\alpha )$ denotes the category with objects abstract (or if the reader wants embedded in $\mathbb{R}^{\infty}$) orientable surfaces with boundary $\alpha $. The morphisms from $S$ to $S'$ are orientable $3$-manifolds bounding $S$ and $S'$. This is the \textit{abstract} version of the corresponding category $\mathcal{C}(M,\alpha )$ without any embedding in a $3$-manifold. The functor in $\mathcal{V}(\underline{\alpha })$ is defined as in the embedded case. Recall the corresponding Bar-Natan modules 
$\mathcal{W}(\alpha )\cong V^{\otimes \underline{\alpha }}$, which only depend on $\underline{\alpha }$, see \cite{K1}. Note that each category has a unique terminal 
$1$-element set given by the isotopy class of the disjoint union of disks $D$ bounding $\alpha $. It is not completely obvious that the colimit of the Bar-Natan functor is isomorphic to $V^{\otimes \underline{\alpha }}=F(D)$ in this case. But the relations defined from $W_{DD}$ are all just the identity. This follows 
from the fact that $W_{DD}$ on both sides capped off by balls is a closed $3$-manifold, and the analysis above then shows that the relations 
are trivial.     

Recall from section \ref{sec=compression} the functor of categories $\mathcal{C}(M,\alpha )\rightarrow \mathcal{C}(\alpha )$ defined by mapping both the objects and morphisms to the 
corresponding abstract surfaces respectively $3$-manifolds, forgetting the embeddings in $M$ respectively $M\times I$. The important point of course is that in general many different surfaces in $M$ up to isotopy become diffeomorphic abstractly. The induced natural homomorphism of the colimit modules is a $V^{\otimes \underline{\alpha }}$-homomorphism
$$\mathfrak{f}: \mathcal{W}(M,\alpha )\rightarrow V^{\otimes \underline{\alpha }}.$$
This is \textit{not onto} in general and nontrivial if and only if $\alpha $ bounds in $M$. Otherwise $\mathcal{W}(M,\alpha )$ is the trivial $V^{\otimes \underline{\alpha }}$-module $0$. On the other hand if $\alpha =\emptyset $ then $\mathfrak{f}$ is onto $R$ because we can embed arbitrary closed surfaces in a $3$-ball in $M$. 

\begin{exmp} Consider $M=S^1\times D^2$ with two oppositely oriented longitudes in the boundary. Then the Bar-Natan module is the $V^{\otimes 2}$-module $V$ generated by the obvious annulus bounding $\alpha $ and colored by $1$. The image in the abstract skein module $V^{\otimes  2}$ is the image of $\Delta $. This is the  cyclic $V^{\otimes 2}$-submodule of $V^{\otimes 2}$ isomorphic to $V$. This shows that the image usually is \textit{not} a split submodule. Note that $V$ is the cyclic $V^{\otimes 2}$-module corresponding to the exact sequence
$$\ker(m)\hookrightarrow V\otimes V\twoheadrightarrow V.$$
For example if $V=R[x]/(x^2)$ and $R$ is an integral domain then $\ker(m)$ is precisely the $V^{\otimes 2}$-ideal of $V^{\otimes 2}$ generated by 
$1\otimes x-x\otimes 1$.  
\end{exmp}

\vskip .1in

For each $V^{\otimes \underline{\alpha }}$-submodule $H$ of $V^{\otimes \underline{\alpha }}$ consider the preimage 
$\mathfrak{f}^{-1}(H)\subset \mathcal{W}(M,\alpha )$. This is a kind of $H$-adic version of Bar-Natan modules
$\mathcal{W}(M,\alpha )$. 

It seems tempting at this point to consider a Heegaard splitting of a closed connected $3$-manifold and the two linear morphisms $V\rightarrow R$ 
assigned by the two handlebodies. It follows from the considerations before that both morphisms $V\rightarrow R$ are just given by 
$\varepsilon \circ \mathfrak{k}^g$, where $g$ is the genus of the handlebody. We see that the functor is \textit{not} a functor on $3$-manifolds but 
on Heegaard splittings, which probably is expected.

\section{Tunneling graphs of $(M,\alpha )$}\label{sec=tunneling}

\vskip .1in

Let $\mathfrak{t}=\mathfrak{t}(M,\alpha )$ denote the set of isotopy classes of incompressible surfaces in $M$ bounding $\alpha $. 

If $(M,\alpha )$ is Haken-irreducible then
the category $\mathcal{C}(\mathfrak{t})$ as defined in section \ref{sec=presentations} determines the relations of Bar Natan modules. But it is a geometrically difficult category,
defined by isotopy of $3$-manifolds in $M\times I$. On the other hand we will see that the combinatorics defining the relations is simple, depending on the isotopy classes only through certain \textit{tunneling numbers}. 
The necessary geometric information will be collected in the tunneling graph of $(M,\alpha )$ defined below. 

An \textit{$r$-partition} of a finite set $\{1,2,\ldots ,n\}$ is an \textit{ordered} collection of disjoint subsets 
$\mathsf{T}_1,\ldots ,\mathsf{T}_r\subset \{1,2,\ldots ,n\}$ such that 
$\mathsf{T}_1\cup \ldots \cup \mathsf{T}_r=\{1,2,\ldots ,n\}$. Note that $\mathsf{T}_i=\emptyset$ for some $i$ is possible.

\begin{defn} Let $S,S'$ be two incompressible surfaces in $M$ bounding $\alpha $, representing two isotopy classes and thus objects of the category $\mathcal{C}(M,\alpha )$.  An \textit{$r$-tunneling invariant from $S$ to $S'$} consists of $r$-partitions $(\mathsf{S}_i)_{1\leq i\leq r}$ and $(\mathsf{S}'_i)_{1\leq i\leq r}$ of the set of components of $S$ and $S'$ (where in each case the components are identified in some way with  a set $\{1,\ldots ,n\}$ for some $n$) such that for each $i$ not both of $\mathsf{S}_i$ and $\mathsf{S}'_i$ are empty, together with an $r$-tuple of nonnegative integer numbers $(\tau_1,\ldots ,\tau_r)$, which is called the \textit{tunneling vector}. We call $(\mathsf{S}_i)$ the \textit{input} and $(\mathsf{S}'_i)$ the \textit{output} partition of the corresponding tunneling invariant. To simplify notation we also let $\mathsf{S}_i$ respectively $\mathsf{S}'_i$  denote the union of components of the surfaces $S$ respectively $S'$ determined by the partition. \end{defn}

Partitions do not involve specific orderings, we have just chosen orderings and partition sets of the components for the definition. 
For example if $S$ has components $S_1,S_2,S_3$, and $S'$ has components $S_1',S_2'$ a typical pair of $3$-partitions could be
$\mathsf{S}_1=\{S_1,S_2\}$, $\mathsf{S}_2=\emptyset$, $\mathsf{S}_3=\{S_3\}$, and $\mathsf{S}'_1=\emptyset$, $\mathsf{S}'_2=\{S'_1,S'_2\}$,
$\mathsf{S}'_3=\{S'_3\}$. The orderings are \textit{not} part of the structure. It is just recorded that $S_1,S_2$ form a \textit{part},
$S'_1,S'_2$ is a second part and $S_3,S'_3$ is the third one.  

\vskip 0.05in

Each morphism $W: S\rightarrow S'$, considered as an object in the category $\mathcal{C}(S,S')$, determines a unique $r$-tunneling invariant, with $r$ the number of components with nonempty boundary. 
We omit closed components of $W$. Thus $W$ is the disjoint union of components $W_i$, $1\leq i\leq r$. 
Let $\mathsf{S}_i$ be the union of components of the intersection of $W_i$ with $M\times \{0\}$ respectively $\mathsf{S}'_i$ be the union of components of the intersection with $M\times \{1\}$. Thus \textit{not both} $\mathsf{S}_i$ and $\mathsf{S}_i'$ can be empty because otherwise $W_i$ is closed. It is known \cite{KL} that we can cancel
handles by Heegaard isotopy such that $W_i$ has at most one $3$-handle respectively one $0$-handle, which occurs if $\mathsf{S}_i=\emptyset$ respectively 
$\mathsf{S}_i'=\emptyset$. Fix some $i$ and let $\mathsf{S}_i=S_1\sqcup \ldots S_k$ and 
$\mathsf{S}_i'=S_1'\sqcup \ldots \sqcup S_{\ell }'$ so that for the genera
$$g(\mathsf{S_i})=\sum_{j=1}^kg(S_i), g(\mathsf{S}_i'=\sum_{j=1}^{\ell}S_j'.$$

Add collars to all components of $\mathsf{S}_i$ and $\mathsf{S}_i'$. First consider $W$ as an abstract bordism from 
$\mathsf{S}_i$ to $\mathsf{S}_i'$, i.\ e.\ consider the image of the bordsim is $\mathcal{C}(\alpha )$. 
Since $W_i$ is connected there will be $1$-handles attached to the $S_j$ respectively $S_j'$ (corresponding to $2$-handles on $\mathsf{S_i}$) to connect the surfaces. Those are \textit{connecting} $1$-handles and the numbers are determined by the 
numbers $k,\ell $ of components. Of course in general there can be connecting $1$-handles canceled by $2$-handles again but we can ignore those. They will not change the algebra after applying the Bar-Natan functor. 
We will assume that we first attach the connecting $1$-handles $\mathsf{S}_i$ and $\mathsf{S}_i'$ such that we have 
connected surfaces. Next we consider the difference in the genus of the surfaces. If $g(\mathsf{S}_i')\geq g(\mathsf{S}_i)$ then 
in the abstract model non-connecting handles have to be added on $\mathsf{S}_i'$ (after connecting) for an abstract bordism, in fact 
exactly $g(\mathsf{S}_i')-g(\mathsf{S}_i)$ will be necessary. Note that the notion of connecting and non-connecting in a way is only defined using that specific order of handle attachment, otherwise things are interchangeably. What we really consider is the excess necessary $1$-handles to 
match the genera beyond connecting the components of $\mathsf{S}_i$ respectively $\mathsf{S}_i'$. 
Now an \textit{embedded} bordism will need additional non-connecting $1$-handles on $\mathsf{S}_i$. 
Now if $r_1^i$ is the total number of non-connecting $1$-handles of $W_i$ we define
$$\tau_i:=r_1^i-(g(\mathsf{S}_i')-g(\mathsf{S}_i)).$$
We think of this as the excess necessary for the embedding of $W_i$.
Now if $g(\mathsf{S}_i'<g(\mathsf{S}_i)$ then we define 
$$\tau_i=r_2^i-(g(\mathsf{S}_i')-g(\mathsf{S}_i)),$$
where $r_2^i$ is the total number of \textit{non-connecting} $2$-handles of $W_i$ corresponding to non-connecting $1$-handles of the turned upside down bordism.

We will next develop the combinatorics how to calculate the pair of linear morphisms $F_{\mathfrak{t}}(W)=:(F_-,F_+)$ for an arbitrary fixed morphism $W$ as
above.  Note that possible closed components of $W$ do not contribute. In fact we will have $F_+: V^{\otimes r}\rightarrow V^{\otimes \underline{S}}$
and $F_-: V^{\otimes r}\rightarrow V^{\otimes \underline{S'}}$. Note that closed components do not contribute because  correspond to some 
morphism $R\rightarrow R$, and by Poincare duality for each component both $F_+$ and $F_-$ would be multiplications by the same powers of $\mathfrak{k}$. 
This in fact has to show up in the actual calculation of the functor $\mathcal{F}_{\mathfrak{t}}$. But of course we can just omit the closed component from 
$W$ to get another morphism, where we do multiply both $F_+$ and $F_-$ by the same powers of $\mathfrak{k}$. In other words, in calculating 
the relations it suffices to consider only morphisms in $\mathcal{C}(S,S')$ without any closed components.(Note that this means that in special case $S=S'=\emptyset$, which could occur if $\alpha =\emptyset$, our argument shows that the are no relations coming from $\mathcal{C}(\emptyset ,\emptyset)$.)

In the following we will always drop all closed components from $W$. Note that this might even change the homology class of $W$. But since we are 
only interested in describing all relations for $\mathcal{W}(M,\alpha )$ we do not drop necessary information by leaving out those components, even though 
they show up in the calculation of the functor $\mathcal{F}_{\mathfrak{t}}$.
A tunneling invariant from $S$ to $S'$ is called \textit{reducible} if there is a sequence $S=S_0\mapsto S_1\mapsto S_2\mapsto \ldots \mapsto S_r=S'$
and tunneling invariants for $S_i\to S_{i+1}$ such that the relation defined by the given tunneling invariant follows algebraically from the relations defined from the sequence by iterated pullback.   

Note that the sum of the homology classes of the surfaces in $\mathsf{S}_i$ is equal to the sum of the homology classes in $\mathsf{S}_i'$ for each $i=1,2,\ldots ,r$, up to choices of orientations of the components. If for a given pair of $r$-partitions a tunneling vector $(k_1,k_2, \ldots ,k_r)$ is realized in a tunneling invariant then each vector 
$(k_1+n_1,k_2+n_2,\ldots ,k_r+n_r)$ is realized. This can be seen by adding cancelling $1$-$2$ pairs. We call a pair of partitions admitting a 
tunneling vector in a tunneling  invariant an \textit{admissible} pair of partitions. Obviously the tunneling vectors of tunneling invariants can be chosen 
minimal because the relations from stabilized tunneling vectors are implied by those from minimal ones. Tunneling invariants of this form and the corresponding vectors will be called \textit{minimal}. 

Let us set up some further algebraic notations. Let $d_n: V\rightarrow V^{\otimes n}$ be defined by $d_0=\varepsilon $, $d_1=\textrm{Id}$, $d_2=\Delta $ and 
$d_n=(\textrm{Id}\otimes d_2)\otimes d_{n-1}$. Our algebra of morphisms is generated by the $d_n$ and $\mathfrak{k}:V\rightarrow V$. 
Note that the family of $\{d_n\}$ satisfies a nice cooperad structure. For example 
$(d_{n_1}\otimes d_{n_2}\otimes \ldots \otimes d_{n_k})\circ d_k=d_{n_1+\ldots +n_k}$. This follows from coassociativity.
Also there is nice compatbility with permutations, for example $\sigma d_n=d_n$ if $\sigma $ is the permutation of $V^{\otimes n}$ induced from the 
permutation of $n$ elements.
For $n\geq 1$ also define $\varepsilon^n: V^{\otimes n}\rightarrow V$ inductively by $\varepsilon^1=\textrm{Id}$ and 
$\varepsilon^n=(\varepsilon \otimes \textrm{Id})\circ \varepsilon^{(n-1)}$.
Then $\varepsilon^n\circ d_n=\textrm{Id}$ showing explicitly a left inverse of $d_n$.  

Suppose that $V$ is manifold induced. Assume that $n_1\neq n_1'$ but $n_1+n_2=n_2+n_2'=:n$. Suppose that there exist $v_i,v_i'$ for $i=1,2$ such that
$(d_{n_1}\otimes d_{n_2})(v_1\otimes v_2)=(d_{n_1'}\otimes d_{n_2'})(v_1'\otimes v_2')$. 
Then $v_1\otimes v_2=v_1'\otimes v_2'=\Delta (v)$ for some $v\in V$. This follows from the geometric fact that 
with $\Delta_n\subset X^n$ denoting the diagonal we have $(\Delta_{n_1}\times \Delta_{n_2})\cap (\Delta_{n_1'}\times \Delta_{n_2'})=\Delta_n$
if $n_1\neq n_1'$. 

For further explicit calculations it is helpful to introduce linear morphisms $d_I: V \rightarrow V^{\otimes n}$, for subsets $I=\{i_1,\ldots ,i_{|I|}\}\subset \{1,\ldots ,n\}$ by having $d_n$ by placing the result of $d_n$ into outputs corresponding to $I$ in the given ordering of factors, similarly: $\varepsilon_k: V^{\otimes n}\rightarrow V^{\otimes (n-1)}$ by having $\varepsilon $ act on the $k$-th factor of $V^{\otimes n}$ and identifying $R\otimes V$ respectively $V\otimes R$ with $V$, so the output doesn't matter.
Similar notation applies to $\mathfrak{k}$. If $V$ is a Frobenius algebra similar notations can be introduced for multiplications. Note that because of 
associativity there is defined inductively $m_n: V^{\otimes n}\rightarrow V$ with $m_1=\textrm{Id}$ and $m_0=\mu :R\rightarrow V$. Obviously $m_n\circ d_n=(\mathfrak{k})^{n-1}$. The algebra of the $d_n,m_n$ of course corresponds to the algebra of diffeomorphism classes of surfaces. What we observe is that the algebra of our kind of \textit{half-bordism} category of oriented $3$-manifolds is similar.   

It suffices to consider $W$ connected such that $F_+: V\rightarrow V^{\otimes \ell}$ and $F_-: V\rightarrow V^{\otimes k}$, where $\ell $ and $k$ as above are the number of elements of the boundary of $W$ in $M\times \{1\}$ respectively $M\times \{-1\}$. 
So consider some $1$-tunneling invariant for the pair $(S,S')$, so $\ell $ and $k$ are given. 
Also given $S,S"$ we can distinguish the cases $g(S')>g(S)$ and $g(S')<g(S)$. 
Here $g(S)$ is the total genus of $S$, which is the sum of the components, and for a bounded surface the genus is the genus of the closed surface 
we get by filling all holes by disks. In the first case we have 
Then $F_-=d_k\circ \mathfrak{k}^{g(S')-g(S)+\tau }$ and $F_+=d_{\ell }\circ \mathfrak{k}^{\tau }$. In the second case we have
$F_-=d_k\circ \mathfrak{k}^{\tau }$ and $F_+=d_{\ell }\circ \mathfrak{k}^{g(S)-g(S')+\tau }$. The linear maps $F_{\pm}$ for more general tunneling invariants 
are just pasted together from the contributions for each pair of partition sets $(\mathsf{S}_i,\mathsf{S}_i')$ and the genera of the surfaces in the corresponding 
partition set. Of course the general answer can be difficult to write down using the notation above. 

\begin{exmp} \textbf{(a)} Suppose that $|S|=|S'|=1$ and $g(S')>g(S)$. Then there is only one possible $1$-partition: $\mathsf{S}_1=\{S\}, \mathsf{S}_1'=\{S'\}$, and one possible $2$-partition (in the case $\alpha =\emptyset $ and $S,S'$ in the trivial homology class in $H_2(M)$):
$\mathsf{S}_1=\{S\},\mathsf{S}_2=\{\emptyset \},\mathsf{S}_1'=\{\emptyset\}$ and $\mathsf{S}_2=\{S'\}$. 
In the $1$-partition case we have $F_{\pm }: V\rightarrow V$ are given by $F_+=\mathfrak{k}^{\tau }$ and $F_-=\mathfrak{k}^{g(S')-g(S)+\tau }$,
where $\tau $ is part of a tunnel invariant of $(S,S')$. In the $2$-partition case we have $F_{\pm }: V\rightarrow V$. Here the tunneling invariant contains the 
pair $\tau_1,\tau_2$. Then $F_- =\mathfrak{k}^{\tau_1}$ and $F_+=\mathfrak{k}^{\tau_2}$. 
It is important to observe that by adding cancelling $1$-$2$ handle pairs one can always realize tunneling invariants with larger values of $\tau_i$.

\noindent \textbf{(b)} Let $|S|=2$ and $|S'|=0$. In this case there are $1$- and $2$-tunneling invariants possible. The second case only occurs if 
each component of $S$ is separately null-homologous. 

\noindent \textbf{(c)} Consider $(D^3,\alpha )$ where $D^3$ is the compact $3$-ball. Then the only incompressible surfaces are disks bounding the 
components of $\alpha $. Recall that $W$ is parallel along $\alpha $. This puts strong restrictions on possible 
partitions. In fact if a component of $W$ contains bounds a disk on one side it also has to bound the same disk on the other side because it will contain 
$\alpha \times I$. Of course it is possible to construct $3$-dimensional bordisms by taking connected sums 
of components of the trivial bordism $D^{|\alpha |}\times I\subset M\times I$ (corresponding to cancelling $1$-$2$-handle pairs that can be added with the $1$-handle connecting and the $2$-handle disconnecting). But these are the only bordisms possible. 
For the resulting homomorphisms we get in any case $F_+=F_-$, so no relations. 
Thus $\mathsf{S}_i=\mathsf{S}_i'$ for all $i$. Thus there will be no relations from this. 
This gives an alternative proof that $\mathcal{W}(D^3, \alpha)\cong \mathcal{W}(\alpha )\cong V^{\otimes \underline{\alpha }}$. 
Note that the condition that $W$ is a product on $\alpha $ \textit{always} restricts possible partitions in this way. If $\mathsf{S}_i$ contains 
a component $S_1$ of the surface $S$ bounding a curve $\alpha_1$ in $\alpha $ then $\mathsf{T}_i$ will contain a surface component $S_1'$ 
of $S'$ bounding $\alpha_1$.  
\end{exmp}
 
In the \textit{abstract} or not embedded setting (or equivalently $M=S^3$), all tunneling invariants can be realized with tunneling vector the $0$-vector, for each given pair of $r$-partitions and
$r\leq |S|+|S'|$. In the embedded case there are many restrictions: (i) homological restrictions on possible pairs of partitions, (ii) the tunneling vectors 
are not necessarily zero because the $3$-manifolds have to be embedded in $M\times I$, for each partition entry, and (iii) restrictions because 
of embeddability of components resulting from the partitions. The tunneling vectors are controlled by (ii) and (iii). So it might be necessary to
introduce additional $1$-$2$ handle pairs both to embed components and to embed the components disjointly.
But even if all these conditions are satisfied it is not obvious that every partition can be realized. It follows from the fact that the bordism and 
integral homology groups are isomorphic that \textit{some} partition can always be realized. So assume that there is given a partition satisfying all
the homological conditions on the partition parts. This means that there exists for each partition part, let's say the $i$-th, an embedded 
manifold $W_i\subset M\times I$ bounding the partition, for some tunneling number associated to this part (and all higher tunneling numbers). 
But it is not at all clear when we can choose the tunneling vector such that the components can be embedded disjointly. 
But nevertheless it is the lack of always tunneling vectors $0$ for all partitions, which finally gives the Bar-Natan modules its structure. 

\begin{exmp} Consider $\alpha =\emptyset $ and $M=F\times I$. Consider $S_i=F\times \{i/4\}$ for $i=1,2,3$ and $S$ the disjoint 
union. Choose orientations of $S_i$ such that $S_2$ and $S_3$ are oriented parallel while $S_1$ has the opposite orientation. 
Let $S'=F\times \{1/2\}$. Then the partition $\mathsf{S_1}=\{S_1,S_3\}$, $\mathsf{S}_2=\{S_2\}$, $\mathsf{S}_1'=\{\emptyset \}$ and 
$\mathsf{S}_2'=\{S_2\}$ is abstractly possible, and all homological considerations suggest tunneling is possible. But it is not possible to embed the constant 
bordism on $S_2$ and the nullbordism of $S_1\cup S_3$ into $(F\times I)\times I$. Essentially it does not suffice for $S_1\cup S_3$ to be 
nullhomologous in $M$, it has to be nullhomologous in $M\setminus S_2$. Of course this restriction is too strong in general. 
The general problem here is the following. Given two orientable embedded manifolds $W_1,W_2\subset M\times I$, which are disjoint in the boundary.
When can we change $W_1,W_2$ to $W_1',W_2'$ such that $W_1'\cap W_2'=\emptyset$.  
If we ask this question up to up to relative \textit{oriented} bordism of $W_1,W_2$ to $W_1',W_2'$ then  
this is a general problem that has been considered in \cite{K2}.  We ask for something weaker because the homology classes of the 
$W_i$ can be changed. But for each fixed choice there is defined an intersection invariant in $H_2(M)$. 
If we let $W_1,W_2$ vary through all possible homology classes and consider the resulting set of obstructions in $H_2(M)$ 
we have defined a possible obstruction. 
It seems interesting to study the ambient $4$-dimensional surgery problems arising from this in detail. 
In fact detecting when disjointness of the bordisms in $M\times I$ is accomplished, corresponding to the partition parts, remains the most 
important and difficult aspect of Bar-Natan theory.
\end{exmp}

In the following we assume that $M$ is Haken-irreducible. 
Let $P\subset \mathfrak{t}$ be a well-ordered subset. Then the relations for $\colim(F|\mathcal{C}(M,\alpha )_P)$ according to Corollary \ref{cor=exact}  can all be derived from a labeled directed multi-graph $\Gamma_P$ constructed as follows:   
The vertices of $\Gamma_P$ are in one-to-one correspondence with $P$, each vertex labeled with the sequence of genera of components 
(more precisely a mapping from the set of orderings of the components into sequences of nonnegative numbers, equivariant with respect to permutations).
Next add a directed edge joining a pair of vertices representing incompressible surfaces corresponding to each tunneling invariant between the two surfaces,
directed towards the surface smaller in the total ordering. 
(Thus in general the tunneling graph actually is a multi-graph). 
Call an edge reducible if the corresponding tunneling invariant is reducible.
Suppose that an edge corresponding to $S\to S'$ is reducible by a sequence $S=S_0\to \ldots \to S_r=S'$.
Then none of the edges $S_i\to S_{i+1}$ for $i=0,\ldots r-1$ will reducible by a sequence through the edge $S\to S'$. This follows just from   
the fact that each arrow is descending with respect to the well-ordering.    

Because deleting a reducible edge will not change reducibility or non-reducibility of the remaining 
edges we can define the tunneling graph by omitting all reducible edges from the graph. 
    
\begin{defn} Let $P\subset \mathfrak{t}$ be well-ordered and $M$ Haken-irreducible. Then
the \textit{tunneling graph} $\Gamma :=\Gamma _P(M,\alpha )$ is the labeled directed graph as defined above.
Note that the graph depends on the choice of well-ordering.\end{defn}

Since homology implies bordism and by the results of \cite{KL} it follows that the set of components of the tunneling graph for $(M,\alpha )$ is determined by a homological equivalence as follows. Suppose that the components of two orientable surfaces $S,S'$ bounding $\alpha $ can be oriented in such a way that the resulting homology classes in $H_2(M,\alpha )$ are equal then $S$ and $S'$ are in the same component of the tunneling graph. The argument here is essentially the same as the one for tube equvalency of Seifert surfaces, see also \cite{K2} for a discussion of Seifert surfaces in $3$-manifolds. By abuse of notation we call incompressible surfaces respectively vertices of $\Gamma_P$ \textit{homologous} if they are in the same component of $\Gamma_P$ and the set of all elements in $P$ homologous to each other a homology class in $P$.

Note that the labels of some edge and its adjacent vertices completely determines the pair of morphisms $F_{\pm}: F(C)\rightarrow F(S_{\pm})$.
Here the tunneling invariant corresponding to the edge is induced from some morphism $W: S_+\rightarrow S_-$ and thus 
$S_+\longleftarrow C\longrightarrow S_-$ where $W\subset M\times [-1,1]$, $S_{\pm}=W\cap (M\times \{\mp 1\}$ and $C=M\times \{0\}$.
Let $\Gamma (M,\alpha ):=\Gamma _{\mathfrak{t}}(M,\alpha )$. Thus, at least if $M$ is Haken-irreducible, all the geometric input necessary to calculate the Bar-Natan module is contained in the tunneling graph. Given the graph  what remains is to work out the \textit{general} combinatorics how this information translates into the structure of the module.  We already know that for $(M,\alpha )$ Haken-irreducible the corresponding module has a presentation as a quotient of $\oplus_{S\in P}V^{\otimes \underline{S}}$ by 
the submodule defined from the image $\oplus_eV^{\otimes r(e)}\rightarrow V^{\otimes \underline{e_+}}\oplus V^{\otimes \underline{e_-}}\subset \oplus_{S\in P}V^{\otimes \underline{S}}$, where $e_{\pm}$ are denote the two isotopy classes of incompressible surfaces corresponding to the 
vertices $e_{\pm }$ adjacent to $e$. The sum runs through all edges, the morphisms are defined as $F_{\pm}$ defined from the combinatorics of the tunneling 
labels of vertices and edges as described above, where $e$ is the edge of some $r(e)$-tunneling invariant.

\vskip .1in

\noindent \textbf{Question.} Which directed labeled graphs in the above way are the tunneling graph of some $(M,\alpha )$?

\vskip .1in

The mapping class group $\textrm{Diff}(M,\partial M)$ of a $3$-manifold $M$ relative to the boundary acts on all tunneling graphs $\Gamma (M,\alpha )$ for all $\alpha $. Much is known 
about $\textrm{Diff}(M,\partial M)$, \cite{H}. In particular for simple Seifert fibred manifolds the structure is well-understood, see \cite{AF}.

\section{From tunneling graphs to Bar-Natan modules}\label{sec=to Bar-Natan}

Let us start with a discussion of the $R$-module structure in some few cases.

Consider the \textit{Khovanov Frobenius algebra} $V=V_K:=R[x]/(x^2)$ with $\mathfrak{k}$ defined by multiplication by $2x$. 
As $R$-module we have $V\cong R\oplus R\cdot x$.
Thus $\mathfrak{k}^2=0$ and $\ker(\mathfrak{k}^g)=V$ for $g\geq 2$. If $R$ has no $2$-torsion then $\ker(\mathfrak{k})=R\cdot x$.
Components colored with $1$ are called \textit{white} components following \cite{AF}.
A surface with only white components is called a white surface.  
Using this language and identifying generators of $V$ with colored surfaces we can say: $d_n(1)$ is a linear combination of surfaces, where 
$n-1$ dots are distributed in all possible ways, $d_n(x)$ is the single surface with all components dotted. 
$d_n(\mathfrak{k}(1))$ is twice $d_n(x)$ and $d_n(\mathfrak{k}(x))=0$. If any higher powers of $\mathfrak{k}$ appear the terms are zero.
So we have the following result:
If a tunneling vector contains any component $\tau_i>1$ then \textit{both} $F_{\pm}$ are zero. So if for a pair $(S,S')$ all
admissible tunneling vectors have a component $>1$ the two incompressible surfaces are linearly independent.
 
It is now easy to give \textit{global} arguments for some of Asaeda and Frohman's techniques in the case of the Khovanov Frobenius algebra.
We see that the collection of 
white incompressible surfaces is linearly independent if $M$ is Haken-incompressible. In fact if $M$ is Haken-incompressible no
homorphism $F_{\pm}$ for given Heegaard bordisms between $S$ and $S'$ can have a white surface in the image, except it is coming 
from some isotopy loop, in which case we keep the isotopy class of the colored surface.  Also the direct  \textit{projection arguments}
given by them are nicely explained. The possible tunneling morphisms are always by tensor products composed of 
contributions of the form $d_k\circ \mathfrak{k}^{\tau +g}$, where $\tau $ is the corresponding component of the tunneling vector, and
$g=0$ or the absolute value of the total genus difference of the surfaces in the two partitions, and $k$ is the number of surfaces 
in the corresponding partition set. 

Let us assume that $S,S'$ are two oriented surfaces. We consider $1$-partitions.  We will see that in this case only (i), (ii) or (iii) below appear. 
Note that this occurs if the partition is admissible.  Let $F_+: V\rightarrow F(S')=V^{\otimes n}$ and  $F_-:V\rightarrow F(S)=V^{\otimes k}$. These homomorphisms appear inside a tensor product 
if the $1$-partition is contained in an $r$-partition. Because of the above we can assume that $\tau =0$ or $\tau =1$.  
Note that if  $\tau =0$ occurs also the relation for $\tau =1$ appears. In the following we assume that $g(S)\geq g(S')$ so $F_+$ is the homomorphism
into the module associated to the lower genus surface. 

\noindent (i) Let $g(S)-g(S')>1$. Then we have $F_+=0$ and $F_-=d_k\circ \mathfrak{k}^{\tau }$.  If $\tau =1$ then the all dotted surface $S$ is zero. 
If $\tau =0$ we get the relation that the sum of components of $S$ with all but one component dotted is zero, and $S$ with all components dotted is zero. 

\noindent (ii) Let $g(S)-g(S')=1$. If $\tau =1$ then $F_+ =0$ and $F_-=d_k\circ \mathfrak{k}$. So we get a relation that twice the surface $S$ with all components 
dotted is $0$. If $\tau =0$ then $F_+=d_n\circ \mathfrak{k}$ and $F_-=d_k$. So we have first the relation that $2$ times the surfaces $S'$ with all dotted components is equal to the linear combination of surfaces $S$ with all but one component dotted. Furthermore we have the relation that the all dotted surface $S$ is $0$. 

\noindent (iii) Let $g(S)-g(S')=0$. If $\tau =1$ then $2$ times the surfaces with all dotted components are equal. If already $\tau =0$ occurs then we have that the linear combinations with all but one component dotted are equal, and the surfaces with all components dotted are equal.
 
The relation $\mathfrak{k}^2=0$ implies that the relations resulting from a particular part in a partition set are described by one of (i), (ii), (iii) above.    

Many calculations are possible in the basis of the special Khovanov Frobenius algebra, which are very difficult in general. 
Even in the case of the universal quadratic Frobenius algebra $V=R[x]/(x^2-hx-t)$ with $R=\mathbb{Z}[h,t]$ with 
$\Delta (1)=1\otimes x+x\otimes 1+h1\otimes 1$ and $\varepsilon (1)=0$, $\varepsilon (x)=1$ it follows that 
$\mathfrak{k}=2x+h$, which is easily seen \textit{not} to be nilpotent. This implies that tunneling phenomena 
of all orders are built into the structure. The algebra involved in calculating Bar-Natan modules for the case of general Frobenius algebras seems more difficult.

Finally we state some special technique about tensor products and push-outs in our setting.
We work in a linear category $\mathcal{L}$ as in section \ref{sec=algebra}. For a diagram
$V \buildrel \psi \over \longleftarrow U \buildrel \varphi \over \longrightarrow W$ let $P(\varphi ,\psi )$
denote the pushout and let $Q(\varphi ,\psi )$ denote the quotient module $(V\oplus W)/P(\varphi ,\psi )$ defined the pushout.
In the following we will always denote $V/(\varphi )$ to denote the quotient of some $R$-module $V$ 
by the image of some $R$-homomorphism $\varphi :U\rightarrow V$. 
First note that if $\psi =\rho \circ \varphi$ then there is a natural isomorphism:
$$Q(\varphi ,\rho \circ \varphi)\cong V/ \varphi \oplus W$$
induced from the epimorphism
$$V\oplus W\rightarrow V/ \varphi \oplus W$$
by $(v,w)\mapsto ([v],w-\rho (v))$, which has kernel precisely $P(\varphi ,\psi )$. 

Given $V_i \buildrel \psi_i \over \longleftarrow U_i \buildrel \varphi_i \over \longrightarrow W_i$ for $i=1,2$ note that 
$P(\varphi_1\otimes \varphi_2, \psi_1\otimes \psi_2)\subset (V_1\otimes V_2)\oplus (W_1\otimes W_2)$ while 
$P(\varphi_1,\psi_1)\otimes P(\varphi_2,\psi_2)\subset (V_1\oplus W_1)\otimes (V_2\oplus W_2)\subset (V_1\otimes V_2)\oplus (W_1\otimes W_2)\oplus 
(V_1\otimes W_2)\oplus (W_1\otimes V_2)$.  But we also have $P(\varphi_1\otimes \psi_2, \psi_1\otimes \varphi_2)\subset (V_1\otimes W_2)\oplus (W_1\otimes V_2)$. It is not hard to see that the sequence:
$$P(\varphi_1\otimes \psi_2,\psi_1\otimes \varphi_2)\subset P(\varphi_1,\psi_1)\otimes P(\varphi_2,\psi_2)\twoheadrightarrow P(\varphi_1\otimes \varphi_2,\psi_1\otimes \psi_2)$$
is exact, with the right hand homomorphism defined by restriction of the projection. 
Taking the quotient of the standard exact sequence 
$$(V_1\otimes W_2)\oplus (W_1\otimes V_2)\subset (V_1\oplus W_1)\otimes (V_2\otimes W_2)\twoheadrightarrow (V_1\otimes V_2)\oplus (W_1\otimes W_2)$$
we have the short exact sequence: 
$$Q(\varphi_1\otimes \psi_2,\psi_1\otimes \varphi_2) \hookrightarrow Q(\varphi_1,\psi_1)\otimes Q(\varphi_2,\psi_2)$$
$$\twoheadrightarrow Q(\varphi_1\otimes \varphi_2,\psi_1\otimes \psi_2).$$
If we apply this in the case $\psi_1=\rho_1\circ \varphi_1$ and $\phi_2=\rho_2\circ \psi_2$ we have the short exact sequence:
$$(V_1\otimes W_2)/(\varphi_1\otimes \psi_2) \oplus (W_1\otimes V_2)\hookrightarrow (V_1/\varphi_1\oplus W_1)\otimes (V_2\oplus W_2/\psi_2)$$
$$\twoheadrightarrow Q(\varphi_1\otimes (\rho_1\circ \varphi_1),(\rho_2\circ \psi_2)\otimes \psi_2),$$
and using $(V_1\otimes W_2)/(\varphi_1\otimes \psi_2)\cong V_1/\varphi_1\oplus W_2/\psi_2$ we get from the sequence above the isomorphism
$$ Q(\varphi_1\otimes (\rho_1\circ \varphi_1),(\rho_2\circ \psi_2)\otimes \psi_2)\cong ((V/\varphi_1)\otimes V_2)\oplus (W_1\otimes (W_2/\psi_2))$$
A standard situtation where this can be applied is for a pair $F_{\pm}$ as follows:
$$
\begin{CD}
V\otimes V @<\mathfrak{k}^{g_1}\otimes \mathfrak{k}^{\tau_2+g_2} << V\otimes V @>\mathfrak{k}^{\tau_1+g_1}\otimes \mathfrak{k}^{g_2}>> V\otimes V
\end{CD}
$$
with $g_1,g_2\geq 0$ integers, where we get for the corresponding quotient module defined from the pushout:
$$Q\cong (V/\mathfrak{k}^{\tau_1})\otimes V)\oplus (V\otimes (V/\mathfrak{k}^{\tau_2})$$

Assume that $M$ is Haken-irreducible. 
Recall that the basic pushout quotient to be computed is $P(d_n\circ \mathfrak{k}^{\tau },d_m\circ \mathfrak{k}^{g+\tau })$ for the morphism
$V^{\otimes n}\longleftarrow V\longrightarrow V^{\otimes m}$ with $n\geq 1$ and $m,n,\tau$ nonnegative integers.
This is the quotient defined from $W: S\leftarrow C\rightarrow S'$ with $W$ connected, $g(S)\geq g(S')$ and $g:=g(S)-g(S')$ and $\tau $ a minimal
tunneling number.  
This is not easy in general for the category of $V^{\otimes j}$-modules if $j>0$.
We will assume that $R$ is a field and only discuss the $R$-module structure. By generalizing the morphism above 
first choose for $n\geq 1$ splitting homomorphisms $\sigma_n: V^{\otimes n}\rightarrow V$ of the injective morphism $d_n: V\rightarrow V^{\otimes n}$. 
Then define the $R$-epimorphism 
$$V^{\otimes  n}\oplus V^{\otimes m}\rightarrow V^{\otimes n}/(d_n\circ \mathfrak{k}^{\tau })\oplus V^{\otimes m}$$
by
$$(x,y)\mapsto ([x],y-d_m\circ \mathfrak{k}^g\circ \sigma_n(x))$$
If $(x,y)$ is in the kernel of this homomorphism then $[x]=0$ thus $x=d_n\circ \mathfrak{k}^{\tau }x'$ for some $x'\in V$. Thus 
$\sigma_n(x)=\mathfrak{k}^{\tau }x'$. It follows that $y=d_m\circ \mathfrak{k}^g\circ \mathfrak{k}^{\tau }x'$. Thus 
$(x,y)\in P(d_n\circ \mathfrak{k}^{\tau },d_m\circ \mathfrak{k}^{g+\tau })$, so we have established the $R$-isomorphism
$$Q(d_n\circ \mathfrak{k}^{\tau },d_m\circ \mathfrak{k}^{g+\tau })\cong V^{\otimes n}/(d_n\circ \mathfrak{k}^{\tau })\oplus V^{\otimes m}$$
Observe that this does not depend on the genus difference $g$
but the isomorphism identifying with the quotient does. This makes it difficult to use the result inductively. 

Suppose that $R$ is a field. 
To understand the problems in determining Bar-Natan modules consider the inductive calculation according to Theorem \ref{thm=minimal terminal} in our case.
So suppose that we have calculated $W(S)_-$ for some incompressible surface $S$. Then $W(S)$ is the cokernel of some
epimorphism as follows:
$$\oplus_{i\in I}V_i\rightarrow V^{\otimes \underline{S}}\oplus W(S)_-,$$
where $i$ runs through \textit{all} edges corresponding to tunneling invariants from $S$ to $S'\leq S$. Thus $V_i=F(C)$ for some 
$S\longleftarrow C\longrightarrow S'$.
Now contributions from $S$ to $S$ for the same input and output partition are easily seen to be 
consequences of relations from loop isotopies. But already different partitions of the same surface $S$ could underly 
tunneling invariants giving rise to different relations (just imagine partitions $\mathsf{T_1}=\{S_1,S_2\}$, $\mathsf{T_2}=\{S_3\}$, and
$\mathsf{S}_1=\{S_1\}$, $\mathsf{S}_2=\{S_2,S_3\}$).   
In general, $|I|\geq 2$ will give rise to many
\textit{back relations}. More precisely, if the tunneling graph were just a tree we could easily construct an $R$-module basis inductively 
easily by eliminating part that subspace of $F(S)$, which is in the image of $F_+$, using the relations $F_+(v)=F_-(v)$ for $v\in F(C)$. 
But of course in general there are many, possibly infinitely many, different tunneling invariants because $W(S)_-$ is a quotient
of $\oplus_{S'<S}F(S)$.
But if we have distinct parallel tunneling invariants (or any two different paths in the tunneling graph from $S$ to some $S'\leq S$) then 
a priori there could be relations $F_+v=F_+' v'\in F(S)$ but $F_-v\neq F_-v'$ with $v\in F(C)$ and $v'\in F(C')$ where we have 
$S\longleftarrow C\longrightarrow S'$ and $S\longleftarrow C'\longrightarrow S'$.   
Of course more complicated situations can occur where one or both relations are not directly geometric but coming from 
an algebraic pullback, so a composition along two edges of the tunneling graph.  
Note that the resulting edge will not be in the tunnel graph but \textit{algebraically} still have the same form considered 
above, i.\ e.\ composed of tensor products of $d_n\circ \mathfrak{k}^{\tau }$ respectively 
$d_m\circ \mathfrak{k}^{\tau +g}$. Of course in general it is not clear whether the genus difference appears on the $+$ or on the 
$-$-side. It appears in general on the side with the smaller genus of the pacakage.   
Situations where we can choose the order such that this appears uniformly throughout the whole tunneling graph are 
combinatorially much simpler.  

\begin{defn} A tunneling invariant for $S_+\geq S_-$ is called \textit{descending} if in the induced pair $F_{\pm }$ 
of homomorphisms each tensor factor is of the form above, i.\ e.\ of the form $F_+=d_n\circ \mathfrak{k}^{\tau }$ and 
$F_-=d_m\circ \mathfrak{k}^{\tau +g}$ with $n\geq 1$ and $m,g,\tau $ nonnegative integers. The important assumption here 
is that the genus differences of the tunneling invariant are $\geq 0$ for \textit{all parts}. \end{defn}

\begin{defn}Let $P\subset \mathfrak{t}$ be an ordered set. Then the tunneling graph $\Gamma_P(M,\alpha )$ is called
\textit{descending} if all tunneling invariants in the graph are descending. \end{defn}

Suppose that the tunneling graph is descending. 
We will see that in some special situations back relations do not contribute.
But in general there are many \textit{distinct} 
partitions underlying parallel tunneling invariants. (just think about $\alpha =\emptyset $ and $M=\Sigma \times I$,
$\Sigma $ a closed oriented surface, see section 10).
If we have parallel $r$-tunneling invariants but with the same partitions on both \textit{input} and \textit{output} then the difference is 
only in the tunneling vectors, and the 
result follows just as for $r=1$. But it is possible to have the same input partition but different output partitions, and a priori 
there seems to be no reason why this should not give relations from $F_+(v)=F_+'(v)$ but with $F_-(v)\neq F_-'(v)$.
Note that for $V$ manifold-induced if $F_+(v)=F_+'(v')$, even with $F_+'\neq F_+$ one can conclude that there exist morphisms 
$\lambda : V''\rightarrow V$ and $\lambda ': V''\rightarrow V'$ such that $F_+\circ \lambda =F_+'\circ \lambda '$. 

\begin{thm} Suppose $P\subset \mathfrak{t}$ can be ordered such that the resulting tunneling graph is descending. 
Furthermore suppose that there are no back-relations in the inductive calculation using the order of $P$.  
Then the $R$-module $\mathcal{W}_P(M,\alpha )$ does not depend on genus differences at the vertices but only on the partitions and 
tunneling vectors.\end{thm}

\begin{proof} This is proved by some inductive application of the method above following Theorem \ref{thm=minimal terminal}. 
Suppose we add an incompressible surface $S$ in the order. Then we have relations of the form above for each $S\leftarrow W\rightarrow S'$ realizing a 
tunnel-invariant connecting $S$ to $S'$ with $S\geq S'$. For $S=S'$ there can be loop relation, which are permutations of the tensor factors and thus do not 
depend on the genera of components of $S,S'$. Otherwise we have tensor products of relations realized by $W$ connected.
If we let $W(S)_-$ denote the local module on all incompressible surfaces with $S'<S$. Then in the exact sequence in 5.7 we
see that the relation for some $r$-tunneling invariant has the form $V^r \rightarrow F(S) \oplus W(S)_- $, and this is a tensor product 
of $V\rightarrow V^{\otimes n}\oplus W(S)_-$ with $V^{\otimes n}\subset F(S)$. Thus if there are no multiple edges in the tunnel graph 
it is easy to see that we can define a direct sum of epimorphisms similarly to the above with the direct sum over all edges originating from 
$S$ to $S'<S$. If there are multiple edges between $S$ and $S'$ for some specific $S'$ then order the edges in some way. 
Then proceed first as before. This will replace $F(S)$ by some quotient $F(S)/R$. Choose a linear complement $F(S)^{(2)}$ of $R$ in $F(S)$ and replace 
$F_+$ for the next edge by the usual $F_+$ composed with the projection onto $F(S)^{(2)}$. Proceed with all multiple edges between $S$ and some $S'$ in this way. Because there are no back relations this will construct a basis of $\mathcal{W}_P(M,\alpha )$.  
\end{proof}

We will see in section \ref{sec=homology} that a typical example for a descending tunnel graph is with $\alpha =\emptyset$ and $M=\Sigma \times I$ for 
$\Sigma $ a closed oriented surface. Even though there are back-relations they won't contribute in this case.  
Not that for this example all incompressible surfaces are parallel to $\Sigma \times \{\frac{1}{2}\}$.
The surfaces can be ordered in such way that surfaces with a smaller number of components are $<$ than surfaces with a larger number of components.
The corresponding tunneling graph will be descending. In \cite{BK} (unpublished) the algebra structure of $\mathcal{W}(\Sigma \times I)$ has been determined
and does depend on the genus of $\Sigma $, and thus on the genus differences in the tunneling graph. But there is an inductive way to construct
an $R$-module basis of the infinite dimensional Bar-Natan module independent of genera. 

The comments in this section clearly show that the homomorphism $\mathfrak{k}$ and the combinatorics of partitions are the essential algebraic ingredient in understanding the origin of the structure of Bar-Natan skein modules from the tunneling graph. Even if the tunneling graph is known 
it seems in general quite difficult to determine the modules explicitly, even over $R$. 
But our results shows precisely \textit{what} kind geometric information is reflected in the relations of the Bar-Natan module
and \textit{in which way}. 

\vskip .3in

\section{Local Bar-Natan skein theory}\label{sec=local}

\vskip .1in

In section \ref{sec=presentations} the importance of using a suitable order on a terminal set $\mathfrak{t}$ or a subset $\mathcal{S}$ of it for the calculation of the 
colimit module of a functor has been emphasized.
In principle any order could be used but the choice of subset usually suggests specific choices in order to understand how the 
module structure is related with the geometry $(M,\alpha )$.  
 
Throughout this section we assume that $M$ is irreducible and connected. Let $\mathfrak{t}$ be the set of isotopy classes of incompressible surfaces in $M$
bounding $\alpha $.
Let $\mathcal{S}:=\mathcal{S}(M,\alpha )$ be the set isotopy classes of incompressible surfaces in $M$ bounding $\alpha $ without any parallel 
\textit{closed} components. If $\alpha =\emptyset$ then the empty surface is an element of $\mathcal{S}$. 

There is the forget function $\mathfrak{p}: \mathfrak{t}\rightarrow \mathcal{S}$ defined by mapping the isotopy class of an incompressible surface bounding $\alpha $ to some element of $\mathcal{S}$ by forgetting from each subsurface of parallel components all except one. 
We will study $P$-local modules for certain subsets $\mathcal{P}\subset \mathcal{S}$ with $P:=\mathfrak{p}^{-1}(\mathcal{P})$.
If an ordering has been chosen on $\mathcal{P}$ there can be defined an ordering on $P$ as follows: First order the elements of $P$ according to their images in $\mathcal{P}$. Then order the elements of $\mathfrak{p}^{-1}(S)$ for some $S\in \mathcal{P}$ using an arbitrary natural order induced from a tubular neighborhood $S\times I\subset M$ (this can be made unique by specifying an orientation of $S$ and $M$).  

A partial ordering of the set $\mathcal{S}$ can be defined as follows:

\begin{defn} For each surface $S$ we define the \textit{genus sequence} of the surface as the following tuple of 
nonnegative numbers $g_1,g_2,g_3,\ldots ,g_r$: $g_1$ is the genus of the surface $S$. Then we consider all unions of $|S|-1$ components of $S$ and 
putting the corresponding genera into lexicographic order. This is followed by all unions of $|S|-2$ components of $S$ and putting the corresponding genera 
into lexicographic order. Finally we consider all connected components and their genera, put into lexicographic order. 
\end{defn}

Note that for each surface the length of its genus sequence is $2^{|S|}-1$. The genus sequence is completely determined by the
sequence of genera of its components. But ordering lexicographically by components the surfaces is in general very different from ordering with respect to the genus sequence as in the definition.

Let $g_b(S)$ be the genus sequence defined from the bounded subsurface of $S$, i.\ e.\ the union of components with nonempty boundary, and similarly 
define $g_c(S)$ defined from the closed components.  

\begin{exmp} Consider $M=S^1\times D^2$ with $\alpha $ defined by $2n$ parallel nontrivial curves in the boundary not bounding disks in $M$. The incompressible surfaces in this case all have the same genus sequences. But note that the different isotopy classes of orientable incompressible surfaces correspond to crossingless matchings of $2n$ elements, see \cite{R}. In \cite{R} a particular order had to be chosen in order to derive a suitable presentation of $\mathcal{W}(M,\alpha )$ for $V=V_K$.
\end{exmp}

Consider a set $\mathfrak{p}^{-1}(S)$ for some $S\in \mathcal{S}$. The elements in this set only differ by the number of parallel components 
of closed components. 
Let $S_1,\ldots ,S_{\ell }$ be the closed components of $S$ representing some element in $\mathcal{S}$ ordered in some way but such that
the genera are nondecreasing $g_1\leq  g_2\leq \ldots \leq g_{\ell }$.
Given $T\in \mathfrak{p}^{-1}(S)$ let $\mathfrak{k}=(k_1,k_2,\ldots ,k_{\ell})$ denote the sequence of positive integers defined by $k_i$, which is the number of components
in $T$ parallel to $S_i$ for $1\leq i\leq \ell$.   
In this way we get a linear order on $\mathfrak{t}$
from a chosen order on $\mathcal{S}$.
Then order each of the sets $\mathfrak{p}^{-1}(S)$ by lexicographic order of the sequences. 

It is easy to determine the vector space $F(T)$. Let $d_b$ be the the number of components with nonempty boundary in 
$S:=\mathfrak{p}(T)$, i.\ e.\ the length of $g_b$.
Let $d:=\sum_{i=1}^{\ell}k_i$.  
Then $F(T)=V^{\otimes d_b}\otimes V^{\otimes d}$. 

Note that in the case that $\mathcal{P}$ does not contain any surfaces with closed components we have that $\mathfrak{p}$ is a one-to-one correspondence $\mathfrak{p}^{-1}(\mathcal{P})\rightarrow \mathcal{P}$. This holds in particular for $\mathcal{P}=\mathcal{P}_b$ the set of all surfaces in $S$ without closed components. So we can define a local Bar-Natan module for surfaces without closed components. This module is of course trivial if $\alpha =\emptyset$.
But it has quite interesting structure if $|\alpha |\geq 1$. In this case we do not distinguish between $\mathcal{P}$ and $P$ and identify
$P\subset \mathcal{S}$. 

\begin{exmp} Suppose that $P$ consists of a single surface $S\in \mathcal{S}$ without closed components. We assume that $|\alpha |\geq 1$. 
Recall that if a surface component in a partition set $\mathsf{S}_i$ contains a certain component of $\alpha $ then $\mathsf{T}_i$ contains a component
bounding the same component. In the case $S=S'$ this implies that $\mathsf{T}_i=\mathsf{S}_i$ for $i=1,\ldots ,r$. It follows that there are \textit{no
nontrivial tunneling relations} in this case. Then $\mathcal{W}_{\{S\}}\cong V^{\otimes \underline{S}}$. Note that no nontrivial isotopy loop relations 
are possible in the \textit{all bounded} case. 
\end{exmp}

\begin{exmp} Suppose that $P=\{S,S'\}$ with $S,S'$ two nonisotopic surfaces bounding $\alpha $ with $|\alpha |\geq 1$. We can assume that $g(S)\geq g(S')$. We have an exact sequence:
$$\oplus_{\mathfrak{p}}V^{\otimes r(\mathfrak{p})}\rightarrow  V^{\underline{S}}\oplus V^{\underline{S'}}\twoheadrightarrow \mathcal{W}_{\{S,S'\}}=:\mathcal{W},$$
where the left hand sum runs through all admissible tunneling invariants $\mathfrak{p}$ from $S$ to $S'$, and $r(\mathfrak{p})$ is the length of the two partitions, 
defined by the tunneling invariant $\mathfrak{p}$. 
Let $\tau (\mathfrak{p})=(\tau_1,\ldots ,\tau_r)$ denote the tunneling vector associated to the partition $\mathfrak{p}$. 
Consider the possible tunneling invariants for a fixed pair of $r$-partitions.
By ordering the components of $S,S'$ according to the partitions we can write
$$F_+=d_{n_1}\mathfrak{k}^{\tau_1+\lambda_1}\otimes \ldots \otimes d_{n_r}\mathfrak{k}^{\tau_r+\lambda_r}: V^{\otimes r}\rightarrow V^{|S'|}$$
respectively
$$F_-=d_{m_1}\mathfrak{k}^{\tau_1+\mu_1}\otimes \ldots \otimes d_{m_r}\mathfrak{k}^{\tau_r+\mu_r}: V^{\otimes r}\rightarrow V^{|S|}.$$
The numbers $\mu_i,\lambda_i$ are defined as follows:
For each $i$ let $g_i$ respectively $g_i'$ denote the genus of the components of $S$ respectively $S'$ in $\mathsf{S}_i$ respectively 
$\mathsf{T}_i$. for $i=1,\ldots ,r$. Then  if $g_i\geq g_i'$ we have $\lambda_i=g_i-g_i'$ and $\mu_i=0$. If $g_i\leq g_i'$ we have 
$\lambda_i=0$ and $\mu_i=g_i'-g_i$.  
For each specific tunneling invariant we can compute the pushout and resulting quotient module $Q(F_+,F_-)$ using the results from section 7. 
It is obvious that the calculation can become very involved if many tunneling invariants are admissible. We consider the most simple nontrivial case. 

Let $|\alpha |=2$ such that $|S|,|S'|\leq 2$, and still assume $g(S)\geq g(S')$. If $|S|=|S'|=1$ there is only the $1$-partition. In this case there is a unique minimal 
tunneling number $\tau $ and $g=g_1\geq g'=g_1'$. Then $F_+=\mathfrak{k}^{g_1-g_1'+\tau }$ and $F_-=\mathfrak{k}^{\tau }$. Thus
$\mathcal{W}\cong V\oplus V/\mathfrak{k}^{\tau }$. 
Note that $\tau \geq 1$ because otherwise $S$ would be compressible. (In particular for $V=V_K$ and $R$ a field,
if $\tau =1$ then $\mathcal{W}$ has dimension 3, otherwise it has dimension $4$.)
If $|S'|=2$ and $|S'|=1$ then $F_+$ is changed to $d_2\circ \mathfrak{k}^{g_1-g_1'+\tau}: V\rightarrow V\otimes V$ and we  have 
$\mathcal{W}\cong V^{\otimes 2} \oplus V/\mathfrak{k}^{\tau }$. We have $\tau \geq 1$ because otherwise $S$ is compressible.
(If $V=V_K$ and $R$ a field then $\mathcal{V}$ has dimension $5$ respectively $6$ if $\tau =1$ respectively $\tau >1$.)
If $|S|=2$ and $|S'|=1$ then $F_-=d_2\circ \mathfrak{k}^{\tau }$ and $F_+=\mathfrak{k}^{g_1-g_1'}\circ \mathfrak{k}^{\tau }$.
By the results from section \ref{sec=to Bar-Natan} we get the $R$-module isomorphism $\mathcal{W}\cong V\oplus V/\mathfrak{k}^{\tau}\oplus (V\otimes V)/d_2$.
In this case $\tau \geq 1$, if $g(S)=g(S')$ because otherwise we would get the surface $S$ by a separating neck cutting of the incompressible surface 
$S'$, which is not possible. (For $V=V_K$ and $R$ a field we have that the dimension of $\mathcal{W}$ is $5$ if $\tau =1$, and is $6$ if $\tau >1$. In the second case there is no relation while in the first case the dotted surface $S'$ is the sum of two copies of $S$ with a dot on one of the components.)
If $g(S)>g(S')$ then also $\tau =0$ is possible. (For $V=V_K$ and $R$ a field in this last case the dimension can drop to $4$. This is because also 
relations from white surfaces contribute.)
The most interesting case is $|S|=|S'|=2$ because in this case two partitions can be possible.
Suppose that only the $1$-partition is admissible. Then we get $F_+=d_2\circ \mathfrak{k}^{\tau +g-g'}=(\textrm{Id}\otimes \mathfrak{k}^{g-g'})\circ (d_2\circ \mathfrak{k}^{\tau })$ and $F_-=d_2\circ \mathfrak{k}^{\tau }$. We know from section \ref{sec=to Bar-Natan} that for $R$ a field we have the $R$-module isomorphism
$$\mathcal{W}\cong V^{\otimes 2}\oplus V^{\otimes 2}/(d_2\circ \mathfrak{k}^{\tau })$$
(For $V=V_K$ and $R$ a field we have if $\tau =0$ then the dimension is $6$, otherwise the dimension is $7$.)

Suppose that only the $2$-partition is admissible. There are two different possibilities for tunneling invariants in this case.
Let $(\tau_1,\tau_2)$ be the tunneling vector.  
Let $g_1\leq g_2$ be the genera of the two components of $S$, and let $g_1',g_2'$ be the corresponding genera of the components of $S'$
related by the partition. Suppose first $g_1\geq g_1'$ and $g_2\geq g_2'$. Then $F_+=\mathfrak{k}^{\tau_1 +g_1-g_1'}\otimes \mathfrak{k}^{\tau_2+g_2-g_2'}$
and $F_-=\mathfrak{k}^{\tau _1}\otimes \mathfrak{k}^{\tau_2}$, thus $\mathcal{V}\cong V^{\otimes 2}\oplus V^{\otimes 2}/(\mathfrak{k}^{\tau_1}\otimes \mathfrak{k}^{\tau_2})$. The other possibility (at least up to numbering the partition sets) is $g_1\geq g_1'$ and $g_2<g_2'$. 
Then $F_+=\mathfrak{k}^{\tau_1+g_1-g_1'}\otimes \mathfrak{k}^{\tau_2}$ and $F_-=\mathfrak{k}^{\tau_1}\otimes \mathfrak{k}^{\tau_2+g_2'-g_2}$.  
Here we have $\mathcal{V}\cong (V/\mathfrak{k}^{\tau_1}\otimes V)\oplus (V\otimes V/\mathfrak{k}^{\tau_2})$. This is the nondescending 
case. 
Note that in all cases above the $R$-module structure for $R$ a field does \textit{not} depend on the genera differences, it only depends on the tunneling
vectors. 
\end{exmp}

\begin{exmp} Suppose that $|\alpha |\geq 1$ and let $P$ be the
set of isotopy classes of \textit{connected} incompressible surfaces bounding $\alpha $.  In particular these surfaces have no closed components. Suppose we consider the set of isotopy classes of incompressible surfaces in a fixed component $\Gamma_{P'}$ of $\Gamma_P$, i.\ e.\ let $P'$ be a homology class of $P$.
We consider the $P'$-local module $\mathcal{W}_{P'}(M;\alpha)$ based on connected incompressible surfaces bounding $\alpha $ in the component $\Gamma_{P'}$ of $\Gamma_P$.
Note that $\mathcal{W}_P(M; \alpha )=\oplus_{P'}\mathcal{W}_{P'}(M,\alpha )$.
Next let $g_1\leq g_2\leq \ldots $ be the possible genera of surfaces
in $P'$. We can choose an order of the elements of $P'$ such that $S<S'$ implies $g(S)\leq g(S')$.
\end{exmp}
 
We next determine the possible tunneling invariants. There exist morphisms $W: S\leftarrow C\rightarrow S'$ for all possible choices 
of $S,S'$ because we assume that $S,S'$ are in the same homology class. It suffices to consider the minimal bordisms. Recall from section that in the case $S\geq S'$ 
the pair of morphisms $F_{\pm }$ is of the form
$F_+=\mathfrak{k}^{g(S)-g(S')+\tau}, F_-=\mathfrak{k}^{\tau }$ and $\tau =\tau (S,S')$ is the \textit{relative tunneling number}. Here we consider $S\leftarrow C\rightarrow S'$ giving rise to $F_+: V=F(C)\rightarrow F(S')=V$ and $F_-:V=F(C)\rightarrow F(S)=V$.   
Now for a surface $S$ define $\rho (S):=\min\{\tau(S,S')\}$, where the minimum is over all surfaces $S'$ with $S\geq S'$. 
We call this the \textit{tunneling number} of $S$. 

\begin{thm} Suppose $R$ is a field. Let $M$ be a connected oriented $3$-manifold and $\alpha $ an oriented nonempty $1$-manifold in 
$\partial M$. Let $P$ be the set of connected incompressible surfaces bounding $\alpha $. Then the local Bar-Natan skein module of a homology class $P'$ of $P$ is isomorphic as $R$-module
$$\mathcal{W}_{P'}(M;\alpha )=\bigoplus_{S\in P'}V/\mathfrak{k}^{\rho (S)},$$
with $\rho (S)$ the tunneling number of $S$ as defined above, and the isomorphism is an isomorphism of $V^{\otimes \underline{\alpha }}$-modules.\end{thm}

\begin{proof} The tunneling graph of $P'$ is the connected graph with a vertex representing a minimal genus connected surface in the component $\gamma_{P'}$. We can change the graph by omitting multiple edges because the relation defined from an oriented edge always has the 
form $F_+(v)=F_-(v)$ with $F_+=\mathfrak{k}^{\tau }$ and $F_-=\mathfrak{k}^{g+\tau }$ and $g$ the genus difference. Here $\tau $ is the tunneling number of the edge
and $g$ is the nonnegative difference of the genera of the vertices adjacent to the edge.  For topological reasons no tunneling number can be 
$0$. The rest is induction following Corollary \ref{cor=exact} to construct an $R$-module basis for the module. If the tunneling graph would be a tree the result would
follow immediately. But $\tau $ only enters the relations in a relative way, and if we consider two paths from a vertex to a lower order surface
the genus difference does not depend on the path. This means that there are no additional relations on the endpoint of the path coming from 
the loop in the tunneling graph. 
\end{proof}

\vskip .1in

Note that if $M$ is small, i.\ e.\ it has no closed incompressible surfaces, and $|\alpha |=1$ then 
$\mathcal{W} (M,\alpha )=\mathcal{W}_P(M,\alpha )$. Another typical nontrivial example for $(M,\alpha )$ and considering $P$ is constructed from a nullhomologous 
link $K\subset N$ with $M=N\setminus N(K)$ where $N(K)$ is an open tubular neighborhood of $K$ with $\alpha $ a union of 
longitudes of $K$ bounded by a Seifert surface. Recall that Seifert surfaces are connected by definition, see \cite{K2}.  
A discussion of different homology classes of incompressible Seifert surfaces bounding $K$ is given in detail in
in \cite{K2}.

We should remark that $\rho (S)>0$ for all surfaces with minimal genus which are not isotopic
to the minimal genus surface chosen. But $\rho (S)>0$ also for all surfaces not of minimal genus.
In fact if $\tau (S,S')=0$ for $g(S)>g(S')$ we could attach $1$-handles to $S'$ to get a surface isotopic to $S$, 
or alternatively we could attach $2$-handles to $S'$ , i.\ e.\ do neck-cutting to get the surface $S$. But $S'$ is incompressible,
so this is not possible. Also $\mathcal{V}_{P'}(M,\alpha )$ has torsion if and only if there is more than one incompressible surface 
bounding $\alpha $ in the homology class determined by $P'$.
 
If $V=V_K$ then a tunneling number $>1$ on edge means that the edge does not contribute a relation. On the other hand 
a tunneling number $1$ means that the dotted surface is equivalent to a smaller surface.
Thus for a component $P'$, the $R$-module generators of $P'$ are the white incompressible surfaces and all dotted incompressible surfaces without any smaller surface of \textit{tunneling-distance} $1$. 
 
It seems quite difficult to determine the $V$-module structure in general cases. In fact exact sequences which have to be analyzed following the methods of section \ref{sec=to Bar-Natan} will \textit{not} split, and the resulting $V^{\otimes \underline{\alpha }}$-modules will in general depend on genus differences. Note that, because of the relations, the
absolute genera are not appearing in the module structure, only differences in genus between incompressible surfaces and parts of incompressible 
surfaces. But note that in the general computation the empty surface is a possible incompressible surface component. 
Thus the genera of nullhomologous incompressible surface are present in the presentation. 

\section{Applications of the pullback principle}\label{sec=pullback}

Because of the irreducibility condition for the edges of the tunneling graph it is important to determine conditions when 
$W: T\longleftarrow C''\rightarrow S'$ is defined via pullback from $W': T\longleftarrow C\longrightarrow S$ and $W'': S\longleftarrow C''\longrightarrow S'$,
see section \ref{sec=presentations}.  

Suppose that $W',W''$ are given as above. We can assume, see \cite{KL}, that $W'$ is constructed from $T\times [-1,-1+\varepsilon]\subset M\times [-1,1]$ 
by adding first a union of embedded $1$-handles $h'^1$ to get a surface $C\subset M\times \{0\}$ and then adding embedded $2$-handles $h'^2$ and $3$-handles
$h'^3$ to finally get $S\subset M\times \{1\}$. Similarly for $W''\subset M\times [-1,1]$. Note that we can assume that each solid handle is embedded into a level. 

\begin{defn} We call $W'$ and $W''$  as above \textit{geometrically separated} if 
$M=M^1\cup M^2$ is a union along a surface, possibly with boundary, and $W'\cap (M^1\times [-1,1])$ respectively $W''\cap (M^2\times [-1,1])$ are 
products, i.\ e. for example $W'\cap (M^1\times [-1,1])=(W'\cap \{-1\})\times [-1,1]$. 
\end{defn}

\begin{thm}(Pullback theorem)  Let $W$ be the bordism from $T$ to $S'$ defined by glueing $W'$ on top of  $W''$ and reparametrizing the embedding 
into $M\times [-1,1]$. This manifold is diffeomorphic to the result $\widetilde{W}$ of embedding all $1$-handles $h'^1$ and $h''^1$ to $T\times [-1,-1+\varepsilon]\subset M\times [-1,1]$ to get a surface $C\subset M\times I$ followed by attaching all two-handles and $3$-handles $h'^2$, $h''^2$, $h'^3$ and $h''^3$ to $T\cup h'^1\cup h'^2 \cup \textrm{collar}$. Thus $\widetilde{W}$ is defined by pullback from $W'$ and $W''$.\end{thm}

Intuitively what the pullback theorem means is that relations defined from two occurences of sequences and inverse sequences of geometrically separated Bar-Natan relations are pullbacks of relations defined from the single occurences of sequences and inverse sequences keeping that part fixed on which the other occurence takes place.  Thus tunneling numbers arising from those bordisms do not appear as edges in the tunneling graph.

\begin{exmp} Let $\alpha =\emptyset$ and $M=\Sigma \times [0,1]$ for $\Sigma =\Sigma_g$ an orientable closed surface of genus $g$. 
Note that this $3$-manifold is Haken-irreducible for all $g\geq 0$. 
The set $\mathfrak{t}$ of incompressible surfaces is parametrized by the number $k\geq 0$ of parallel copies of $\Sigma \times \{1/2\}$ in $\Sigma \times [0,1]$. Call the corresponding isotopy class $S_k$. 
Note that 
if $k =0$ the empty surface is an incompressible surface and is a vertex of $\Gamma _0$, the component of $\Gamma $ defined by incompressible surfaces with $k$ even. The other component is denoted $\Gamma_1$. 

Now for given $k\geq 0$ consider the surface $C=C_k$ of $k+1$ components defined from $S_{k+2}$ by a connected sum of two of its components. There are morphisms $C\to S_k=S_-$ and $C\to S_{k+2}=S_+$ inducing the morphisms
$F_{+/-}: F(C)=V^{\otimes (k+1)}\rightarrow V^{\otimes (k+1\pm 1)}=F(S_{\pm })$, which are given by $F_-=\varepsilon \circ \mathfrak{k}^{2g}$ 
and $F_+=d_1^{i,i+1}$ where we use the notation from section \ref{sec=local}.
\end{exmp}

Recall the definition of the tensor algebra:
$$T(V):=\bigoplus_{k\geq 0}V^{\otimes k}.$$

\begin{thm} (see \cite{AF}) For each genus $g$ orientable closed surface $\Sigma $ we have
$$\mathcal{W}(\Sigma \times [0,1])\cong T(V)/R,$$
where $R$ is the $2$-sided ideal generated by relations $\Delta (v)=\varepsilon \circ \mathfrak{k}^{2g}(v)$ for all $v\in V$, with $\Delta : V\rightarrow V\otimes V$
placing $V\otimes V$ into an arbitrary tensor product $V^{\otimes k}$. $\square$
\end{thm}

\begin{rem} 
In \cite{BK} we develop oriented Bar-Natan theory. Then the tensor product $T(V)$ is replaced by an oriented version 
$$\widetilde{T}(V):=\bigoplus_{k\geq 0}\left(\bigoplus_{\lambda ,|\lambda |=k}V^{\otimes k}\right).$$
with $\lambda $ running through all sequences $(\lambda_1,\ldots ,\lambda_k)$ with $\lambda_i\in \{\pm 1\}$ depending on the orientation of the component, and $|\lambda|$ is the length of the tuple.
\end{rem}

\section{Bar-Natan modules and homology of spaces} \label{sec=homology}

In this section we assume that $V$ is manifold-induced, i.\ e.\ $V=H_*(X)$ for some oriented manifold $X$ of dimension $n$ with 
homology in only even dimensions. Moreover we assume that $\Delta $ is the homomorphism induced by the diagonal 
$X\rightarrow X\times X$ and $\varepsilon $ is induced by $X\rightarrow *$. Then for $n\geq 0$ the homomorphism $d_n$
is induced by the natural diagonal mapping $\delta_n: X\rightarrow X^n$.  

\begin{thm} Suppose that $(M,\alpha )$ is Haken-irreducible and has a tunneling graph $\Gamma _P(M,\alpha )$ for an ordered subset $P\subset \mathfrak{t}$ without loops and with all tunneling vectors the zero vector. Then there exists a natural space $Y$, determined by the tunneling graph, such that
$$\mathcal{W}_P(M,\alpha )\cong H_*(Y).$$ \end{thm}

\begin{proof} The space $Y$ is constructed inductively using the well-ordering on $P$, i.\ e.\
we construct a sequence of spaces $Y(S)$ with continuous mappings $Y(S)\rightarrow Y(S')$ for $S<S'$. 
Then $Y$ will be the direct limit of this direct system of spaces. 
Let $S_1$ be the smallest element of $P$
and define $Y(S_1):=X^{|S_1|}$. Now suppose that we have constructed the spaces $Y(S')$ for all $S'<S$. Then define $Y(S)_-$ to be the direct limit
of the spaces $Y(S')$ over all $S'<S$, which comes with natural maps $\iota_{S'S}: X^{|S'|}\rightarrow Y(S)_-$ for all $S'<S$. 
Let $\amalg $ denote disjoint union of spaces. 
Now consider the quotient of 
$X^{|S|}\amalg Y(S)_-$ by the equivalence relation defined by $X^{|S|}\ni x\sim y=Y(S)_-$ if and only if there exists an edge from 
$S$ to some $S'<S$ with tunneling invariant homomorphisms $F_{\pm}$ induced from \textit{obvious} mappings $i: X^{|C|}\rightarrow X^{|S|}$
and $k: X^{|C|}\rightarrow X^{|S'|}$ (where the tunneling invariant is induced from some morphism $S\longleftarrow C\longrightarrow S'$),
and $i(u)=x$ and $\iota_{SS'}\circ k(u)=y$ for some $u\in X^{|C|}$. 
This is really the important point. If tunneling vectors are nonzero then the tunneling homomorphisms involve powers of $\mathfrak{k}$, 
which cannot be described by products of embeddings $\delta_n$.   
Let $Y(S)$ denote the resulting space. Then $Y(S)$ is the union of 
two subspaces, which are the images of the natural maps $X(S)\rightarrow Y(S)$ and $Y(S)_-\rightarrow Y(S)$. 
The intersection of the two subspaces is homeomorphic to $X^{|C|}$. So the result follows by the Mayer-Vietoris sequence and the fact that homology commutes with direct limits. \end{proof}

A typical example is $M=S^1\times D^2$ with $\alpha $ the union of $2n$ longitudes with total homology class $0$. 
In \cite{R} Heather Russel constructs for \textit{alternating orientations} of the components of $\alpha $ (which is equivalent to the unoriented case) the space $Y$ explicitly as the homology of the $(n,n)$-Springer variety. The interesting point is of course \textit{why} this particular space from complex representation theory appears at this point. 

\vskip .1in

\noindent \textbf{Problem} Find explicit descriptions of the spaces $Y$ in other cases where $(M,\alpha )$ is Haken-irreducible and tunneling vectors are all $0$. 

\begin{small}

\end{small}


\begin{thebibliography}{999} 
\bibitem{A}
L.\ Abrams; Frobenius Algebra Structures in Topological Quantum Field Theory and Quantum Cohomology, PhD thesis 1997
\bibitem{AF} M.\ Asaeda, C.\ Frohman; A note on the Bar-Natan skein module, International Journal of Mathematics, Vol.\ 18, No.\ 10 (2007), 1225--1243, World Scientific Publishing
\bibitem{Be}
A.\ Beliakova, E.\ Wagner, On link homology from extended cobordisms, 
\href{http://arxiv.org/PS_cache/arxiv/pdf/0910/0910.5050v1.pdf}{arXiv:0910.5050} 
\bibitem{BN}
D.\ Bar-Natan; Khovanov homology for tangles and cobordisms, Geom.\ Topology 9 (2005), 1143--1199
\bibitem{BK} J.\ Boerner, U.\ Kaiser; On the structure of Bar-Natan modules, Preprint 2010, unpublished
\bibitem{Bo} F.\ Borceux; \textit{Handbook of Catgeorical Algebra I: Basic Categroy Theory}, Cambridge University Press, 1994
\bibitem{Ce}
J.\ Cerf; La stratification naturelle des espaces de functions differentiables reeles et la theoreme de la pseudoisotopie, Publ.\ Math.\ I.H.E.S. 39 (1970)
\bibitem{Fa} Fadali, Lyla; Bar-Natan skein modules in black and white, dissertation 2016, UC San Diego
\bibitem{CG}
A.\ Casson, C.\ Gordon; Reduced Heegaard splittings, Topology Appl.\ 27 (1987), 275--283
\bibitem{DM} P.\ Deligne, J.\ W.\ Morgan; Notes on Supersymmetry, in \textit{Quantum Fields and Strings: A course for Mathematicians}, Volume 1, AMS 1999
\bibitem{G}
M.\ Grandis; \textit{Directed algebraic topology - models of non-reversible worlds}, Cambridge University Press 2009, New Mathematical Monographs 3
\bibitem{H}
A.\ Hatcher, Spaces of incompressible surfaces, manuscript (author's homepage), updated version of \textit{Homeomorphisms of sufficiently large $P^2$-irreducible 3-manifolds} in Topology 1976
\bibitem{He}
J.\ Hempel; \textit{$3$-manifolds}, Annals of Math.\ Studies No.\ 86, Princeton Univ.\ Press 1976
\bibitem{HRW} M.\ Hogencamp, D.\ E.\ V.\ Rose, P.\ Wedrich; \textit{A Kirby color for Khovanov homology} \href{https://arxiv.org/abs/2210.05640}{arXiv:2210.05640} 
\bibitem{K1}
U.\ Kaiser; Frobenius algebras and skein modules of surfaces in $3$-manifolds, in \textit{Algebraic Topology - Old and New}, Banach Center Publications, Volume 85 (2009), 59--81 
\bibitem{K2} 
U.\ Kaiser; \textit{Link theory in manifolds}, Lecture Notes in Mathematics 1669, Springer 1996
\bibitem{Kh3}
M.\ Khovanov; Link homology and Frobenius extensions, Fund.\ Math.\ 190 (2006), 179--190
\bibitem{Ker}
T.\ Kerler; \textit{On the connectivity of cobordisms and half-projective TQFT's}, Comm.\ Math.\ Phys.\ 198 (1998), no.\ 3, 535--590
\bibitem{KL}
C.\ Kearton, Lickorish; Piecewise linear critical levels and collapsing, Transactions AMS vol.\ 170 (1972), 415--424
\bibitem{Ko}
J.\ Kock; \textit{Frobenius algebras and $2D$ topological quantum field theories}, Lecture Math.\ Soc.\ Student Texts No.\ 59, Cambridge University Press, 2003
\bibitem{L}
J.\ Lurie; On the classification of topological field theories, \textit{Current developments in Mathematics}, 2008, 129--280, Int.\ Press, Somerville, MA 2009
\bibitem{M}
Saunders MacLane; \textit{Categories for the working mathematician}, Springer Graduate texts in mathematics 5, 
Springer Verlag 1971
\bibitem{Ma}
V.\ Manturov, The Khovanov complex for virtual links, J.\ Math.\ Sci.\ (N.Y.) 144, 2007, no.\ 5, 4451--4467
\bibitem{Mc}
J.\ McCleary; \textit{User's Guide To Spectral Sequences}, Mathematics Lecture Series 12, Publish or Perish, 1985
\bibitem{QW} H. Queffelec, P.\ Wedrich; \textit{Khovanov homology and categorification of skein modules}, Quantum Topology 12 (2021), no.\ 1, 129--209
\bibitem{MSV} M.\ Mackaay, M.\ Stosic, P.\ Vaz; \textit{ $\textrm{sl}(N)$ link homology ($N\geq 4$) using foams and the Kapustin-Li formula}, Geometric Topology 13(2) (2009), 1075--1128
\bibitem{R}
H.\ Russell; The Bar-Natan skein module of the solid torus and the homology of $(n,n)$ Springer varieties, Geom.\ Dedicata 142 (2009), 71--89
\bibitem{TT}
V.\ Turaev, P.\ Turner; Unoriented topological quantum field theory and link homology, Alg.\ Geometric Top.\ 6 (2006), 1069--1093 (electronic)
\bibitem{W}
C.\ Weibel; \textit{An introduction to homological algebra}, Cambridge studies in advanced mathematics 38, Cambridge University Press 1994

\end{thebibliography}
\end{document}